\documentclass{amsart}

\usepackage{latexsym}
\usepackage{enumerate}
\usepackage{amssymb}
\usepackage{amsfonts}
\usepackage{amsmath}
\usepackage{mathrsfs}
\usepackage{amsthm}
\usepackage{geometry}
\usepackage[colorlinks,
linkcolor=red,
anchorcolor=blue,
citecolor=green
]{hyperref}
\usepackage{cleveref}
\usepackage[all]{xy}

\newcommand{\partialbar}{\bar{\partial}}
\newcommand{\zbar}{\bar{z}}
\newcommand{\Ric}{\mathrm{Ric}}
\newcommand{\Vol}{\mathrm{Vol}}
\newcommand{\V}{\mathrm{Vol}}
\newcommand{\inj}{\mathrm{inj}}
\newcommand{\sq}{\backslash}

\newcommand{\hl}{H^0_{L^2}}
\newcommand{\linebundle}{\mathcal{L}}
\newcommand{\dya}{\mathfrak{a}}

\newcommand{\dist}{\mathrm{dist}}
\newcommand{\re}{\mathrm{Re}}
\newcommand{\im}{\mathrm{Im}}
\newcommand{\model}{\mathrm{mod}}

\newcommand{\Aut}{\mathrm{Aut}}

\newtheorem{thm}{Theorem}[section]
\newtheorem{definition}{Definition}[section]

\newtheorem{prop}{Proposition}[section]
\newtheorem{coro}{Corollary}[section]
\newtheorem{lmm}{Lemma}[section]

\theoremstyle{remark} 
\newtheorem*{rmk}{Remark}

\title[Peak sections and Bergman kernels]{Peak sections and Bergman kernels on K\"ahler manifolds with complex hyperbolic cusps}

\author{Shengxuan Zhou}
\address{Beijing International Center for Mathematical Research\\
Peking University\\
Beijing\\ 100871\\ China}
\email{zhoushx19@pku.edu.cn}

\begin{document}

\begin{abstract}
By revisiting Tian's peak section method, we obtain a localization principle of the Bergman kernels on K\"ahler manifolds with complex hyperbolic cusps, which is a generalization of Auvray-Ma-Marinescu's localization result Bergman kernels on punctured Riemann surfaces \cite{aumamar1}. Then we give some further estimates when the metric on the complex hyperbolic cusp is a K\"ahler-Einstein metric or when the manifold is a quotient of the complex ball. By applying our method directly to Poincar\'e type cusps, we also get a partial localization result.
\end{abstract}

\maketitle

\tableofcontents

\section{Introduction}
\label{intro} 

Let $(M,\omega )$ be an $n$-dimensional K\"ahler manifold, and $\linebundle $ be a positive line bundle on $M$ equipped with a hermitian metric $h$. Then we can define the Bergman space for each $m\in\mathbb{N}$ as the following Hilbert space
$$ H^0_{L^2} \left( M ,\omega ,\linebundle^m ,h^m \right) = \left\lbrace S\in H^0 \left( M ,\linebundle^m \right) : \int_{M} \left\Vert S \right\Vert_{h^m}^2 d\V_\omega < \infty \right\rbrace ,$$
with the $L^2 $ inner product $\left\langle S_1 , S_2 \right\rangle_{L^2,M ;\omega ,h^m } = \int_{M} \left\langle S_1 , S_2 \right\rangle_{h^m} d\V_\omega $. Clearly, this space will also have a standard $L^2 $ norm $\left\Vert S \right\Vert_{L^2 ,M ;h^m ,\omega } = \left( \left\langle S,S \right\rangle_{L^2 ,M ; h^m , \omega } \right)^{\frac{1}{2}} $. Sometimes we denote the inner product and norm briefly by $\left\langle s_1 , s_2 \right\rangle_{L^2} $ and $\left\Vert s \right\Vert_{L^2} $, respectively.
Hence the Bergman kernel can be defined as
$$ B_{M, \omega , \linebundle ,h, m} (x,y) = \sum_{i\in I} S_{m,i} (x) \otimes \left( S_{m,i} (y) \right)^* ,\;\; \forall x,y\in M, $$
and the Bergman kernel function can be defined as
$$ \rho_{M,\linebundle ,h, \omega ,m} (x) = B_{M, \omega , \linebundle ,h, m} (x,x) = \sup_{S\in H^0_{L^2} \left( M,\linebundle^m \right)} \frac{ \left\Vert S(x) \right\Vert_{h^m}^2 }{\int_{M} \left\Vert S \right\Vert_{h^m}^2 d\V_{\omega} } =\sum_{i\in I} \left\Vert S_{m,i} (x) \right\Vert^2_{h^m} , $$
for any $ x\in M $, where $ \left\lbrace S_{m,i} \right\rbrace_{i\in I} $ is an $L^2 $ orthonormal basis associated with $h^m$ and $d\V_{\omega} =\frac{\omega^n }{n!}  $ in the Hilbert space $H^0_{L^2} \left(  M ,\omega ,\linebundle^m ,h^m \right) $. For abbreviation, we denote the Bergman space, Bergman kernel and Bergman kernel function briefly by $H^0_{L^2} \left(  M ,\linebundle^m  \right) $, $B_{\omega ,m} $ and $\rho_{\omega ,m}$, respectively. Note that $B_{\omega ,m} $ is a $ \left( \pi_1^* \linebundle^m \otimes \pi_2^* \linebundle^{-m} \right) $-valued function on $M\times M$, and $\rho_{\omega ,m}$ are functions on $M$, where $\pi_i $ is the projection of $M\times M $ onto the $i$-th component.

The Bergman kernel plays an important role in K\"ahler geometry. In the pioneering work \cite{tg1}, Tian used his peak section method to prove that Bergman metrics converge to the original polarized metric in the $C^2$-topology. By a similar method, Ruan \cite{wdr1} proved that this convergence is $C^\infty $. Later, Zelditch \cite{sz1}, also Catlin \cite{cat1} independently, used the Szegö kernel to obtain an alternative proof of the $C^\infty $-convergence of Bergman metrics, moreover, they gave an asymptotic expansion of Bergman kernel function, which is the K\"ahler potential of the Bergman metric, on polarized Kähler manifolds. Such an asymptotic expansion is called the Tian-Yau-Zelditch expansion. This expansion can be also obtained by Tian's peak section method, see Liu-Lu \cite{liuzql1}. By using the heat kernel, Dai-Liu-Ma \cite{dailiuma1} gave another proof of the Tian–Yau–Zelditch expansion and they further considered the asymptotic behavior of Bergman kernels on symplectic manifolds and Kähler orbifolds (see also Ma-Marinescu’s book \cite{mamari1}). There are many important applications of the Tian–Yau–Zelditch expansion, for example, \cite{don1} and \cite{kwz2}. 

More recently, Auvray-Ma-Marinescu \cite{aumamar1,aumamar2} and Sun-Sun \cite{sunsun1} independently considered the asymptotic behavior of the Bergman kernels on punctured Riemann surfaces. Sun-Sun also discusses the logarithmic K-semistability of punctured Riemann surfaces by considering the asymptotic behavior of the Bergman kernels in \cite{sunsun1}. Through revisiting Tian's peak section method, we generalize the results of Auvray-Ma-Marinescu to K\"ahler manifolds with complex hyperbolic cusps. The results and methods presented in this paper can be applied to more general cases and may be helpful in understanding log-K stability or the geometry of Shimura varieties.

In this paper, an $n$-dimensional $(n\geq 1)$ complex hyperbolic cusp is defined in terms of a collection of data $\left( V/\Gamma_V ,\omega_V ,\linebundle_V /\Gamma_V ,h_V \right)$ with the following property. Let $\left( D,\omega_D \right)$ be a flat Abelian variety of complex dimension $n-1$, and $\left( \linebundle_D ,h_D \right) $ be a negative Hermitian line bundle on $D$ whose curvature form is $-\omega_D $. If $n=1$, the space $D$ becomes a point, and $\linebundle_D \cong D\times \mathbb{C} $. Then $\partial \partialbar$-lemma implies that $h_D$ is unique up to scaling. It is clear that $v \mapsto h_D (v,v)$ gives a smooth function on the total space $\linebundle_D $. By a straightforward computation, we see that $\omega_V = -\sqrt{-1} \partial\overline{\partial} \log \left( -{\log h_D} \right) $ is a K\"ahler form on $V= \left\lbrace 0 < h_D < 1\right\rbrace $ with ${\rm Ric}(\omega_V) = -(n+1) \omega_V $, where $\pi_D :\linebundle_D \to D $ is the projection of line bundle. For $r>0 $, we will denote by $V_r$ the open subset $ \left\lbrace 0 < h_D < r\right\rbrace $ of the total space of $\linebundle_D $. Let $ \linebundle_V $ be a trivial line bundle on $V$ with a Hermitian metric $h_V $, such that there exists a section $e_{\linebundle_V} \in H^0 \left( V,\linebundle_V \right) $ satisfying that $\left\Vert e_{\linebundle_V} \right\Vert^2_{h_V} = \left| \log h_D \right| $. Let $\Gamma_{V} $ be a finite subgroup of $\Aut \left( V,\omega_V \right) $. Since $h_D $ is invariant under the action of $\Aut \left( V,\omega_V \right) $ on $V$, the action of $\Gamma_{V} $ on $V$ can be extended to the total space of $ \linebundle_D $ (see also Lemma \ref{lmminjectivityradiusoncusp}). Now we further extend the action of $ \Gamma_V $ on $V$ to the total space of the line bundle $\linebundle_V $ such that the biholomorphic maps corresponding to the action are automorphisms of the line bundle $\linebundle_V $, and $h_V $ is invariant under this action. Then there exists a homomorphism $\Theta_{\Gamma_V } : \Gamma_V \to U(1) \subset \mathbb{C}^* $ such that $g \left( e_{\linebundle_V } \right) = \Theta_{\Gamma_V } (g) e_{\linebundle_V }  $, $\forall g \in \Gamma_V $. If $\Gamma_V =0$, we can omit it in the notations. More details about complex hyperbolic cusps can be found in Mok's paper \cite{mok1}.

Now we define the asymptotic complex hyperbolic cusp, which is our main local model in this paper.

\begin{definition}[Asymptotic complex hyperbolic cusp]
\label{definitionasymptoticcomplexhyperbolioccusp}
Let $(M,\omega )$ be a K\"ahler manifold, let $(\linebundle ,h)$ be a Hermitian line bundle on $M$ and let $\alpha >0$ be a constant. We say that an open subset $U$ of $M$ is an asymptotic complex hyperbolic cusp of order $\alpha $ if it satisfies that:
\begin{enumerate}[1)]
\item $\Ric (h) = \omega $ on $U$ .
\item There exist a complex hyperbolic cusp $\left( V/\Gamma_V ,\omega_V ,\linebundle_V /\Gamma_V ,h_V \right)$, a constant $r\in \left( 0,e^{-1} \right) $, a biholomorphic map $\varphi_U : V_r /\Gamma_V \to U $ and a section $e_\linebundle \in H^0 \left( U,\linebundle\big|_{U} \right)$ such that $\left\Vert e_{\linebundle} \right\Vert_{h}^2 \circ \varphi_U = e^u \left| \log h_D \right| $, where $u=O\left( \left| \log h_D \right|^{-\alpha } \right)$ as $h_D \to 0^+ $ to all orders with respect to $\omega_V$. 
\end{enumerate}
Moreover, we say that $U$ is an asymptotic complex hyperbolic cusp of order $\infty $ if for any $ p\in \mathbb{N} $, $u=O\left( \left| \log h_D \right|^{-p } \right)  $ as $h_D \to 0^+ $ to all orders with respect to $\omega_V$.

By abuse of notation, it is sometimes convenient to write $V_r /\Gamma_V $ instead of $U$. 
\end{definition}

Since $M$ is a manifold, we see that the action of the group $\Gamma_V $ on $V$ is free in Definition \ref{definitionasymptoticcomplexhyperbolioccusp}, and $\linebundle_V /\Gamma_V $ is a line bundle. In general, the quotient space $ V /\Gamma_V $ is only an orbifold, and $\linebundle_V /\Gamma_V $ is only a line orbibundle. Before we state our first main result, we need the following assumption.

\begin{definition}[Admissible pair]
\label{definitionnumberssigmagamma}
Let $m\in\mathbb{N } $ and $\xi ,\kappa ,r>0 $. Then a pair $ \left( \gamma  ,\sigma \right) \in \mathbb{R}_{+}^2 $ is said to be $(m, \xi ,\kappa ,r) $-admissible if $r \geq \sigma $, $ \gamma \geq e^{-\kappa \frac{\sqrt{m}}{\log m }} $ and $\log \sigma \geq {e^{-\xi \frac{\log m}{\sqrt{m}}}} \log \gamma  $. 
\end{definition}

For using H\"ormander's $L^2$-estimate, we need the following condition for non-complete K\"ahler manifold. This is a modification of the standard pseudoconvex, i.e., the existence of a psh exhaustion function. 

\begin{definition}[Pseudoconvex]
\label{definitionpseudoconvex}
In this paper, we say that a K\"ahler manifold $(M,\omega )$ is pseudoconvex if there exists a proper plurisubharmonic function $\psi \in C^\infty (M;\mathbb{R} ) $ such that $\psi^{-1} ((-\infty ,k])$ is the union of a compact subset and finitely many complex hyperbolic cusps for any $k\in\mathbb{R}$, and $M = \cup_{k=1}^{\infty} \psi^{-1} ((-\infty , k]) $. 
\end{definition}

\begin{rmk}
  Note that complex hyperbolic cusp $V_r$ is always complete as $h_D \to 0 $, so H\"ormander's $L^2$-estimate holds for $\partialbar$-closed $(0,1)$-form with compact support on pseudoconvex manifold. Let $\psi = -\log (r-\log h_D) $. We see that $(V_r ,\omega )$ are pseudoconvex in this sense.
\end{rmk}
    
Now we can state our first main result, which is the following localization principle. The following result is the higher dimensional analogue of \cite[Theorem 1.1-2]{aumamar1} and \cite[Theorem 1.2]{aumamar2}. 
\begin{thm}
\label{thmcusplocalizationprinciple}
Let $(M,\omega )$ be a K\"ahler manifold, let $(\linebundle ,h)$ be a Hermitian line bundle on $M$ and $\alpha >0$ be a constant such that $U \cong V_r /\Gamma_V $ is an asymptotic complex hyperbolic cusp of order $\alpha $, and $\Ric (h) \geq \epsilon \omega $ and $\Ric (\omega ) \geq -\Lambda \omega $ for some constants $\epsilon ,\Lambda >0$. Suppose that if $(M,\omega )$ is not complete, then $M$ is pseudoconvex and $ \bar{U} $ is complete.

Let $B_{V_t /\Gamma_V ,\omega ,m } (x,y) $ denotes the Bergman kernel on $\left( V_t /\Gamma_V ,\omega ,\linebundle ,h \right) $. Given constants $\xi ,\kappa ,\beta >0$. Then there exists a large constant $ C>1$ satisfying the following property.

Let $m\in\mathbb{N}$, and $\left( \gamma_m ,\sigma_m \right) $ be an $(m, \xi ,\kappa ,r) $-admissible. Then for any $ m\in\mathbb{N} $, the Bergman kernel satisfying that
\begin{eqnarray}
 \left\Vert B_{M,\omega ,m} (x,y) \right\Vert_{C^k ; h^m ,\omega } & \leq & C m^{-l} \left| \log h_D (x) \right|^{-\beta } \left| \log h_D (y) \right|^{-\beta } , \label{equation1thmcusplocalizationprinciple}
\end{eqnarray}
for any $ (x,y) \in \left( \pi_V \left( V_{\gamma_m } \right) \times \left(U\sq \pi_V \left( V_{\sigma_m } \right) \right) \right) \cup \left( \left( U \sq \pi_V \left( V_{\gamma_m } \right) \right) \times \pi_V \left( V_{\gamma_m } \right) \right) $, and
\begin{eqnarray}
  \left\Vert B_{ V_{\sigma_m} /\Gamma_V , \omega ,m } (x,y) - B_{M,\omega ,m} (x,y) \right\Vert_{C^k ; h^m ,\omega } & \leq & C m^{-l} \left| \log h_D (x) \right|^{ -\beta } \left| \log h_D (y) \right|^{ - \beta } , \label{equation2thmcusplocalizationprinciple}
\end{eqnarray}
for any $ (x,y) \in \pi_V \left( V_{\gamma_m } \right) \times \pi_V \left( V_{\gamma_m } \right) $, where $\pi_V $ is the quotient map $ V \to V /\Gamma_V \cong U $, $\left\Vert \cdot \right\Vert_{C^k ; h^m ,\omega } $ is the $C^k$-norm associate with the Hermitian metric $ \pi_1^* h^m  \otimes \pi_2^* h^{-m} $ and the K\"ahler metric $ \omega  $ on $U\times U$.

Moreover, when $m \geq C $, the Bergman kernel functions satisfying that
\begin{eqnarray}
\sup_{V_{\gamma_m } /\Gamma_V } \left| \frac{\rho_{M ,\omega ,m} (x)}{\rho_{V_{\sigma_m } /\Gamma_V , \omega ,m} \left( x \right) } -1 \right| & \leq & C m^{-l} . \label{equation3thmcusplocalizationprinciple}
\end{eqnarray}
\end{thm}

\begin{rmk}
Since there exists a constant $\delta >0 $ independent of $m$ such that $\inj \left( U\sq V_{\gamma_m } \right) \geq \delta \frac{\log m }{\sqrt{m}} $ (see Lemma \ref{lmminjectivityradiushyperbpliccusp}), we see that the asymptotic behavior of $B_{M,\omega ,m} (x,y) $ on $ \left(U\sq V_{\gamma_m } \right) \times \left(U\sq V_{\gamma_m } \right) $ is same to the standard case. See \cite[p. 2242]{hzh1} or \cite[Chapter 4]{mamari1} for details in the standard case.
\end{rmk}

Our proof of Theorem \ref{thmcusplocalizationprinciple} is based on Tian's Peak section method. The key step, Proposition \ref{propcoraseestimate}, is to show that the decay speed of the peak sections is sufficiently fast at a distance. We can use the Agmon type estimate (see Corollary \ref{coroagmontypeestimate}) to obtain this estimate by observing the Bergman kernels can be represented simply by the peak sections.

Now we will give some refinements of Theorem \ref{thmcusplocalizationprinciple} in two special cases. The first case is about the complex hyperbolic cusps with K\"ahler-Einstein metrics. We now recall some properties of it.

Fix a polarized $(n-1)$-dimensional flat Abelian variety $\left( D,\omega_D \right) $ with polarization $\left( \linebundle_D^{-1} ,h_D^{-1} \right) $. Now we can construct an $n$-dimensional complex hyperbolic cusp $\left( V,\omega_V ,\linebundle_V ,h_V \right)$ by the statement before Theorem \ref{thmcusplocalizationprinciple}, where $V= \left\lbrace 0<h_D <1 \right\rbrace $ is an open subset of the total space of the line bundle $ \linebundle_D $, $ \omega_V = -\sqrt{-1} \partial\partialbar \log \left( -\log h_D \right) $, and $ \left( \linebundle_V ,h_V \right) \cong \left( \mathbb{C} ,-\log h_D \right) $ is a Hermitian line bundle on $V$. Let $r\in (0,1)$ be a constant. 

For any given smooth real-valued function $f\in C^\infty \left( \partial \bar{V}_r \right) $, Datar-Fu-Song \cite[Theorem 1.1, Theorem 1.3]{dfs1} proved that there exists a unique complete K\"ahler-Einstein metric $\omega_{r,f} = \omega_V + \sqrt{-1} \partial\partialbar u_{r,f} $ on $V_r $ by considering the Dirichlet problem of the Monge-Amp\'ere equation, where $V_r = \left\lbrace 0<h_D <r \right\rbrace \subset V $, and $u_{r,f} \in C \left( \bar{V}_r \right) $ is a smooth real-valued function on $V_r $ such that $ u_{r,f} \big|_{ \partial \bar{V}_r } =f $, and $\left| u_{r,f} \right| = o(1) $ as $h_D \to 0^+ $. When $n\geq 2$, Fu-Hein-Jiang \cite[Main Theorem]{fhj1} further studied the asymptotic behavior of $\omega_{r,f} $ and $u_{r,f} $ by establishing sharper estimates. For any $k\in\mathbb{Z}_{\geq 0} $, they proved that there exists a constant $c_{r,f} \in\mathbb{R} $ such that
\begin{eqnarray}
\label{fhjasymptoticbehavior}
\left\Vert \nabla^k \left( u_{r,f} + \log \left( 1+c_{r,f} \left| \log h_D \right|^{-1} \right) \right) \right\Vert_{\omega_V } = O \left( \left| \log h_D \right|^{\frac{n+k}{2} - \frac{1}{4}} e^{-2\sqrt{ \lambda_{1,D} \left| \log h_D \right| }} \right) 
\end{eqnarray}
as $ h_D \to 0^+ $, where $\lambda_{1,D} >0 $ denotes the first eigenvalue of $ \left( D,\omega \right) $. See \cite[Main Theorem]{fhj1} for more details. By considering the change of the Hermitian line bundle under the scalar multiplication $ v \mapsto e^{ -\frac{  c_{r,f} }{2} } v $ on the total space of $ \linebundle_D $, (\ref{fhjasymptoticbehavior}) implies that the data $ \left( V_r ,\omega_{r,f} ,\linebundle_V ,e^{-u_{r,f}} h_V \right) $ is an asymptotic hyperbolic cusp of order $\infty $. See also Lemma \ref{lmmdatarfusongmanifold1}. Since the K\"ahler metrics on asymptotic hyperbolic cusps can always be expressed as $\omega_V +\sqrt{-1} \partial\partial u $ for some function $u$, Fu-Hein-Jiang's estimate holds for all asymptotic hyperbolic cusps with K\"ahler-Einstein metrics.

Now we can state our result for the Bergman kernels on asymptotic complex hyperbolic cusps with K\"ahler-Einstein metric. For abbreviation, we denote $e^{-\frac{\sqrt{m}}{\log (m)}}$ briefly by $r_m $ for any $m\geq 2$. When $\Gamma_V =0$ and $ \linebundle^{-1}_D $ is globally generated, the $C^0$-estimates for the quotients between Bergman kernel functions, (\ref{inequalitydatarfusongmanifold3}) and (\ref{inequalityballquotient2}), can be improved to the $C^k$ estimates. We will prove them in Proposition \ref{propdatarfusongmanifold2} and Proposition \ref{propballquotient2}, respectively. Note that the $C^k$ estimates (\ref{inequalitydatarfusongmanifold3}) and (\ref{inequalityballquotient2}) is given in \cite[Theorem 1.3]{aumamar2} in the $1$-dimensional case.

\begin{thm}
\label{thmdatarfusongmanifold}
Let $(M,\omega )$ be a K\"ahler manifold and let $(\linebundle ,h)$ be a Hermitian line bundle on $M$ such that $\Ric (h) \geq \epsilon \omega $ and $\Ric (\omega ) \geq -\Lambda \omega $ for some constants $\epsilon ,\Lambda >0$. Suppose that $M$ is pseudoconvex if $(M,\omega )$ is not complete. 

Assume that $U \cong V_r/\Gamma_V $ is an asymptotic complex hyperbolic cusp of order $\infty $ on $(M,\omega , \linebundle ,h )$ such that $ \bar{U} $ is complete. Then for any given constants $k, l\in\mathbb{N} $ and $\beta>0 $, there are constants $m_0 ,C>0 $ such that for any integer $m\geq m_0 $, the Bergman kernels satisfying that 
\begin{eqnarray}
 \left\Vert B_{M ,\omega ,m} (x,y) \right\Vert_{C^k ; h^m_{V} , \omega_{V}  } & \leq & C m^{-l} \left| \log h_D (x) \right|^{-\beta} \left| \log h_D (y) \right|^{-\beta} , \label{inequalitydatarfusongmanifold1} 
\end{eqnarray}
for any $ (x,y) \in \left( \pi_V \left( V_{r^e_m } \right) \times \left(U\sq \pi_V \left( V_{r_m } \right) \right) \right) \cup \left( \left( U \sq \pi_V \left( V_{r_m } \right) \right) \times \pi_V \left( V_{r^e_m } \right) \right) $, and
\begin{eqnarray}
 \left\Vert B_{M,\omega ,m } (x,y) - B_{V/\Gamma_V , \omega_V ,m} (x,y) \right\Vert_{C^k ; h^m_{V} , \omega_{V} } & \leq & C m^{-l} \left| \log h_D (x) \right|^{-\beta} \left| \log h_D (y) \right|^{-\beta} , \label{inequalitydatarfusongmanifold2} 
\end{eqnarray}
for any $ (x,y) \in \pi_V \left( V_{r_m } \right) \times \pi_V \left( V_{r_m } \right) $. Moreover, the Bergman kernel functions satisfying that
\begin{eqnarray}
\sup_{V_{r_m } } | \frac{\rho_{M ,\omega ,m} (x)}{\rho_{V/\Gamma_V , \omega_{V} ,m} \left( x \right) } -1 | & \leq & C m^{-l} , \label{inequalitydatarfusongmanifold3}
\end{eqnarray}
where $\left\Vert \cdot \right\Vert_{C^k ; h^m ,\omega } $ is the $C^k$-norm associate with the Hermitian metric $ \pi_1^* h^m \otimes \pi_2^* h^{-m} $ and the K\"ahler metric $ \pi_1^* \omega + \pi_2^* \omega  $ on $V\times V$. 
\end{thm}

The next case is about the quotients of the complex hyperbolic ball. Let $\left(  {\mathbb{B}}^n ,\omega_{\mathbb{B}} \right)$ be the $n$-dimensional complex hyperbolic ball with constant bisectional curvature $-1$, and let $\Gamma $ be a lattice in $ \Aut \left( {\mathbb{B}}^n ,\omega_{\mathbb{B}} \right) \cong PU (n,1) $. Here $\Gamma $ is a lattice means that $\Gamma $ is a discrete subgroup of $ PU (n,1)  $ such that $ \left( \mathbb{B}^n / \Gamma ,\omega_{\mathbb{B}} \right) $ has finite volume. When all parabolic elements of $\Gamma $ are unipotent, there exist a finite collection of disjoint open subsets $ U_i \subset \mathbb{B}^n / \Gamma $, $i=1,\cdots , N$, and complex hyperbolic cusps $\left( V_i ,\omega_{V_i }  ,\linebundle_{V_i } ,h_{V_i } \right) $ such that $ \left( U_i ,\omega_{\mathbb{B}^n} \right) \cong \left( V_i ,\omega_{V_i } \right) $, and $ \left( \mathbb{B}^n / \Gamma \right) \sq  \left( \cup_{i=1}^N U_i \right) $ is compact. See \cite[Theorem 1]{mok1} for more details. Note that any neat lattice in $PU(n,1) $ satisfies the above condition, and hence $\mathbb{B}^n / \Gamma $ can be decomposed as above when $\Gamma \subset PU(n,1) $ is a neat lattice.

Recall that any lattice admits a finite index neat normal sublattice \cite[Theorem 6.11]{rag1}. Let $\Gamma_{neat} \unlhd \Gamma $ be a neat normal sublattice such that $\Gamma' = \Gamma / \Gamma_{neat} $ is a finite group. It follows that $\mathbb{B}^n / \Gamma \cong \left( \mathbb{B}^n / \Gamma_{neat} \right) /\Gamma' $. The decomposition result of $ \mathbb{B}^n / \Gamma_{neat} $ now implies that there exist a finite collection of disjoint open subsets $ U_i \subset \mathbb{B}^n / \Gamma $, $i=1,\cdots , N$, and complex hyperbolic cusps $\left( V_i /\Gamma_{V_i } ,\omega_{V_i }  ,\linebundle_{V_i } /\Gamma_{V_i } ,h_{V_i } \right) $ such that $ \left( U_i ,\omega_{\mathbb{B}^n} \right) \cong \left( V_i /\Gamma_{V_i } ,\omega_{V_i } \right) $, and $ \left( \mathbb{B}^n / \Gamma \right) \sq  \left( \cup_{i=1}^N U_i \right) $ is compact. Now we suppose that $\mathbb{B}^n / \Gamma $ is a manifold, that is to say the lattice $\Gamma $ is torsion free. At this time, $\mathbb{B}^n / \Gamma $ satisfies the following assumptions at infinity.

\begin{thm}
\label{thmballquotient}
Let $(M,\omega )$ be a K\"ahler manifold and let $(\linebundle ,h)$ be a Hermitian line bundle on $M$ such that $\Ric (h) \geq \epsilon \omega $ and $\Ric (\omega ) \geq -\Lambda \omega $ for some constants $\epsilon ,\Lambda >0$. Suppose that $M$ is pseudoconvex if $(M,\omega )$ is not complete. 

We follow the notations used in Definition \ref{definitionasymptoticcomplexhyperbolioccusp}. Assume that $U \cong V_r $ is an asymptotic complex hyperbolic cusp on $(M,\omega , \linebundle ,h )$ that satisfies $u=0$ and $ \bar{U} $ is complete. Then for any given constants $k, l\in\mathbb{N} $ and $\beta>0 $, there are constants $m_0 ,C>0$ such that for any integer $m\geq m_0 $, the Bergman kernels satisfying that 
\begin{eqnarray}
\sup_{V_{r^2 } \times V_{r^2 } } \left| \log h_D (x) \right|^\beta \left| \log h_D (y) \right|^\beta \left\Vert B_{M, \omega ,m } (x,y) - B_{ V/\Gamma_V , \omega_V ,m} (x,y) \right\Vert_{C^k ; h^m_{V} , \omega_{V} } & \leq & C m^{-l} , \label{inequalityballquotient1} 
\end{eqnarray}
and the Bergman kernel functions satisfying that
\begin{eqnarray}
\sup_{V_{r^2 } } | \frac{\rho_{M ,\omega ,m} (x)}{\rho_{V/\Gamma_V , \omega_{V} ,m} \left( x \right) } -1 | & \leq & C m^{-l} , \label{inequalityballquotient2}
\end{eqnarray}
where $\left\Vert \cdot \right\Vert_{C^k ; h^m ,\omega } $ is the $C^k$-norm associate with the Hermitian metric $ \pi_1^* h^m \otimes \pi_2^* h^{-m} $ and the K\"ahler metric $ \pi_1^* \omega + \pi_2^* \omega  $ on $\left( V/\Gamma_V \right) \times \left( V/\Gamma_V \right) $. 
\end{thm}

\begin{rmk}
In the case $n=1$, we give a new proof of Auvray-Ma-Marinescu's estimate for Bergman kernels on punctured Riemann surfaces \cite[Theorem 1.1]{aumamar1} that relies on Tian's peak section method.
\end{rmk}

Let $X= \mathbb{B}^n /\Gamma $, and let $\mathcal{K}_X $ be the canonical bundle on $X$. Then the K\"ahler metric $\omega_{\mathbb{B}^n}$ gives a Hermitian metric $h_X $ on $\mathcal{K}_X $. Since $\Ric \left( \omega_{\mathbb{B}^n} \right) = -(n+1) \omega_{\mathbb{B}^n} $, we have $ \Ric \left( h_X \right) = (n+1) \omega_{\mathbb{B}^n} $, and hence $ \left( X , (n+1)\omega_{\mathbb{B}^n} , \mathcal{K}_X , h_X \right) $ is a polarized complete K\"ahler orbifold. Now we can describe the asymptotic behavior of the supremum of the Bergman kernel functions of $ \left( X , \omega_{\mathbb{B}^n} , \mathcal{K}_X , h_X \right) $. This is the higher dimensional analogue of \cite[Corollary 1.4]{aumamar1}. Note that Theorem \ref{thmcusplocalizationprinciple} also holds for K\"ahler orbifolds with asymptotic complex hyperbolic cusps, so we don't need to assume that $\Gamma $ is torsion free here.

\begin{thm}
\label{thmsupballquotient}
Under the conditions stated above, there exist an integer $\mathcal{N} =\mathcal{N} (n,\Gamma ) \in\mathbb{N} $, a constant $C=C(n,\Gamma ) >0 $, and a sequence of positive constants $\left\lbrace c_i \right\rbrace_{i=1}^{\infty} $, such that $c_{i} =c_{i+\mathcal{N}} $, $ \forall i\in\mathbb{N} $, and satisfy the following properties.

If $ \Gamma $ is cocompact, then
\begin{eqnarray}
\sup_{X} \rho_{X , \omega_{\mathbb{B}^n} , \mathcal{K}_X , h_X ,m} & = & c_m m^{n} + O\left( m^{n-1} \right) ,\textrm{ as $m\to\infty $.} \label{inequalitythmsupballquotient1}
\end{eqnarray}

If $ \Gamma $ is not cocompact, then
\begin{eqnarray}
\sup_{X} \rho_{X , \omega_{\mathbb{B}^n} , \mathcal{K}_X , h_X ,m} & = & c_m m^{n+\frac{1}{2}} + O\left( m^{n } \right) ,\textrm{ as $m\to\infty $.} \label{inequalitythmsupballquotient2}
\end{eqnarray}

Moreover, if $\Gamma $ is neat, then $\mathcal{N}=1$.
\end{thm}

Estimates for the upper bound of the Bergman kernel functions of the Shimura varieties are relevant in arithmetic geometry \cite{aaeu1, aumamar1, jk1, pmeu1}. It is possible to generalize our argument for Theorem \ref{thmsupballquotient} to the more general Shimura varieties, and possibly useful for arithmetic geometry. 

Now we apply the arguments in the proof of Theorem \ref{thmcusplocalizationprinciple} to Poincar\'e type cusps. The basic model of the Poincar\'e type cusp consists of a K\"ahler manifold $ \left( \mathbb{D}_{r}^* \right)^k \times B_r^{ 2n- 2k } (0)$ and a Kahler metric 
$ -\sqrt{-1} \sum_{i=1}^{k} \partial\partialbar \log \left( \log \left| z_i \right|^2 \right) + \pi_{k,n}^* \omega_{k,n} $
on $ \left( \mathbb{D}_{r}^* \right)^k \times B_r^{ 2n- 2k } (0)$, where $\pi_{k,n} $ is the projection of $ \left( \mathbb{D}_{r}^* \right)^k \times B_r^{ 2n- 2k } (0) \subset \mathbb{C}^{ n } $ onto $B_r^{ 2n- 2k } (0)$, $\omega_{k,n} $ is a K\"ahler metric on $B_r^{ 2n- 2k } (0)$, and $\mathbb{D}_{r} = \left\lbrace z\in\mathbb{C} \big| |z|<r \right\rbrace $, $\mathbb{D}_{r}^* = \left\lbrace z\in\mathbb{C} \big| 0<|z|<r \right\rbrace $ for any $r>0$. By abuse of notation, we will write the K\"ahler metric $ -\sqrt{-1} \sum_{i=1}^{k} \partial\partialbar \log \left( \log \left| z_i \right|^2 \right) +  \pi_{k,n}^* \omega_{k,n} $ above simply $ \omega_{\model ,k,n } $ when $ \omega_{k,n} $ is an Euclidean metric. 

Fix a simple normal crossing divisor $D$ in an $n$-dimensional compact K\"ahler manifold $ \bar{M} $. By definition, we see that for each $x\in D$, there exists a neighborhood $U_x $ of $x$ and a biholomorphic map 
$$ \varphi_x : U_x \sq D \to \prod_{i=1}^{k_x} \mathbb{D}_{r_{x,i}}^* \times B_{r_{x,k,n}}^{2n-2k_x} (0) \subset \left( \mathbb{D}_{1}^* \right)^{k_x} \times \mathbb{C}^{ n- k_x }  ,$$
where $r_{x,j} \in (0,1] $ and $r_{x,k,n} >0 $ are constants. Let $\omega $ be a K\"ahler metric on $\bar{M}\sq D$. Then we say that $ \omega $ is a metric with Poincar\'e type cusp along $D$ if for any $x\in U$, the data $\left( U_x ,\varphi_x \right) $ constructed above satisfying that 
$$ \left( \varphi^{-1}_x \right)^{*} \omega = -\sqrt{-1} \sum_{i=1}^{k} \partial\partialbar \log \left( \log \left| z_i \right|^2 \right) +   \pi_{k,n}^* \omega_{k,n} + \sqrt{-1} \partial\partialbar u_x ,$$
where $u_x $ is a smooth function on $ \prod_{i=1}^{k_x} \mathbb{D}_{r_{x,i}}^* \times B^{2n-2k_x }_{r_{x,k,n}} (0) $, and there exists a constant $\alpha >0$ such that $u_x \in \cap_{i=1 }^{k_x} O\left( \left| \log \left| z_i \right| \right|^{-\alpha } \right) $ as $\inf_{1\leq i\leq k_x} \left| z_i \right| \to 0^+ $ to all orders with respect to $\left( \varphi^{-1}_x \right)^{*} \omega$.

Let $M$ denote the manifold $ \bar{M} \sq D $. Suppose that $\left( M ,\omega \right) $ is a complete K\"ahler manifold with Poincar\'e type cusp along $D$ and $(\linebundle ,h)$ is a Hermitian line bundle on $M$ such that $\Ric (h) \geq \epsilon \omega $ and $\Ric (\omega ) \geq -\Lambda \omega $ for some constants $\epsilon ,\Lambda >0$. We further assume that for any $x\in D$, there exists a local frame $e_x \in H^0 \left( U_x \sq D ,\linebundle \right) $ such that $ \left\Vert e_x \right\Vert^2 =e^{-u_x - \tau_{k,n} \circ \pi_{k,n} } \left( \prod_{i=1}^{k_x} \left| \log \left| z_i \right|^2 \right| \right) $, where $\pi_{k,n} $ is the projection of $ \left( \mathbb{D}_{1}^* \right)^k \times \mathbb{C}^{ n- k } $ onto $\mathbb{C}^{ n- k } $, and $\tau_{k,n} $ is a K\"ahler potential of $\omega_{k,n} $. It follows that $ \Ric (h) =\omega $ on a neighborhood of $ D $.

By imitating the argument in the proof of Theorem \ref{thmcusplocalizationprinciple}, one can give an analogy of Theorem \ref{thmcusplocalizationprinciple} on K\"ahler manifolds with Poincar\'e type cusps.

\begin{prop}
\label{proppoincarelocalizationprinciple}
Let $\bar{M} $ be a compact K\"ahler manifold, $D$ be a simple normal crossing divisor, and $(M,\omega )= \left( \bar{M} \sq D ,\omega \right) $ be a complete K\"ahler manifold with Poincar\'e type cusp along $D$. Assume that $(\linebundle ,h )$ be a positive line bundle on $M$ that satisfies the above conditions. For any open subset $U\subset M$, we will denote by $B_{U, \omega ,m } (x,y) $ the Bergman kernel on $\left( U ,\omega ,\linebundle ,h \right) $. Fix a Riemannian metric $\bar{g} $ on $\bar{M}$. Write $\bar{d}(x) = \dist_{\bar{g}} \left(  x,D \right) $, $\forall x\in M $. Given constants $\xi ,\kappa ,\beta >0$. Then there exists a constant $ C>0$ satisfying the following property.

Let $m\geq C $ be an integer, and $\left( U_m ,U'_m \right) $ be a pair of open subsets of $M$ such that the distance $\dist_{\omega} \left( U'_m ,M\sq U_m \right) \geq \xi \frac{\log m}{\sqrt{m}} $, and $\left\lbrace x\in M \;\big| \; \bar{d} (x) \leq e^{-\kappa \frac{\sqrt{m}}{\log m }} \right\rbrace \subset U'_m $. Then the Bergman kernel satisfying that
\begin{eqnarray}
 \sup_{ U'_{m } \times \left(M\sq U_m \right) \cup \left(M\sq U_m \right) \times U'_{m } }   \left| \log \bar{d} (x) \right|^\beta \left| \log \bar{d} (y) \right|^\beta  \left\Vert B_{\omega ,m} (x,y) \right\Vert_{C^k  } & \leq & C m^{-l} , \label{equation1thmpoincarelocalizationprinciple}\\
\sup_{ U'_{m } \times U'_{m } } \left| \log \bar{d} (x) \right|^\beta \left| \log \bar{d} (y) \right|^\beta \left\Vert B_{U_m ,\omega ,m } (x,y) - B_{\omega ,m} (x,y) \right\Vert_{C^k } & \leq & C m^{-l} . \label{equation2thmpoincarelocalizationprinciple}\end{eqnarray}
\end{prop}

This paper is organized as follows. In Section \ref{Preliminaries}, we collect some preliminary results that will be used many times. Then we will calculate the Bergman kernels on the complex hyperbolic cusp in Section \ref{model}. The proofs of Theorem \ref{thmcusplocalizationprinciple} and Proposition \ref{proppoincarelocalizationprinciple} are presented in Section \ref{tpsonprs}. Finally, the proofs of Theorem \ref{thmdatarfusongmanifold}, Theorem \ref{thmballquotient}, and Theorem \ref{thmsupballquotient} are contained in Section \ref{bergmanperturbedmanifoldssection}.

\vspace{0.2cm}

\textbf{Acknowledgement.} I would like to express my deepest gratitude to Professor Gang Tian, my supervisor, for his constant encouragement and guidance.

\section{Preliminaries}
\label{Preliminaries}

Now we state the H\"omander's $L^2$ estimate as follows without proof.  
\begin{prop}
\label{l2m}
Let $(M,\omega )$ be an $n$-dimensional complete K\"ahler manifold. Let $(\linebundle ,h)$ be a hermitian holomorphic line bundle, and let $\psi $ be a function on $M$, which can be approximated by a decreasing sequence of smooth function $\left\lbrace \psi_i \right\rbrace_{i=1}^{\infty} $. Suppose that $$\sqrt{-1}\partial\bar{\partial} \psi_i + \Ric(\omega )+\Ric(h)  \geq \gamma\omega $$
for some positive continuous function $\gamma $ on $M$, $\forall i\in\mathbb{N}$. Then for any $\linebundle $-valued $(0,1)$-form $\zeta\in L^{2}$ on $M$ with $\bar{\partial} \zeta =0$ and $\int_{M} ||\zeta ||^{2} e^{-\psi} \omega^n $ finite, there exists a $\linebundle $-valued function $u\in L^{2}$ such that $\partialbar u=\zeta$ and $$\int_{M} ||u||^{2} e^{-\psi } \omega^n \leq \int_{M} \gamma^{-1} ||\zeta ||^{2} e^{-\psi} \omega^n ,$$
where $||\cdot ||$ denotes the norms associated with $h$ and $\omega $.
\end{prop}

The proof of H\"omander's $L^2$ estimate can be found in \cite[Chapter 5]{dm1} or \cite[Chapter IV]{lh1}. Note that the $L^2$ estimate also holds for complete K\"ahler orbifolds \cite[Proposition 2]{blukas1} and weakly pseudoconvex K\"ahler manifolds \cite[Corollary 5.3]{dm1}.

We introduce the Cheeger-Gromov $C^{m,\alpha }$-norm for K\"ahler manifolds now. 
\begin{definition}[Holomorphic Norms]
\label{holonorm}
Let $(M,\omega ,x)$ be a pointed K\"ahler manifold. We say that the holomorphic $C^{m,\alpha}$-$norm\; on\; the\; scale\; of\; r$ at $x$: 
$$\left\Vert (M,\omega ,x) \right\Vert^{holo}_{C^{m,\alpha} ,r} \leq Q,$$
provided there exists a biholomorphic chart $\phi : B_r (0)  \to   (U,x)\subset M$ such that $\phi (0) =x $, $ |D\phi |\leq e^Q \textrm{ on } B_r (0) \textrm{ and } \left| D\phi^{-1} \right| \leq e^Q \textrm{ on } U $, and for all multi-indices $I\textrm{ with }0\leq |I|\leq m $, $ r^{|I|+\alpha} \left\Vert D^{I}\omega_{i\bar{j}} \right\Vert_{\alpha} \leq Q$. Globally we define
$$ \left\Vert (M,\omega) \right\Vert^{holo}_{C^{m,\alpha} ,r}=\sup_{x\in M} \left\Vert (M,\omega,x) \right\Vert^{holo}_{C^{m,\alpha} ,r}.$$

\end{definition}

By using H\"ormander's $L^2$ estimate appropriately, Tian initiate his peak section method in \cite{tg1}. One of the key steps of Tian's peak section method is to construct the following global section.

\begin{lmm}[{\cite[Lemma 1.2, Lemma 2.3]{tg1}}]
\label{peaksec}
For an $n$-tuple of integers $P=\left( p_1 ,p_2 ,...,p_n \right)\in \mathbb{Z}_{+}^{n}$, an integer $p'> p=\sum_{j=1}^{n} p_j $, and constants $\Lambda ,\epsilon ,Q >0$, there are constants $m_0 , C$ which depending on $\Lambda $, $n$, $p$, $p'$, $\epsilon $, $Q$ with the following property.

Let $(M,\omega ,\linebundle ,h)$ be a polarized manifold such that $\Ric\left( \omega \right) \geq -\Lambda \omega $, $\Ric (h) \geq \epsilon \omega $, and $x\in M$. Suppose that $M$ is pseudoconvex if $(M, \omega)$ isn't complete. Assume that there exists a local coordinate $\left( z_1 ,\cdots ,z_n \right) : B_{\frac{\log m}{Q \sqrt{m}}} (x) \to U\subset \mathbb{C}^n $, such that $x= (0,\cdots ,0) $, $ e^{-Q} \omega_{Euc} \leq \omega \leq e^Q \omega_{Euc} $, and the hermitian matrix $\left( g_{i\bar{j}} \right) $ satisfies that $g_{i\bar{j}} (0) = \delta_{ij} $ , $dg_{i\bar{j}} (0) = 0$, and $\left\Vert g_{i\bar{j}} \right\Vert^{*}_{1,\frac{1}{2}; U} \leq Q $, where $\left\Vert f \right\Vert^{*}_{k,\alpha } $ is the interior norms on a domain in $\mathbb{R}^{2n}$. We further assume that there exists a holomorphic frame $e_\linebundle $ of $\linebundle$ on this coordinate such that the local representation function of $h$, $a=h\left( e_\linebundle ,e_\linebundle \right) $, satisfying that $a(0)=1$, $\frac{\partial^{|I|} a}{\partial z^I }  (0) =0$ for each milti-index $I$ with $|I|\leq 3$, and $\left\Vert a \right\Vert^{* }_{3,\frac{1}{2} ;U }  \leq  \frac{Q \left( \log m \right)^2 }{m}$.

Then there are sequences $a_m$ and $b_m $, smooth $\linebundle $-valued sections $\varphi_m $, and holomorphic global sections $S_m $ in $H^{0} \left( M,\linebundle^m \right)$, $\forall m>m_0$, satisfying
\begin{eqnarray*}
\int_M \left\Vert \varphi_{m} \right\Vert_{h^m}^{2} dV_\omega &\leq & \frac{C}{m^{8p' +2n}} ,\\
\int_{M} \left\Vert S_m \right\Vert_{h^m}^{2} dV_\omega & = & 1, \\
\int_{M\big\sq\left\lbrace |z|\leq \frac{\log(m) }{Q\sqrt{m}} \right\rbrace }\left\Vert S_m \right\Vert_{h^m}^{2} dV_\omega  & \leq &   \frac{C}{m^{2p'}}  ,
\end{eqnarray*}
and locally at $x$,
\begin{align*}
S_m (z)  =  \lambda_{P} \left( 1+ \frac{a_m}{m^{2p'}} \right) \left( z_1^{p_1} \cdots z_n^{p_n} +\varphi_m \right) e_L^m ,
\end{align*}
where $||\cdot ||_{h^m}$ is the norm on $\linebundle^m$ given by $h^m$, $|a_m |\leq C $, $\varphi_m $ is holomorphic on $\left\lbrace |z|\leq \frac{\log(m) }{Q\sqrt{m}} \right\rbrace$, and $||\varphi_{m} ||_{h^m} \leq b_m |z|^{2p'} $ on $U$, moreover
\begin{align*}
\lambda_{P}^{-2} = \int_{\left\lbrace |z|\leq \frac{\log(m) }{Q \sqrt{m}} \right\rbrace } \left| z_1^{p_1} \cdots z_n^{p_n} \right|^2 a^m dV_\omega ,
\end{align*}
and hence we can assume that $ \left| m^{n+p} \lambda^{-2}_P - \pi^n P! \right| \leq Cm^{-\frac{1}{2}} $, where $dV_\omega = \left( \frac{\sqrt{-1}}{2}\right)^n \det\left( g_{i\bar{j}}\right) dz_1 \wedge d\zbar_1 \wedge \cdots \wedge dz_n \wedge d\zbar_n $ is the volume form.
\end{lmm}

\begin{rmk}
By the theory of Cheeger-Gromov convergence, we can find such coordinates and holomorphic frames on polarized manifolds with a bound of Ricci curvature and a lower bound of injective radius. See \cite[Theorem 1.1]{ma1} or \cite[Chapter 11]{pp2} for details.
\end{rmk}

As a corollary, Tian shows that the global holomorphic sections constructed above have the following almost-orthogonal property.

\begin{coro}[\cite{tg1}, Lemma 3.1]
\label{coropeaksec}
Let $S_m$ be a holomorphic global section constructed in Lemma \ref{peaksec}, and $T$ be another holomorphic global section of $\linebundle^m$ with $\int_M \left\Vert T \right\Vert_{h^m}^2 d\V_\omega =1 $, which contains no term $z^P$ in its Taylor expansion at $x_0$. Then
$$ \left| \int_M \left\langle S_m , T \right\rangle_{h^m} d\V_\omega \right| \leq Cm^{-\frac{1}{2}} , \;\; \forall m\geq m_0 ,$$
where $m_0 ,C$ are positive constants depending on $\epsilon ,\Lambda ,Q$.
\end{coro} 

Then we recall the Tian-Yau-Zelditch expansion theorem of Bergman kernels on a given manifold. The proof can be found in \cite{cat1,dailiuma1,liuzql1,sz1}. Note that Liu-Lu's proof in \cite{liuzql1} only relies on Tian's peak section method.

\begin{thm}
\label{thmregularpart}
For constants $\Lambda ,\epsilon ,\xi ,\delta k,Q >0$, there are constants $m_0 , m_1 , C$ which depending on $\Lambda $, $\epsilon $, $\xi $, $\delta $, $k$, $Q$ with the following property.

Let $(M,\omega )$ be a K\"ahler manifold such that $\Ric\left( \omega \right) \geq -\Lambda \omega $, let $(\linebundle ,h)$ be a Hermitian line bundle on $(M,\omega )$ such that $\Ric (h) \geq \epsilon \omega $, and let $U$ be an open subset of $M$. If $(M, \omega)$ isn't complete, we suppose that $M$ is pseudoconvex. Write $U_m = \left\lbrace y\in M \big| \dist_{\omega} (y,U) < \frac{\xi \log m}{\sqrt{m}} \right\rbrace $ for each $m>m_0 $. Assume that $\bar{U}_{m_0}$ is complete, $\Ric (h) = \omega $ on $U_m $, $\inf_{x\in U_m } \inj (x)\geq \frac{ \delta \log m}{\sqrt{m}} $, and $\sum_{j=0}^{2k} \sup_{U_m} \left\Vert \nabla^j \Ric (\omega ) \right\Vert \leq Q .$

Then there exist coefficients $a_j \in C^\infty (M)$, $\forall j\in\mathbb{N} $, such that
$$ \left\Vert \rho_{\omega ,m} - \sum_{j=0}^{k} a_j m^{n-j} \right\Vert_{C^k (U)} \leq Cm^{n-k-1} ,\;\; \forall m>m_1 .$$
\end{thm}

\begin{rmk}
In Lemma \ref{peaksec}, Corollary \ref{coropeaksec} and Theorem \ref{thmregularpart}, we assume the complete and pseudoconvex condition just so that H\"ormander's $L^2$ estimate can be used. So these results still hold when $M$ is replaced by an open subset of $M$.
\end{rmk}

\section{Bergman kernels on complex hyperbolic cusp}
\label{model}

In this section, we calculate the Bergman kernels on the complex hyperbolic cusp by find a special $L^2$ orthonormal basis, and then we consider the asymptotic behavior of Bergman kernels. We only consider the case $n\geq 2$ in this section, because the expression of Bergman kernels on the punctured unit disc is known. For more details, see Subsection 3.1 in \cite{aumamar1}.

\subsection{Expression of the Bergman kernels on complex hyperbolic cusps}
\hfill

First we consider the expansion of holomorphic functions on the total space of line bundle. We will use the letter $\linebundle $ to denote a holomorphic line bundle on a complex manifold $X$. Let $V$ be a connected open neighborhood of $X$ in the total space $\linebundle $. Since $\linebundle^{-1} \otimes  \linebundle \cong \mathbb{C} $, we see that for each section $S\in H^0 \left( X,\linebundle^{-1} \right)$, the map $v\mapsto S\left( v \right)$ gives a holomorphic function on the total space of $\linebundle $. Then we show that any holomorphic function on $V\setminus X$ has a similar structure. 

\begin{lmm}
\label{holoontotal}
Let $f$ be a holomorphic function on $V \setminus X$. Then there exists a sequence of holomorphic sections $S_k \in H^0 \left( M,{\linebundle}^{-k} \right) ,$ $\forall k\in\mathbb{Z} ,$ such that $f$ has the following power series expansion on $V \setminus X$: 
$$ f(v) = \sum_{k\in\mathbb{Z}} f_{S_k} \left( v \right) =\sum_{k\in\mathbb{Z}} S_k \left( v^{k} \right) ,\;\; \forall v \in V\sq X \subset \linebundle , $$
where $v^k = v^{\otimes k} $ is a point in the total space of $\linebundle^k $.

Moreover, if $X$ is compact and $\linebundle $ is a negative line bundle, then we have $S_k =0$, $\forall k<0$.
\end{lmm}

\begin{proof}
Let $\left\lbrace U_j \right\rbrace_{j\in J} $ be an open neighborhood of $X$ such that for each $j$, there exists a trivialization $\psi_j :\linebundle |_{U_j} \cong U_j \times \mathbb{C} $ satisfies that $U_j \times \mathbb{D}_1 \subset  \psi_j \left( V\cap \linebundle |_{U_j} \right) $, where $\mathbb{D}_1 $ is the unit disc in $\mathbb{C}$. By the Laurent series expansion on $\mathbb{D}^*_1 = \mathbb{D}_1 \setminus \left\lbrace 0 \right\rbrace $, we can find a sequence of holomorphic functions $\left\lbrace \varphi_{j,k} \right\rbrace_{k\in\mathbb{Z}} \subset \mathcal{O} \left( U_j \right) $ for each $j$, such that $$f\circ \psi_j^{-1} (z,w_j) = \sum_{k\in\mathbb{Z}} \varphi_{j,k} (z) w_j^k ,\; \forall  z \in U_j ,\; w_j \in \mathbb{D}^*_{1} ,$$ 
and the coefficients satisfy
$$\varphi_{j,k} (z) = \frac{1}{2\pi \sqrt{-1}} \int_{\partial \mathbb{D}_1 } \frac{f\circ \psi_{j}^{-1} \left( z,w_j \right) }{w_j^{k+1}} dw_j .$$

Assume that $ U_{j_1} \cap U_{j_2} \neq \emptyset $ for some $j_1 ,j_2 \in J $. Then the transition function of $\linebundle $ on $U_{j_1} \cap U_{j_2} $ is a holomorphic function $\psi_{j_2 j_1} \in \mathcal{O} \left( U_{j_1} \cap U_{j_2} \right) $ defined by $\psi_{j_2}^{-1} \left( z, \psi_{j_2 j_1} (z) \right) = \psi_{j_1}^{-1} (z,1) $, $\forall z\in U_{j_1} \cap U_{j_2} $. The Cauchy integral formula now shows that
\begin{eqnarray*}
\varphi_{j_1 ,k} (z) & = & \frac{1}{2\pi \sqrt{-1}} \int_{\partial \mathbb{D}_1 } \frac{f\circ \psi_{j_1}^{-1} \left( z,w_{j_1} \right) }{w_{j_1}^{k+1}} dw_{j_1} \\
& = & \frac{1}{2\pi \sqrt{-1}} \int_{\partial \mathbb{D}_1 } \frac{f\circ \psi_{j_2}^{-1} \left( z, \psi_{j_2 j_1} (z) w_{j_1} \right) }{w_{j_1}^{k+1}} dw_{j_1} \\
& = & \frac{1}{2\pi \sqrt{-1}} \int_{\partial \mathbb{D}_{\psi_{j_2 j_1} (z)} } \frac{f\circ \psi_{j_2}^{-1} \left( z,  w_{j_2} \right) }{\psi^{-k-1}_{j_2 j_1} (z)w _{j_2}^{k+1}} d \left( \psi^{-1}_{j_2 j_1} (z) w_{j_2} \right) \\
& = & \frac{\psi^{k}_{j_2 j_1} (z)}{2\pi \sqrt{-1} } \int_{\partial \mathbb{D}_{1} } \frac{f\circ \psi_{j_2}^{-1} \left( z,  w_{j_2} \right) }{w_{j_2}^{k+1}} d w_{j_2} \\
& = & \psi^{k}_{j_2 j_1} (z) \varphi_{j_2 ,k} (z).
\end{eqnarray*}

It follows that $\left\lbrace \varphi_{j,k} \right\rbrace_{j\in J}$ gives a holomorphic section $S_k \in H^{0} \left( X,\linebundle^{-k} \right) $. Then we can find an open neighborhood $V_1$ of $X$ in $V$, such that $f$ has the above expansion in $V_1 \setminus X $. Then the connectedness of $V$ implies that $f$ has this expansion in the whole $V \setminus X$.

Now we assume that $X$ is compact and $\linebundle $ is negative. By the Kodaira embedding theorem, we can find a constant $k_0 >0$ such that $H^{0} \left( X,\linebundle^{-k} \right) \neq 0 $ for each $k>k_0$. If $H^{0} \left( X,\linebundle^{k} \right) \neq 0 $ for some $k>0$, then there exists an integer $k_1 >k_0$ such that $H^{0} \left( X,\linebundle^{k_1} \right) \neq 0 $. It follows that we can find a holomorphic function $\phi =S\otimes S' \in H^0 \left( X,\linebundle^{-k_1} \otimes \linebundle^{k_1} \right) \cong H^0 \left( X,\mathbb{C} \right) $ such that $\phi \neq 0$, but $\phi (x)=0$ for some $x$. But the compactness of $X$ implies that $H^0 \left( X,\mathbb{C} \right) \cong \mathbb{C} $, contradiction. Hence we have $H^{0} \left( X,\linebundle^{k} \right) \neq 0 $, $\forall k>0$. This gives $S_k =0$, $\forall k<0$. This completes the proof. 
\end{proof}

Let $\left( \linebundle_D ,h_D \right) $ be a negative line bundle on $D$, and $\left( V,\omega_V ,\linebundle_V ,h_V \right)$ be an $n$-dimensional $(n\geq 2)$ complex hyperbolic cusp defined in Section \ref{intro} now. 

For each $x\in D$, we can find an open neighborhood $U$ of $x$ in $D$, a biholomorphic chart $z = \left( z_1 ,\cdots ,z_{n-1} \right) :U\cong \mathbb{B}^{n-1}_{1} \subset \mathbb{C}^{n-1} $, and a non-vanishing section $e_U \in H^0 \left( U,\linebundle_D \right) $. Let $w$ be a holomorphic map defined by $v=w(v)e_U $, $\forall v\in \linebundle_D \big|_{U} $. Recall that $V$ is a subset of the total space of $\linebundle_D$, we see that the holomorphic map $(z,w)$ gives a holomorphic chart:
$$ (z,w) : \linebundle_D \big|_{U} \cap V \cong  \left\lbrace (z,w) \in \mathbb{B}^{n-1}_{1} \times \mathbb{C} \; \bigg| \; 0< |w| \cdot \left\Vert e_U \right\Vert_{h_D} <1 \right\rbrace .$$

Set $\varphi = \log \left( \left\Vert e_U \right\Vert^2_{h_D} \right) $. Then the volume form on $\pi_D^{-1} (U) \cap V $, $d\V_{\omega_V }$, can be calculated as following:
\begin{eqnarray}
d\V_{\omega_V} & = & \frac{ \left( -\sqrt{-1} \partial\partialbar \log \left( -\log h_D \right) \right)^n }{n!} \nonumber \\
& = & \frac{1}{n!} \left( \sqrt{-1} \frac{\partial\partialbar \log h_D}{ \left| \log h_D \right|} +\sqrt{-1} \frac{\partial \log h_D \wedge \partialbar \log h_D}{\left| \log h_D \right|^2 } \right)^n \nonumber \\
& = & \frac{1}{n!} \left( \sqrt{-1} \frac{\partial\partialbar \varphi }{\left| \log h_D \right|} +\sqrt{-1} \frac{\partial \log h_D \wedge \partialbar \log h_D}{\left| \log h_D \right|^2 } \right)^n \label{calculationvolumeformhyperboliccusp}\\
& = & \frac{\left(\sqrt{-1} \right)^n }{n!} \cdot \frac{ \left( \partial\partialbar \varphi \right)^{n-1} \wedge \partial \left( \log |w|^2 +\varphi \right) \wedge \partialbar \left( \log |w|^2 +\varphi \right) }{\left| \log h_D \right|^{n+1} } \nonumber \\
& = & \frac{ \sqrt{-1} \pi_D^* \left( d\V_{\omega_D} \right) \wedge dw \wedge d \bar{w}  }{n |w|^2 \left| \log h_D \right|^{n+1} } = \frac{ \pi_D^* \left( d\V_{\omega_D} \right) \wedge \omega_V }{n \left| \log h_D \right|^{n-1} } . \nonumber
\end{eqnarray}

Now we can find an orthonormal basis of $\hl \left( V,\linebundle^m_V \right) $.

\begin{lmm}
\label{lmmorthonormalfver}
For $m\geq n+1$, the set
\begin{equation}
\label{orthonormalsetfver}
\left\lbrace \left( \frac{n\cdot q^{m-n} }{ 2\pi (m-n-1)! } \right)^{\frac{1}{2}} f_{S_{q,j} } e^m_{\linebundle_V} \bigg| \; 1\leq j \leq N_q ,\; q\geq 1  \right\rbrace
\end{equation}
forms an orthonormal basis of $\hl \left( V,\linebundle^m_V \right) $, where $\left\lbrace S_{q,j} \right\rbrace_{j=1}^{N_q} $ is an $L^2 $ orthonormal basis in the Hilbert space $\hl \left( D,\linebundle^{-q}_D \right) $, and $f_{S_{q,j} } $ denote the holomorphic function corresponding to $S_{q,j} $ defined in Lemma \ref{holoontotal}.
\end{lmm}

\begin{proof}
Let $\widetilde{S}_1 $ and $\widetilde{S}_2 $ be the holomorphic functions corresponding to the holomorphic sections $S_1 \in  H^{0} \left( D,\linebundle_D^{-q_1} \right) $ and $S_2 \in  H^{0} \left( D,\linebundle_D^{-q_2} \right) $, respectively. Then we have
\begin{eqnarray*}
\int_{V} \left\langle f_{{S}_1 } e^m_{\linebundle_V} , f_{{S}_2} e^m_{\linebundle_V} \right\rangle_{h_V^m} d\V_{\omega_V} & = & \int_{V}  f_{S_1 } \bar{f}_{S_2} \left| \log h_D \right|^{m} d\V_{\omega_V} \\
& = & \frac{1 }{n } \int_{V} f_{S_1 } \bar{f}_{S_2} \left| \log h_D \right|^{m+1-n}    d\V_{\omega_D}  \wedge \omega_{V} \\
& = & \int_{D} \left( \frac{ 1 }{n } \int_{V_z}  f_{S_1 } \bar{f}_{S_2} \left| \log h_D \right|^{m+1-n} \omega_{V} \right) d\V_{\omega_D} ,
\end{eqnarray*}
where $V_z = \pi^{-1}_{D} (x)\cap V$, $z\in D$, and $\pi_D : \linebundle_D \to D $ is the projection of line bundle. For each $z \in D$, we can find $e\in \pi^{-1}_{D} (x) $ such that $h_D (e)=1 $. Then we have
\begin{eqnarray*}
\frac{ 1 }{n } \int_{V_z}  f_{S_1 } \bar{f}_{S_2} \left| \log h_D \right|^{m+1-n} \omega_{V} & = & \sqrt{-1} S_1 \left( e^{q_1} \right) \overline{S_2 \left( e^{q_2} \right) } \int_{\mathbb{D}^*_1} \frac{ w^{q_1} \bar{w}^{q_2} \left| \log |w|^2 \right|^{m-1-n} }{n |w|^2 } dw \wedge d\bar{w} \\
& = & \sqrt{-1} \delta_{q_1 ,q_2} h^{-q_1} \left( S_1 ,S_2 \right) \int_{\mathbb{D}^*_1} \frac{ |w|^{2 q_1 -2} \left| \log |w|^2 \right|^{m-1-n} }{n } dw \wedge d\bar{w} ,
\end{eqnarray*}
where $ \delta_{q_1 ,q_2} $ is the Kronecker symbol. It follows that
$$ \int_{V} \left\langle f_{{S}_1 } e^m_{\linebundle_V} ,f_{S_2} e^m_{\linebundle_V} \right\rangle_{h_V^m} d\V_{\omega_V} = 0 ,$$
if $q_1 \neq q_2 $. Now we assume that $q_1 =q_2 =q$. Since $m\geq n+1$, we have 
$$ \int_{\mathbb{D}^*_1} \sqrt{-1} |w|^{ -2} \left| \log |w|^2 \right|^{m-1-n}  dw \wedge d\bar{w} = \int_{0}^{2\pi} 2^{m-n-1} d\theta \int_{0}^{1} r^{-1} \left| \log r \right|^{m-1-n} dr =\infty ,$$
and hence $e_{\linebundle_V}^m \notin \hl \left( V,\linebundle^m_V \right) .$ When $q\geq 1$, by a direct computation, we obtain
\begin{eqnarray*}
\int_{\mathbb{D}^*_1} \frac{ |w|^{2 q -2} \left| \log |w|^2 \right|^{m-1-n} }{ -\sqrt{-1}  } dw \wedge d\bar{w} & = & \int_{0}^{2\pi} 2^{m-n} d\theta \int_{0}^{1} r^{2q-1} \left| \log r \right|^{m-1-n} dr \\
& = & 2^{m-n+1} \pi \int_{0}^{\infty} e^{-2q t} t^{m-1-n} dt \\
& = & 2^{m-n+1} \pi \cdot (2 q)^{n-m} \cdot (m-n-1)! = \frac{2\pi (m-n-1)!}{q^{m-n}} .
\end{eqnarray*}
Then we have
\begin{eqnarray*}
\int_{V} \left\langle f_{S_1} e^m_{\linebundle_V} ,f_{S_2} e^m_{\linebundle_V} \right\rangle_{h_V^m} dV_{\omega_V} & = &  \int_{D} \left( \frac{ 1 }{n } \int_{V_z}  f_{S_1} \bar{f}_{S_2} \left| \log h_D \right|^{m+1-n} \omega_{V} \right) dV_{\omega_D} \\
& = & \int_{D} h^{-q} \left( S_1 ,S_2 \right) dV_{\omega_D} \int_{\mathbb{D}^*_1} \frac{\sqrt{-1} |w|^{2 q -2} \left| \log |w|^2 \right|^{m-1-n} }{n } dw \wedge d\bar{w} \\
& = & \frac{2\pi (m-n-1)!}{n\cdot q^{m-n} } \int_{D} h_D^{-q} \left( S_1 ,S_2 \right) d\V_{\omega_D} .
\end{eqnarray*}
Combining Lemma \ref{holoontotal} and the above calculation, we can conclude that the set (\ref{orthonormalsetfver}) forms an orthonormal basis of $\hl \left( D,\linebundle^{-q}_D \right) $.
\end{proof}

\begin{rmk}
In the case $n=1$, analysis similar to that in the proof of Lemma \ref{lmmorthonormalfver} shows that the set
\begin{equation}
\label{orthonormalsetdim1}
\left\lbrace \left( \frac{n\cdot q^{m-1} }{ 2\pi (m-2)! } \right)^{\frac{1}{2}} e^m_{\linebundle_V} \bigg| \; q\geq 1  \right\rbrace
\end{equation}
forms an orthonormal basis of $\hl \left( V,\linebundle^m_V \right) $ when $m\geq 2$.
\end{rmk}

As another application of the calculation of the volume form on $\left( V ,\omega_V \right) $, we can estimate the injectivity radius on $\left( V ,\omega_V \right) $.

\begin{lmm}
\label{lmminjectivityradiusoncusp}
There exists a constant $\delta =\delta (D) >0 $, such that the injectivity radius $\inj_V (x)  $ becomes a strictly increasing function of $h_D (x) $ on $V_\delta $, and $-\frac{ \delta }{ \log \left( h_D (x) \right) } \leq \inj_V (x) \leq -\frac{ \pi }{ \log \left( h_D (x) \right) } $, $\forall v\in V_\delta $, where $\inj_V (x)$ is the injectivity radius at $x\in V$.
\end{lmm}

\begin{proof}
Let $\pi_{Q} : \mathbb{C}^{n-1} \to D $ be the universal covering map of $D$. Then the morphism $ \pi_{Q}^* \linebundle_D \to \linebundle_D $ induced by $\pi_{Q} $ is also a universal covering map. For abbreviation, we use the same letter $\pi_{Q} $ to designate it. Since $\mathrm{Pic} \left( \mathbb{C}^{n-1} \right) =0 $, we have $\left( \pi_{Q}^* \linebundle_D ,\pi^* h_D \right) \cong \left( \mathbb{C} ,e^{|z|^2} \right) $, and the isomorphism gives a holomorphic chart $(z,w)\in\mathbb{C}^{n-1} \times \mathbb{C} $ of $\pi_{Q}^{-1} (V) $. It is easy to shows that $\inj_{\pi_{Q}^{-1} (V)} (z,w)$ is a function of $h_D \circ \pi_{Q} (z,w) $ on $\pi_{Q}^{-1} (V) $. For each smooth curve $\gamma :[0,1] \to \pi_{Q}^{-1} (V) $, we can construct a family of curves by $\gamma_t (s) = \left( z \left( \gamma (s) \right) , t w\left( \gamma (s) \right)  \right)  $, $\forall t\in (0,1) $. By a straightforward computation, we obtain $ \left| \gamma'_{t} (s) \right| \leq \left| \gamma' (s) \right| $, and equality holds only when $\left| \gamma' (s) \right| =0 $. It follows that $\inj (z,w) $ is a strictly increasing function of $h_D \circ \pi_{Q} (z,w) $ on $\pi_{Q}^{-1} (V) $.

Let $\gamma (s) = \left( 0,e^{\sqrt{-1} s} w \right) $, $s\in [0,2\pi ]$. Clearly, $ [\gamma ] \neq 0 $ in the fundamental group $ \pi_1 \left( \pi_Q^{-1} (V) \right) $. Then $L (\gamma) = \frac{2\pi }{-\log \left( |w|^2 \right) } $ implies that $\inj_{\pi_{Q}^{-1} (V)} (z,w) \leq -\frac{ \pi }{ \log \left( h_D (v) \right) } $. Choosing $x_1 ,x_2 \in \pi_{Q}^{-1} (x) $ and a smooth curve $\gamma :[0,1] \to \pi_{Q}^{-1} (V) $ such that $x_1 \neq x_2 $, $\gamma (0) =x_1 $ and $\gamma (1) =x_2 $. Assume that $L \left( \gamma \right) \leq 1 $. Then we have $\inf_\gamma \left| w(\gamma ) \right|^2 \geq \left| h_D (x) \right|^e  $, and there exists a constant $\epsilon =\epsilon (n) >0 $ such that $ \dist_{\omega_V} \left( x_1 ,x_2 \right) \geq \frac{\epsilon \left| z\left( x_1 \right) - z\left( x_2 \right) \right| }{ \sqrt{ -e \log \left( h_D (x) \right) } }  .$ Then we can find a constant $\delta_1 =\delta_1 (D) >0 $ such that $\inj_{ \pi_{Q}^{-1} (V)} = \inj_V \circ \pi_{Q} $ on $\pi_{Q}^{-1} \left( V_{\delta_1} \right) $.

Since the holomorphic sectional curvature of $\omega_V$ is a constant, we can conclude that the sectional curvature of $\omega_V$ is bounded. Recall that the volume form on $V$ can be expressed as
$$ d\V_{\omega_V} = \frac{ \pi_D^* \left( d\V_{\omega_D} \right) \wedge \omega_V }{n \left| \log h_D \right|^{n-1} } .$$
Thus we can find a constant $\delta_2 >0$ depending only on $\left( V,\omega_V  \right)$, such that 
$$ \Vol_{\omega_V} \left( B_{\frac{-1 }{\log \left( h_D (x) \right) } } (x) \right) \geq \delta_2 \left( \frac{-1 }{\log \left( h_D (x) \right) } \right)^{2n} ,\forall x\in V_{\delta_2 } .$$
Then the Cheeger's lemma implies this lemma.
\end{proof}

Let $\Gamma_{V} $ be a finite subgroup of $\Aut \left( V,\omega_V \right) $. Then we see that the injectivity radius function is invariant under the action of $\Gamma_{V} $. Hence $h_D $ is invariant under the action of $\Aut \left( V,\omega_V \right) $ on $V$, and the action of $\Gamma_{V} $ on $V$ can be extended to the total space of $\linebundle_D $. Now we further extend the action of $ \Gamma_V $ on $V$ to the line bundle $\linebundle_V $ such that $h_V $ is invariant under this action. Notice that by the invariance of $h_V $ under the action of $ \Gamma_V $, we have a homomorphism $\Theta_{\Gamma_V } : \Gamma_V \to U(1) \subset \mathbb{C}^* $ such that $g \left( e_{\linebundle_V } \right) = \Theta_{\Gamma_V } (g) e_{\linebundle_V }  $, $\forall g \in \Gamma_V $.

Fix $m\geq n+1 $. Now we consider the group action $\Gamma_V \times H^0 \left( V,\linebundle^m_V \right) \to H^0 \left( V,\linebundle^m_V \right) $ given by
\begin{equation}
\label{groupactionconstructiononH0VLVm}
\Delta_{\Gamma_V ,m} : \left\{ \begin{array}{ll}
\Gamma_V \times H^0 \left( V,\linebundle^m_V \right) & \to \;\;\;\;\;\;\; H^0 \left( V,\linebundle^m_V \right) ,\\
\;\;\;\;\; \left( g,fe_{\linebundle_V}^m \right) & \mapsto \Theta_{\Gamma_V } (g)^{m} \left( f \circ g^{-1} \right) e_{\linebundle_V}^m , 
\end{array} \right.
\end{equation}
for any $ g\in\Gamma_V \subset \Aut \left( V,\omega_V \right) $ and $ f\in\mathcal{O}(V) $.

Let $\hl \left( V,\linebundle^m_V \right)_q $ be the subspace of $ \hl \left( V,\linebundle^m_V \right) $ generated by $\left\lbrace f_{S_{q,j}} e^m_{\linebundle_V} \right\rbrace_{j=1}^{N_q} $, and $\hl \left( V,\linebundle^m_V \right)_{ \Gamma_V } $ be the subspace of $ \hl \left( V,\linebundle^m_V \right) $ containing all holomorphic sections that are invariant under the action $\Delta_{\Gamma_V ,m} $. Since $\hl \left( V,\linebundle^m_V \right)_q $ is invariant under the action $\Delta_{\Gamma_V ,m} $, we see that the Hilbert space $\hl \left( V,\linebundle^m_V \right)_{\Gamma_V } $ can be generated by $ \oplus_{q=1}^{\infty } \left( \hl \left( V,\linebundle^m_V \right)_{\Gamma_V } \cap \hl \left( V,\linebundle^m_V \right)_q \right) $. Then there are non-negative integers $N'_{m,q} \leq N_q $ and $L^2$ orthonormal sets $ \left\lbrace S'_{m,q,j} \right\rbrace_{j=1}^{N'_{m,q} } \subset \hl \left( D,\linebundle^{-q}_D \right)_{\Gamma_V } $ such that the set
$$ \left\lbrace \left( \frac{n\cdot q^{m-n} }{ 2\pi (m-n-1)! } \right)^{\frac{1}{2}} f_{S'_{m,q,j}} e^m_{\linebundle_V} \bigg| \; 1\leq j \leq N'_{m,q} ,\; q\geq 1  \right\rbrace $$
forms an orthonormal basis of $\hl \left( V,\linebundle^m_V \right)_{\Gamma_V } $. Now we can define the $\Delta_{\Gamma_V ,m} $-invariant Bergman kernel functions by $\rho_{D,\omega_D ,\Delta_{\Gamma_V ,m} ,q} = \sum_{j=1}^{N'_{m,q}} \left\Vert S'_{m,q,j} \right\Vert_{h_D^{-q}}^2 $.

Then we can obtain the expression of Bergman kernel functions on $V/\Gamma_V $ for $m\geq n+1$.

\begin{prop}
\label{propbergmanknv}
For $m\geq n+1$, the Bergman kernel functions on $\left( V/\Gamma_V ,\omega_V ,\linebundle_V /\Gamma_V ,h_V \right)$ can be expressed as
\begin{equation}
\label{expresspropbergmanknv}
\rho_{ V/\Gamma_V , \omega_V ,m} \left( \pi_V (x) \right) =  \frac{n \left| \Gamma_V \right| \left| \log h_D \left( x \right) \right|^m }{ 2\pi (m-n-1)! } \sum_{q=1}^{\infty} q^{m-n} \rho_{D,\omega_D ,\Delta_{\Gamma_V ,m} ,q} \left( \pi_D (x) \right) h_D (x)^q ,
\end{equation}
where $x\in V$, $\pi_V :V\to V/\Gamma_V $ is the quotient map, and $ \pi_D : V\subset \linebundle_D \to D $ is the projection.
\end{prop}

\begin{proof}
By Lemma \ref{lmmorthonormalfver}, we have
\begin{eqnarray*}
\rho_{ V/\Gamma_V , \omega_V ,m} \left( \pi_V (x) \right) & = & \left| \Gamma_V \right| \sum_{q=1}^{\infty} \frac{n q^{m-n } }{2\pi (m-n-1)!} \left( \sum_{j=1}^{N'_{m,q}} \left| \widetilde{S}'_{m,j,q} (x) \right|^2 \cdot \left\Vert e^m_{\linebundle_V} \right\Vert^2_{h_\linebundle^m}  \right) \\
& = & \left| \Gamma_V \right| \sum_{q=1}^{\infty} \frac{n q^{m-n } h_D (x)^q \left| \log h_D (x) \right|^m }{2\pi (m-n-1)!} \left( \sum_{j=1}^{N'_{m,q}} \left\Vert S'_{m,q,j} \left( \pi_D (x) \right) \right\Vert_{h_D^{-q}}^2  \right) \\
& = & \frac{n \left| \Gamma_V \right| \left| \log h_D (x) \right|^m }{ 2\pi (m-n-1)! } \sum_{q=1}^{\infty} q^{m-n} \rho_{D,\omega_D ,\Delta_{\Gamma_V ,m} ,q} \left( \pi_D (x) \right) h_D (x)^q ,
\end{eqnarray*}
which proves the proposition. 
\end{proof}

\begin{rmk}
As in the proof of Proposition \ref{propbergmanknv}, we see that for any integer $k\geq 0$ and $m\geq n+1$, the Bergman kernel functions on $\left( V/\Gamma_V ,\omega_V ,\linebundle_V /\Gamma_V , \left| \log h_D \right|^k h_V \right)$ can be expressed as
\begin{equation}
\label{expresspropbergmanknvloghd}
\rho^{(k)}_{ V/\Gamma_V , \omega_V ,m} \left( \pi_V (x) \right) =  \frac{n \left| \Gamma_V \right| \left| \log h_D \left( x \right) \right|^{m(k+1)} }{ 2\pi (m(k+1) -n-1)! } \sum_{q=1}^{\infty} q^{m(k+1) -n} \rho_{D,\omega_D ,\Delta_{\Gamma_V ,m} ,q} \left( \pi_D (x) \right) h_D (x)^q ,
\end{equation}
where $x\in V$, $\pi_V :V\to V/\Gamma_V $ is the quotient map, and $ \pi_D : V\subset \linebundle_D \to D $ is the projection.
\end{rmk}

\subsection{Asymptotic behavior of Bergman kernel functions on complex hyperbolic cusps}
\hfill

In this subsection, we will consider the asymptotic behavior of Bergman kernel functions on complex hyperbolic cusp. 

Let $\left( V/ \Gamma_V ,\omega_V ,\linebundle_V / \Gamma_V ,h_V \right)$ be a complex hyperbolic cusp defined on an polarized flat Abelian variety $\left( D,\omega_D \right)$ with a polarization $\left( \linebundle^{-1}_D ,h^{-1}_D \right)$. Now we consider the Bergman kernel function near cusp. We can use an argument similar to that in \cite[Section 3]{aumamar1}. 

For each integer $m\geq n+1$, set $$ b_m (a,t) = \frac{n \left| \Gamma_V \right| t^m }{ 2\pi (m-n-1)! } \sum_{q=1}^{\infty} q^{m-n} \rho_{D,\omega_D ,\Delta_{\Gamma_V ,m} ,q} (a) e^{-qt} ,\;\; \forall a\in D ,\; \forall t\in (0,\infty ) .$$
It is clear that $\rho_{V/\Gamma_V , \omega_V ,m} (x ) =  b_m \left( \pi_{D} (x) , -\log h_D (x) \right) $, for any $x\in V$. Then we consider the Taylor expansion of each term in this series. 

Let $f (t) = (1+t) e^{-t} $. Then we see that 
$$ b_m (a,mt ) = \frac{n \left| \Gamma_V \right| m^m e^{-m} }{ 2\pi (m-n-1)! } \sum_{q=1}^{\infty} q^{-n} \rho_{D,\omega_D ,\Delta_{\Gamma_V ,m} ,q} (a) \left( f(qt-1 ) \right)^m ,\;\; \forall x\in D ,\; \forall t\in (0 ,\infty ) .$$
Now we consider the asymptotic behavior of $f_m $ as $m\to\infty $. We can assume that
$$ \sum_{j=1}^{\infty } \frac{1}{j!} \left[ m\sum_{k=3}^{\infty } \frac{ (-1)^{k-1} }{ k } t^{k} \right]^j = \sum_{l=3}^{\infty } \lambda_{m,l} t^l \in \mathbb{R}[[ t ]] ,$$
where $\lambda_{m,l } \in\mathbb{R} $, and $ \mathbb{R}[[ t ]] $ is the formal power series ring over $\mathbb{R} $. Clearly, we have $ \lambda_{m,3} = \frac{m}{3} $.

\begin{lmm}
\label{lmmasymptoticexpansionfunctionfm}
Let $ N\geq 3 $ be an integer. Then we can find a constant $C=C(N) >0 $ such that
\begin{eqnarray}
\label{equationasymptoticexpansionfm}
 \left| \left( f(t ) \right)^m - \left( 1+ \sum_{l=3}^{N } \lambda_{m,l} t^l \right)  e^{-\frac{mt^2}{2}}  \right| \leq \frac{C  }{m^{\frac{N+1}{2} - \left\lfloor \frac{N+1}{3} \right\rfloor } \left( 1+|t+1|^{2N} \right) } ,
\end{eqnarray}
for any integer $m\geq n+1 $ and constant $t\in (-1,\infty ) $.
\end{lmm}

\begin{proof}
It is easy to see that for each given $m\in\mathbb{N} $, there exists a constant $C=C(m,N) >0 $ such that (\ref{equationasymptoticexpansionfm}) holds. So we only need to prove it for sufficiently large $m\geq e^{10N^2} $. We will divide the proof into three parts according to the value of $t$.

When $|t| \leq \frac{\log (m)}{\sqrt{m}} $, the Taylor expansion gives a constant $C_1 =C_1 (N) >0 $ such that
\begin{eqnarray*}
 \left|  \left( f(t ) \right)^m - \left( 1+ \sum_{l=3}^{N } \lambda_{m,l} t^l \right)  e^{-\frac{mt^2}{2}}  \right| & = & e^{-\frac{mt^2}{2}} \left| e^{m\left( \log (1+t) - t +\frac{ t^2}{2} \right)  } - \left( 1+ \sum_{l=3}^{N } \lambda_{m,l} t^l \right) \right| \\
& \leq & C_1 m^{ \left\lfloor \frac{N+1}{3} \right\rfloor - \frac{N+1}{2} } \left( m t^{2} \right)^{\frac{N+1}{2}} e^{-\frac{mt^2}{2}} \leq C_1^2 m^{ \left\lfloor \frac{N+1}{3} \right\rfloor - \frac{N+1}{2} } ,
\end{eqnarray*}
where $\left\lfloor \cdot \right\rfloor $ is the greatest integer function. 

By a straightforward computation, we see that $t^l e^{-\frac{mt^2}{2}} \leq C_2 e^{-\frac{mt^2}{4}} $ for $|t| \geq \frac{\log (m)}{\sqrt{m}} $ and $l=3,\cdots ,N$, where $ C_2 =C_2 (N) >0 $ is a constant. It is sufficient to prove that there are constants $m_0 ,C >0 $ depending only on $N$, such that $ \left( f(t ) \right)^m \leq  C m^{-N} (1+t)^{-N}  $ for $m\geq m_0 $, and $|t| \geq \frac{\log (m)}{\sqrt{m}} $.

When $t \leq -\frac{\log (m)}{\sqrt{m}} $, wee see that $f(t) $ is increasing, and hence the Taylor expansion implies that
$$  \left( f(t ) \right)^m \leq \left( f \left( -\frac{\log (m)}{\sqrt{m}} \right) \right)^m = e^{ m\log \left( 1- \frac{\log (m)}{\sqrt{m}} \right) +m\frac{\log (m)}{\sqrt{m}} } \leq e^{-| \log (m) |^2 +C_3 \frac{| \log (m) |^3}{\sqrt{m}} } ,$$
where $ C_3 =C_3 (N) >0 $ is a constant. Clearly, we can find a constant $m_0 = m_0  (N)>0$ such that $ \left( f(t ) \right)^m \leq m^{-N} $ for $m\geq m_0 $ and $t \leq -\frac{\log (m)}{\sqrt{m}} $.

Now we assume that $t \geq \frac{\log (m)}{\sqrt{m}} $. By a direct calculation, we can conclude that $ f_{m,N} (t) = (1+t)^{N}  \left( f(t ) \right)^m $ is descending for $t\geq \frac{N}{m} $, and hence we can assume that $ (1+t)^{N} f_m $ is descending for $t \geq \frac{\log (m)}{\sqrt{m}} $. Apply the Taylor expansion again, we have
$$ f_{m,N} (t) \leq f_{m,N} \left( \frac{\log (m)}{\sqrt{m}} \right) = e^{ (m+N)\log \left( 1+ \frac{\log (m)}{\sqrt{m}} \right) +m\frac{\log (m)}{\sqrt{m}} } \leq e^{-| \log (m) |^2 +C_4 \frac{| \log (m) |^3}{\sqrt{m}} +N\frac{\log (m)}{\sqrt{m}} } ,$$
where $ C_4 =C_4 (N) >0 $ is a constant. Clearly, we can find a constant $m_1 = m_1  (N)>0$ such that $f_m (t) \leq m^{-N} (1+t)^{-N} $ for $m\geq m_0 $ and $t \geq \frac{\log (m)}{\sqrt{m}} $. This completes the proof.
\end{proof}

Write $G_{m,N} (t) = \left( 1+ \sum_{k=3}^{N } \lambda_{m,k} t^k \right)  e^{-\frac{mt^2}{2}} $ for each $N\geq 3$. Then we can approximate the function $b_{m} $ by suitable combinations of the functions $G_{m,N}$ as follows.

\begin{lmm}
\label{lmmapproximationsumfm}
There exists a constant $C=C\left(N ,V ,\left| \Gamma_V \right| \right) >0$, such that
\begin{eqnarray}
\label{equationlmmasymptoticexpansionfm}
 & & \sup_{a\in D} \left| \frac{2\pi (m-n-1)! }{n\left| \Gamma_V \right| m^m e^{1-m} } b_m (a,mt) \right. \\
& & \;\;\;\;\;\;\;\;\; - \left. t\sum_{q=1}^\infty q^{1-n} \rho_{D,\omega_D ,\Delta_{\Gamma_V ,m} ,q} (a) e^{-qt} G_{m-1 ,N} (qt-1)  \right| \leq \frac{C  }{m^{\frac{N+1}{2} - \left\lfloor \frac{N+1}{3} \right\rfloor } } , \nonumber
\end{eqnarray}
for any integer $m\geq n+1 $ and $t\in (0,\infty ) $.
\end{lmm}

\begin{proof}
By Lemma \ref{lmmasymptoticexpansionfunctionfm}, we see that
\begin{eqnarray*}
& & \left| \frac{2\pi (m-n-1)! }{n\left| \Gamma_V \right| m^m e^{1-m} } b_m (a,mt) - t\sum_{q=1}^\infty q^{1-n} \rho_{D,\omega_D ,\Delta_{\Gamma_V ,m} ,q} (a) e^{-qt} G_{m-1 ,N} (t)  \right| \\
& \leq & t \sum_{q=1}^{\infty} q^{1-n} \rho_{D,\omega_D ,\Delta_{\Gamma_V ,m} ,q} (a) e^{-qt} \left| \left(f(qt-1) \right)^{m-1} -G_{m-1 ,N} (qt-1) \right| \\
& \leq & Cm^{-\frac{N+1}{2} + \left\lfloor \frac{N+1}{3} \right\rfloor }  \sum_{q=1}^{\infty}  t e^{-qt} \left( 1+|q t|^{2N} \right)^{-1} \\
& \leq & C^2 m^{-\frac{N+1}{2} + \left\lfloor \frac{N+1}{3} \right\rfloor } e^{-t} \int_{0}^{\infty } t \left( 1+t^N \xi^N \right)^{-1} d\xi \\
& \leq & C^2 m^{-\frac{N+1}{2} + \left\lfloor \frac{N+1}{3} \right\rfloor } e^{-t} \int_{0}^{\infty } \left( 1+ \xi^N \right)^{-1} d\xi \leq C^3 m^{-\frac{N+1}{2} + \left\lfloor \frac{N+1}{3} \right\rfloor } e^{-t} ,
\end{eqnarray*}
where $C=C\left(N ,V ,\left| \Gamma_V \right| \right)>0$ is a constant, and then we get (\ref{equationlmmasymptoticexpansionfm}).
\end{proof}

As a corollary, we can estimate the sup of Bergman kernel functions.

\begin{coro}
\label{corobgmkernelsup}
There exists a constant $C=C\left(N ,V ,\left| \Gamma_V \right| \right) >0$, such that
\begin{eqnarray}
\label{inequalitycorobgmkernelsup}
\left| (2\pi )^{\frac{3}{2}} n^{-1} m^{-n-\frac{1}{2} } \sup_{x\in V/\Gamma_V} \rho_{V/\Gamma_V ,\omega_V ,m} (x) - \left| \Gamma_V \right| \alpha_{D,m} \right| \leq \frac{C}{\sqrt{m}}  ,
\end{eqnarray}
where $\alpha_{D,m} = \sup_{a\in D} \sup_{q\in\mathbb{N} } q^{-n} \rho_{D,\omega_D ,\Delta_{\Gamma_V ,m} ,q} (a) $. Moreover, $\alpha_{D,m} =\alpha_{D,m+\left| \Gamma_V \right| } $, $ \forall m\in\mathbb{N} $, and hence there are constants $c_1 , \cdots ,c_{\left| \Gamma_V \right|} >0$ depending only on $(D,\linebundle_D )$, $\Gamma_V$ and $ \Theta_{\Gamma_V} $, such that 
\begin{eqnarray}
\label{convergencecorobgmkernelsup}
 \left| c_{ j} \left( k\left| \Gamma_V \right| +j \right)^{-n-\frac{1}{2}} \sup_{V} \rho_{V/\Gamma_V ,\omega_V ,k\left| \Gamma_V \right| +j} - 1 \right| \leq \frac{C }{\sqrt{k}} ,\;\; \forall k\in\mathbb{N} ,
\end{eqnarray}
for $j=1,\cdots ,\left| \Gamma_V \right| $.
\end{coro}

\begin{rmk}
By the standard Tian-Yau-Zelditch expansion, we see that $$ \limsup_{q\to\infty} q^{-n} \left( \sup_{a\in D} \rho_{D,\omega_D ,\Delta_{\Gamma_V ,m} ,q} (a) \right) \leq \left| \Gamma_V \right| \limsup_{q\to\infty} q^{-n} \left( \sup_{a\in D} \rho_{D,\omega_D ,q} (a) \right) = \lim_{l\to\infty} (2\pi )^{n-1} l^{-1} =0 .$$ 
Hence we can find an integer $ q_{m} \in \mathbb{N} $ and a point $ a\in D$, such that $\alpha_{D,m} = q_m^{-n} \rho_{D,\omega_D ,\Gamma_V ,q_a } (a) $.
\end{rmk}

\begin{proof}
Let $t=q_m^{-1} $ such that $\alpha_{D,m} = q_{m}^{-n} \rho_{D,\omega_D ,\Delta_{\Gamma_V ,m} ,q_m } (a) $. Then Stirling's formula implies that
\begin{eqnarray}
& & (2\pi )^{\frac{3}{2}} n^{-1} m^{-n-\frac{1}{2} } \sup_{x\in V/\Gamma_V} \rho_{V/\Gamma_V ,\omega_V ,m} (x) \nonumber \\
& \geq & (2\pi )^{\frac{3}{2}} n^{-1} m^{-n-\frac{1}{2} }  b_m (a,mq_{m}^{-1} ) \\
& \geq & \left( 1-Cm^{-1} \right) \left| \Gamma_V \right| \sum_{j=1}^{\infty } j^{-n} \rho_{D,\omega_D ,\Delta_{\Gamma_V ,m} ,j} (a) \left( f(jq_a^{-1} -1) \right)^m \nonumber \\
& \geq & \left( 1-Cm^{-1} \right) \left| \Gamma_V \right| \alpha_{D,m} (a) , \nonumber
\end{eqnarray}
where $C >0$ is a constant independent of $m$. For abbreviation, we use the same letter $C$ for constants independent of $m$. By Lemma \ref{lmmapproximationsumfm}, to prove (\ref{inequalitycorobgmkernelsup}), we only need to prove that
\begin{eqnarray*}
 \sup_{t>0} \left( t\sum_{q=1}^\infty q^{1-n} \rho_{D,\omega_D ,\Delta_{\Gamma_V ,m} ,q } (a) e^{1-qt} G_{m-1 ,3} (qt-1) \right) \leq \alpha_{D,m} (a) +  \frac{C}{\sqrt{m}}  .
\end{eqnarray*}
Note that the standard Tian-Yau-Zelditch expansion implies that $\rho_{D,\omega_D ,\Delta_{\Gamma_V ,m} ,q } (a) \leq C' q^{n-1} $ for some constant $C' >0$ independent of $m$ and $q$. An easy computation now shows that
\begin{eqnarray*}
& & \sup_{t>0} \left( t\sum_{\lfloor t^{-1} \rfloor -1 \leq q\leq \lfloor t^{-1} \rfloor +1 } q^{1-n} \rho_{D,\omega_D ,\Delta_{\Gamma_V ,m} ,q } (a) e^{1-qt} G_{m-1 ,3} (qt-1) \right) \\
& \leq & \alpha_{D,m} + \sup_{t>0} \left( C e^{-\frac{(m-1) t^2 }{10}} +Cm t^3 +Cm t^3 e^{-\frac{(m-1) t^2 }{10}} \right) \leq \alpha_{D,m} +  \frac{C}{\sqrt{m}}  .
\end{eqnarray*}
Then we estimate the remaining terms as follows.
\begin{eqnarray*}
& & \sup_{t>0} \left( t\sum_{ \left| q -\lfloor t^{-1} \rfloor \right| \geq 2 } q^{1-n} \rho_{D,\omega_D ,\Delta_{\Gamma_V ,m} ,q } (a) e^{1-qt} G_{m-1 ,3} (qt-1) \right) \\
&\leq & \sup_{t>0} \left( C t\sum_{ q=1 }^\infty \left( 1+m |qt|^3 \right) e^{-\frac{(m-1)|qt|^2}{2}} \right) \leq C\int_{0}^{\infty } t \left( 1+m |t\xi |^3 \right) e^{-\frac{(m-1)|t\xi |^2}{2}}  d\xi \\
& = & C\int_{0}^{\infty } \left( 1+m y^3 \right) e^{-\frac{(m-1)\zeta^2 }{2}}  d\zeta + C\sup_{\zeta >0} \left( m\zeta^3 e^{-\frac{(m-1)\zeta^2 }{2}} \right) \leq \frac{C}{\sqrt{m}} ,
\end{eqnarray*}
and (\ref{inequalitycorobgmkernelsup}) is proved. 

Since $ \Theta_{\Gamma_V } \left( g \right)^{\left| \Gamma_V \right| } = 1 $, $\forall g\in\Gamma_V $, we see that $\rho_{D,\omega_D ,\Delta_{\Gamma_V ,m} ,q } = \rho_{D,\omega_D ,\Delta_{\Gamma_V ,m +\left| \Gamma_V \right| } ,q } $, $\forall m\in\mathbb{N} $. It follows that $\alpha_{D,m} =\alpha_{D,m+\left| \Gamma_V \right| } $, $ \forall m\in\mathbb{N} $, which completes the proof.
\end{proof}

\begin{rmk}
For any integer $k\geq 0$ and $m\in\mathbb{N} $, let $ \rho^{(k)}_{ V/\Gamma_V , \omega_V ,m} (x) $ be the $m$-th Bergman kernel function on $\left( V/\Gamma_V ,\omega_V ,\linebundle_V /\Gamma_V , \left| \log h_D \right|^k h_V \right)$.

As in the proof of Corollary \ref{convergencecorobgmkernelsup}, the series expansion (\ref{expresspropbergmanknvloghd}) shows that there exists a constant $C=C\left(k, N ,V ,\left| \Gamma_V \right| \right) >0$, such that
\begin{eqnarray}
\label{inequalitycorobgmkernelsup2}
\left| (2\pi )^{\frac{3}{2}} n^{-1} \left(m k+m \right)^{-n-\frac{1}{2} } \sup_{x\in V/\Gamma_V} \rho^{(k)}_{V/\Gamma_V ,\omega_V ,m} (x) - \left| \Gamma_V \right| \alpha_{D,m} \right| \leq \frac{C}{\sqrt{m}}  ,
\end{eqnarray}
where $\alpha_{D,m} = \sup_{a\in D} \sup_{q\in\mathbb{N} } q^{-n} \rho_{D,\omega_D ,\Delta_{\Gamma_V ,m} ,q} (a) $.
\end{rmk}

\section{Localization of Bergman kernels}
\label{tpsonprs}

In this section, we establish some basic estimates for peak sections, and then use it to estimate the distance of global Bergman kernel and the Bergman kernel on asymptotic complex hyperbolic cusp. We will also show that this localization approach is also valid for the Poincar\'e type metrics.

\subsection{Localization of peak sections}
\hfill

Let $(M,\omega )$ be a K\"ahler manifold, $\linebundle $ be a holomorphic line bundle on $M$ equipped with a hermitian metric $h$ whose curvature form is $\omega $. For any closed subspace $X\subset \hl \left( M,\linebundle \right) $ and complex bounded linear functional $0\neq T\in X^* $, we say that a section $S\in X $ is the peak section of $T$, if $ \int_M \left\Vert S \right\Vert^2 d\V_\omega =1 $, and $ T(S)  = \left\Vert T \right\Vert $. By the Riesz representation theorem, $S$ is the peak section of $T$ if and only if $T(S) \in\mathbb{R} $, and $S\perp \ker T$ in $X\subset \hl \left( M,\linebundle \right)$. Note that $\hl \left( M,\linebundle \right) $ is a Hilbert space, and so is $X$. It follows that there exists a unique peak section of $T$ in $X$.

Let $U$ be an open subset of $M$, $X\subset \hl \left( U,\linebundle \right) $ be a closed subspace, and $0\neq T\in X^* $. Then the restriction map gives a bounded linear embedding $ \hl \left( M,\linebundle \right) \to \hl \left( U,\linebundle \right) $. When $X\cap \hl \left( M,\linebundle \right) \nsubseteq \ker T$, we see that $T\neq 0$ on $X\cap \hl \left( M,\linebundle \right) $, and hence we can find a unique peak section of $T\big|_{X\cap \hl \left( M,\linebundle \right) }$ in $ X\cap \hl \left( M,\linebundle \right) $. The peak section of $T\big|_{X\cap \hl \left( M,\linebundle \right) }$ is called the peak section on $M$. The following lemma gives an estimate of the $L^2$ distance between the peak sections of $T\in X^* $ on $U$ and on $M$.

\begin{lmm}
\label{peaksectionrestriction}
Let $S_{U} \in X $ and $S_M \in X\cap \hl \left( M,\linebundle \right) $ be the peak sections on $U$ and $M$ respectively. Assume that $\left\Vert T\big|_{X\cap \hl \left(  M,\linebundle \right)} \right\Vert  \geq \left( 1-\epsilon \right) \left\Vert T \right\Vert >0 $ for some $\epsilon >0 $. Then
$$ \int_{M\sq U} \left\Vert S_{M} \right\Vert^2 d\V_\omega + \int_{U} \left\Vert S_U - S_{M} \right\Vert^2 d\V_\omega \leq 2\epsilon .$$
\end{lmm}

\begin{proof}
By definition, we have $T\left( S_M \right) = \left\Vert T\big|_{X\cap  \hl \left( M,\linebundle \right) } \right\Vert \leq \left\Vert T \right\Vert = T\left( S_U \right)$. Hence there exists a constant $\delta \in \left( 0,\epsilon \right] $ such that $T\left( s_M \right) = \left( 1-\delta \right) T\left( s_U \right) $. It follows that $s_M - \left( 1-\delta \right) s_U \in \ker T$. Then we have
$$ \int_{U} \left\langle S_{U} , S_M - \left( 1-\delta \right) S_U \right\rangle d\V_\omega =0 .$$
Combining these, we can conclude that
\begin{eqnarray*}
\int_{U} \left\Vert S_U - S_{M} \right\Vert^2 d\V_\omega & = & \int_U \left\Vert S_{U} \right\Vert^2 +\left\Vert S_{M} \right\Vert^2 -2\re \left\langle S_{U} , S_M \right\rangle  d\V_\omega \\
& = & \int_U \left\Vert S_{U} \right\Vert^2 +\left\Vert S_{M} \right\Vert^2 -2\left( 1-\delta \right) \left\Vert S_{U} \right\Vert^2  d\V_\omega \\
& = & 2\delta - \int_{M\sq U} \left\Vert S_{M} \right\Vert^2 d\V_\omega \leq 2\epsilon -\int_{M\sq U} \left\Vert S_{M} \right\Vert^2 d\V_\omega ,
\end{eqnarray*}
and the lemma follows.
\end{proof}

Now we show how to use the peak section to give some estimates about Bergman kernels. Taking a vector $e_p \in \linebundle^m_{p} $ at $p\in U$ such that $\left\Vert e_p \right\Vert_{h^m} =1 $. Then $T_p (S)=e_p^{*} \left( S \right) $ gives a bounded linear functional of $\hl \left( U,\linebundle^m \right)  $. If $T_p \neq 0$, then the peak section $S_U \in \hl \left( U,\linebundle^m \right)  $ satisfies that $B_U (p,x) = S_U (p) \otimes S_U (x)^* $, $\forall x\in U $, where $B_U $ is the Bergman kernel on $U$. So the Bergman kernel can be represented by a peak section. By combining Lemma \ref{peaksec}, Theorem \ref{thmregularpart} and Lemma \ref{peaksectionrestriction}, we can give a proof of the following Agmon type estimate.

\begin{coro}
\label{coroagmontypeestimate}
Given constants $n,\delta , \Lambda ,\epsilon , Q >0$ and $k\geq 0$, there exists a constant $ C$ which depending only on $n, \delta , \Lambda , \epsilon , k,Q$ with the following property.

Let $(M,\omega )$ be a K\"ahler manifold, $\linebundle $ be a holomorphic line bundle on $M$ equipped with a hermitian metric $h$ such that $\Ric\left( \omega \right) \geq -\Lambda \omega $, $\Ric (h) \geq \epsilon \omega $, and $x\in M$. If $(M, \omega)$ isn't complete, we suppose that $M$ is pseudoconvex. 

Fix $m\geq C $. Write $U_m = B_{\frac{3\delta \log m}{\sqrt{m}} } \left( x_0 \right) $. Assume that $\bar{U}_m $ is compact, $\Ric (h) = \omega $ on $U_m $, $\inf_{y\in U_m } \inj (y)\geq \frac{\delta \log m}{\sqrt{m}} $, and $\sum_{j=0}^{2k} \sup_{U_m} \left\Vert \nabla^j \Ric (\omega ) \right\Vert \leq Q .$ Then we have
$$ \left\Vert B_{M,\omega ,m} (x,y) \right\Vert_{C^l ; h^m ,\omega } \leq C m^{-k+n+l} , $$
for any $x,y\in B_{\frac{2\delta \log m}{\sqrt{m}} } \left( x_0 \right) $ satisfying that $ \dist_{\omega } (x,y) \geq \frac{\delta\log m}{\sqrt{m}} $, where $B_{M,\omega ,m} (x,y)$ is the $m$-th Bergman kernel on $\left( M,\omega ,\linebundle ,h \right) $. 
\end{coro}

\begin{proof}
Without loss of generality, we can assume that $\delta =1$. For each $m\geq 2$, we denote $ \frac{ \log (m)}{\sqrt{m} } $ briefly by $\varepsilon_m$. By the theory of Cheeger-Gromov convergence, we can assume that for any point $y\in U_m $, there exists a biholomorphic map $\left( z_1 ,\cdots ,z_n \right) : B_{\mu \varepsilon_m } (y) \to W_y\subset \mathbb{C}^n$ around $y$ such that $z_1 (y)=\cdots =z_n (y)=0 $, $ e^{-Q'} \omega_{Euc} \leq \omega \leq e^{Q' } \omega_{Euc} $, and the hermitian matrix $\left( g_{i\bar{j}} \right) $ satisfies that $g_{i\bar{j}} (0) = \delta_{ij} $, $dg_{i\bar{j}} (0) = 0$, and $\left\Vert g_{i\bar{j}} \right\Vert^{*}_{2k+2,\frac{1}{2}; W_y } \leq Q' $, where $\mu \in (0,1) $, $Q' \geq 10 $ are constants depending only on $n, k,Q$, and $\left\Vert f \right\Vert^{*}_{k,\alpha } $ is the interior norms on a domain in $\mathbb{R}^{2n}$. We further assume that there exists a holomorphic frame $e_y$ of $\linebundle $ on this coordinate such that the local representation function of $h$, $a=\log h\left( e_y ,e_y \right) $, satisfying that $a(0)=0$, $\frac{\partial^{|I|} a}{\partial z^I }  (0) =0$ for each milti-index $I$ with $|I|\leq 3$, and $\left\Vert a \right\Vert^{* }_{2k+4,\frac{1}{2} ;U_y }  \leq Q' \varepsilon_m^2 $.

For notational convenience the same letter $C\geq 10 $ will be used to denote large constants depending only on $n,\Lambda ,k, \epsilon , Q$. 

Let $\left\lbrace S_j \right\rbrace_{j\in J} $ be an orthonormal basis in $\hl \left( M,\linebundle^m \right) $, and $x,y $ be points in $B_{2\varepsilon_m } \left( x_0 \right) $ satisfying that $ \dist_{\omega } (x,y) \geq \varepsilon_m $. Assume that $S_j =f_{j,x} e^m_x $ on $B_{\mu \varepsilon_m} (x) $, and $S_j = f_{j,y} e^m_y $ on $B_{\mu \varepsilon_m} (y) $. Then we have $B_{M,\omega ,m} = a^m_y \left( \sum_{j} f_{j,x} \bar{f}_{j,y} \right) e^m_x \otimes e_y^{-m} $ on $B_{\mu \varepsilon_m} (x) \times B_{\mu \varepsilon_m} (y) $. 

Since $\left\Vert g_{i\bar{j}} \right\Vert^{*}_{2k+2,\frac{1}{2}; W_x } +\left\Vert g_{i\bar{j}} \right\Vert^{*}_{2k+2,\frac{1}{2}; W_y } \leq 2 Q' $ and $\left\Vert a \right\Vert^{* }_{2k+4,\frac{1}{2} ;W_x } + \left\Vert a \right\Vert^{* }_{2k+4,\frac{1}{2} ;W_y } \leq  \frac{2 Q' \left( \log m \right)^2 }{m}$, it is sufficient to show that 
$$ \left| \sum_{j} \frac{\partial^{\left| P_1 \right| } f_{j,x} }{\partial z_x^{P_1}}  \frac{\partial^{\left| P_2 \right| } \bar{f}_{j,y}  }{\partial \zbar_y^{P_2}} \right| \leq Cm^{-k+n+\left| P_1 \right| + \left| P_2 \right| } ,\;\; \forall \left| P_1 \right| + \left| P_2 \right| \leq k ,$$
where $P_i =\left( p_{i,1} ,p_{i,2} ,...,p_{i,n} \right)\in \mathbb{Z}_{+}^{n}$ are $n$-tuples of integers.

For each $S\in H^0 \left( W_y ,\linebundle^m \right) $, $f= e_y^{-m} (S) $ is a holomorphic function on $W_y $, and hence the map $T_{P,y} (S) = \frac{\partial^{|P|} f}{\partial z^P} (y) $ gives a linear functional $T_{P,y} $ on $H^0 \left( W_y ,\linebundle^m \right) $, where $P=\left( p_1 ,p_2 ,...,p_n \right)\in \mathbb{Z}_{+}^{n}$ is an $n$-tuple of integers. Then Lemma \ref{peaksec} shows that $T_{P,y} \big|_{H_{L^2}^0 \left( M ,\linebundle^m \right)} \neq 0 $ for sufficiently large $m$. For any $y\in M$ and open neighborhood $U$ of $y$, let $S_{P,y,U}$ denotes the peak section of $T_{P,y} \big|_{H_{L^2}^0 \left( U ,\linebundle^m \right)} \neq 0 $. Then we only need to show that
\begin{eqnarray}
\left| T_{P_1 ,x,M} \left( S_{P_1,x,M} \right) \cdot T_{P_2 ,y,M} \left( S_{P_1,x,M} \right) \right| \leq Cm^{-k+n+\left| P_1 \right| + \left| P_2 \right| } ,\;\; \forall \left| P_1 \right| + \left| P_2 \right| \leq k . \label{goalcoroagmontype}
\end{eqnarray}

Combining Theorem \ref{thmregularpart} with Lemma \ref{peaksectionrestriction}, we can conclude that
$$ \int_{M \sq B_{ \mu e^{-10 Q'} \varepsilon_m } \left(y_1 \right) } \left\Vert S_{0,y_1 ,M} \right\Vert^2 d\V_\omega \leq Cm^{-k} ,\;\; \forall m\geq C, \;\; \forall y_1 \in B_{\frac{11 \varepsilon_m }{4}} \left( x_0 \right) .$$
Hence we see that
\begin{eqnarray}
\left| T_{P_1 ,y_1 } \left( S_{0,y_2 ,M } \right) \right| & \leq & \left( \int_{M \sq B_{\frac{\varepsilon_m}{100} } \left( y_2 \right) } \left\Vert S_{0,y_2 ,M } \right\Vert^2 d\V_\omega \right)^{\frac{1}{2}} \left| T_{P_1 ,y_1 } \left( S_{P_1 ,y_1 ,M\sq B_{\frac{\varepsilon_m}{100} } \left( y_2 \right) } \right) \right| \label{estimate1coroagmontype} \\
& \leq & Cm^{-\frac{k}{2}} \left| T_{P_1 ,y_1 } \left( S_{P_1 ,y_1 ,  B_{\frac{\varepsilon_m}{100} } \left( y_1 \right) } \right) \right| , \nonumber
\end{eqnarray}
for any $m\geq C$, and $ y_1 ,y_2 \in B_{\frac{5\varepsilon_m }{2}} \left( x_0 \right) $ satisfying that $ \dist_{\omega } \left( y_1 ,y_2 \right) \geq \frac {\varepsilon_m }{40} $.

By Lemma \ref{peaksec}, we can find a holomorphic section 
$$ S'_{P_1 ,y_1 , B_{\frac{\varepsilon_m}{100} } \left( y_1 \right) } \in \hl \left( B_{\frac{\varepsilon_m}{100} } \left( y_1 \right) ,\omega ,\linebundle^m ,h^m \right) ,$$
such that $ \left\Vert S'_{P_1 ,y_1 , B_{\frac{\varepsilon_m}{100} } \left( y_1 \right) } \right\Vert_{L^2 ; h^m ,\omega} =1 $, $ 0 < T_{P_1 ,y_1 } \left( S'_{P_1 ,y_1 , B_{\frac{\varepsilon_m}{100} } \left( y_1 \right) } \right) \leq C m^{ \frac{n+\left| P_1 \right| }{2} } $, and Corollary \ref{coropeaksec} now implies that
$$ \left| \int_{B_{ \frac {\varepsilon_m }{100}} \left( y_1 \right) } \left\langle S'_{P_1 ,y_1 , B_{\frac{\varepsilon_m}{100} } \left( y_1 \right) } ,S'' \right\rangle_{h^m } d\V_{\omega} \right| \leq C m^{-\frac{1}{2}} \left\Vert S'' \right\Vert_{L^2 ; h^m ,\omega} , $$
for any $m\geq C$, and $ S'' \in \ker T_{P_1 ,y_1} \cap \hl \left( B_{\frac{\varepsilon_m}{100} } \left( y_1 \right) ,\omega ,\linebundle^m ,h^m \right) $. Let 
$$ \vartheta_{P_1 ,y_1 ,\frac{\varepsilon_m}{100} } = \int_{B_{ \frac {\varepsilon_m }{100}} \left( y_1 \right) } \left\langle S'_{P_1 ,y_1 , B_{\frac{\varepsilon_m}{100} } \left( y_1 \right) } ,S_{P_1 ,y_1 , B_{\frac{\varepsilon_m}{100} } \left( y_1 \right) } \right\rangle_{h^m } d\V_{\omega} .$$
Then we have $ S'_{P_1 ,y_1 , B_{\frac{\varepsilon_m}{100} } \left( y_1 \right) } - \vartheta_{P_1 ,y_1 ,\frac{\varepsilon_m}{100} } S_{P_1 ,y_1 , B_{\frac{\varepsilon_m}{100} } \left( y_1 \right) }  \in \ker T_{P_1 ,y_1}$, and we have
\begin{eqnarray}
1-  \vartheta_{P_1 ,y_1 ,\frac{\varepsilon_m}{100} }^2 & = & \left\Vert S'_{P_1 ,y_1 , B_{\frac{\varepsilon_m}{100} } \left( y_1 \right) } - \vartheta_{P_1 ,y_1 ,\frac{\varepsilon_m}{100} } S_{P_1 ,y_1 , B_{\frac{\varepsilon_m}{100} } \left( y_1 \right) } \right\Vert^2_{L^2 ;h^m ,\omega } \label{estimate2coroagmontype} \\
& = & \left\langle S'_{P_1 ,y_1 , B_{\frac{\varepsilon_m}{100} } \left( y_1 \right) } , S'_{P_1 ,y_1 , B_{\frac{\varepsilon_m}{100} } \left( y_1 \right) } - \vartheta_{P_1 ,y_1 ,\frac{\varepsilon_m}{100} } S_{P_1 ,y_1 , B_{\frac{\varepsilon_m}{100} } \left( y_1 \right) } \right\rangle_{L^2 ;h^m ,\omega } \nonumber \\
& \leq & Cm^{-\frac{1}{2}} , \;\; \forall m\geq C . \nonumber
\end{eqnarray}
We thus get $ \vartheta_{P_1 ,y_1 ,\frac{\varepsilon_m}{100} } \geq 1-Cm^{-\frac{1}{2}} $, $ \forall m\geq C $. Now (\ref{estimate1coroagmontype}) becomes
\begin{eqnarray}
\left| T_{P_1 ,y_1 } \left( S_{0,y_2 ,M } \right) \right| & \leq & Cm^{-\frac{k}{2}} \left| T_{P_1 ,y_1 } \left( S_{P_1 ,y_1 ,  B_{\frac{\varepsilon_m}{100} } \left( y_1 \right) } \right) \right| \label{estimate3coroagmontype} \\
& = & C m^{-\frac{k}{2}} \vartheta_{P_1 ,y_1 ,\frac{\varepsilon_m}{100} }^{-1} \left| T_{P_1 ,y_1 } \left( S'_{P_1 ,y_1 ,  B_{\frac{\varepsilon_m}{100} } \left( y_1 \right) } \right) \right| \leq C^2  m^{-\frac{k}{2} + \frac{n+\left| P_1 \right| }{2} }, \nonumber
\end{eqnarray}
for any $m\geq C$, and $ y_1 ,y_2 \in B_{\frac{5\varepsilon_m }{2}} \left( x_0 \right) $ satisfying that $ \dist_{\omega } \left( y_1 ,y_2 \right) \geq \frac {\varepsilon_m }{40} $.

It is easy to check that
\begin{eqnarray}
 T_{P_1 ,y_1 } \left( S_{0,y_2 ,M} \right) \overline{T_{0 ,y_2 } \left( S_{0,y_2 ,M} \right) } & = &  \sum_{j} T_{P_1 ,y_1 } \left( S_{j} \right) \overline{T_{0 ,y_2 } \left( S_{j} \right) } \label{estimate4coroagmontype} \\
 & = & T_{P_1 ,y_1 } \left( S_{P_1 ,y_1 ,M} \right) \overline{T_{0 ,y_2 } \left( S_{P_1 ,y_1 ,M} \right) } . \nonumber
\end{eqnarray}
Let $y_1 =x $. Then we have
\begin{eqnarray*}
\left\Vert S_{P_1 ,x ,M} \left( y_2 \right) \right\Vert_{h^m } & = & \left| T_{0 ,y_2 } \left( S_{P_1 ,x ,M} \right) \right| \\
& = & \left| T_{P_1 ,x } \left( S_{0 ,y_2 ,M} \right) \right| \cdot \left| T_{0 ,y_2 } \left( S_{0 ,y_2 ,M} \right) \right| \cdot \left| T_{P_1 ,x } \left( S_{P_1 ,x ,M} \right) \right|^{-1} \\
& \leq & C m^{-\frac{k-n- \left| P_1 \right| }{2}} \cdot m^{\frac{n}{2}} \cdot m^{-\frac{n+\left|P_1 \right| }{2}} \\
& = & C m^{\frac{-k+n}{2}} ,\;\; \forall m\geq C ,\;\; \forall y_2 \in B_{\frac{5\varepsilon_m }{2}} \left( x_0 \right) \big\sq B_{\frac{\varepsilon_m}{40}} \left( x \right) .
\end{eqnarray*}
It follows that
\begin{eqnarray}
& & \left| T_{P_2 ,y } \left( S_{P_1 , x ,M } \right) \right| \label{estimate5coroagmontype} \\
& \leq & \left( \int_{B_{2\varepsilon_m} \left( x \right) \sq B_{\frac{\varepsilon_m}{40} } \left( x \right) } \left\Vert S_{P_1 ,x ,M}   \right\Vert^2 d\V_\omega \right)^{\frac{1}{2}} \left| T_{P_2 ,y } \left( S_{P_2 ,y ,M\sq B_{\frac{\varepsilon_m}{100} } \left( x \right) } \right) \right| \nonumber \\
& \leq & Cr^n_m m^{-\frac{k}{2}}  \left| T_{P_2 ,y } \left( S_{P_2 ,y , B_{\frac{\varepsilon_m}{100} } \left( y \right) } \right) \right| \nonumber \\
& \leq & C^2 m^{\frac{-k+\left| P_2 \right| }{2}} (\log m)^n  , \nonumber
\end{eqnarray}
and (\ref{goalcoroagmontype}) is proved.
\end{proof}

\begin{rmk}
It is also shown in the above argument that for each given constant $\mu \in (0,1) $, the peak section, $S_{P_1 ,x ,M} $, satisfying that
\begin{equation}
\label{estimatermkbeloecoroagmontype}
\int_{B_\frac{2 \delta \log m}{ \sqrt{m}} (x) \big\sq B_\frac{\mu \delta \log m}{\sqrt{m}} (x)  } \left\Vert S_{P_1 ,x ,M} \right\Vert^2 d\V_{\omega } \leq Cm^{-\frac{k}{2} } (\log m)^n ,\;\; \forall m\geq C,
\end{equation}
where $ C$ is a constant depending only on $n, \Lambda , \epsilon , \delta , \mu , k,Q$.
\end{rmk}

\subsection{Approximate the peak section by approximating the metric }
\label{peaksecperturbedmanifoldssubsection}
\hfill

Let $M$ be a K\"ahler manifold, $(\linebundle ,h) $ be a Hermitian line bundle on $M$, let $\omega $ and $\omega' $ be K\"ahler metrics on $M$, and let $u$ be a smooth real-valued function on $M$. Suppose that there are constants $\delta ,\epsilon >0 $ such that $ e^{-\epsilon } \omega^n \leq \omega'^n \leq  e^\epsilon \omega^n $ and $\left| u \right| \leq \delta $. Write $ h' =e^{-u} h $. Clearly, we see that $ \hl \left( M,\omega ,\linebundle ,h \right) = \hl \left( M, \omega' ,\linebundle ,h' \right) $ as linear subspaces of $H^0 \left( M,\linebundle \right) $, and the $L^2 $-norms are equivalent. So we can use the same notation $\hl \left( M,\linebundle \right)$ for them if we don't emphasize the norms. Let $X $ be a closed subspace of $ \hl \left( M,\linebundle \right) $. Fix a bounded linear functional $0\neq T : \hl \left( M,\linebundle \right) \to \mathbb{C} $.

Let $S_T $ be the peak section of $T$ in $X\subset \hl \left( M,\omega ,\linebundle ,h \right) $, and $S'_T $ be the peak section of $T$ in $ X\subset \hl \left( M, \omega' ,\linebundle ,h' \right) $. Then we can estimate the $L^2$-norm of $ S_T -S'_T $ as follows.

\begin{lmm}
\label{lmmdistancedifferencepeaksection}
Under the above assumptions, we have
\begin{eqnarray}
\left\Vert S_T -S'_T \right\Vert^2_{L^2 ;h,\omega } \leq e^{ \delta +\epsilon  } - e^{ - \delta -\epsilon  } ,
\end{eqnarray}
and
\begin{eqnarray}
e^{-\frac{\delta +\epsilon }{2}} T \left( S'_T \right) \leq T \left( S_T \right) \leq e^{\frac{\delta +\epsilon }{2} } T \left( S'_T \right) .
\end{eqnarray}
\end{lmm}

\begin{proof}
By definition, we can obtain
\begin{eqnarray}
\left\Vert S_T \right\Vert^2_{L^2 ; h' ,\omega' } =  \frac{1}{n!} \int_M \left\Vert S_T \right\Vert^2_{ h' }  \omega^n \leq \frac{ e^{\delta +\epsilon } }{n!} \int_M \left\Vert S_T \right\Vert^2_{ h }  \omega^n = e^{\delta +\epsilon } .
\end{eqnarray}
Hence we have
\begin{eqnarray}
\label{estimate2lmmdistancedifferencepeaksection}
T \left( S'_T \right) \geq T \left( \left\Vert S_T \right\Vert^{-1}_{L^2 ;  h' , \omega' } S_T \right) = \left\Vert S_T \right\Vert^{-1}_{L^2 ; h' , \omega' } T \left( S_T \right) \geq e^{-\frac{\delta +\epsilon }{2}} T \left(S_T \right) .
\end{eqnarray}
Similarly, we can conclude that
\begin{eqnarray}
\left\Vert S'_T \right\Vert^2_{L^2 ; h, \omega } =  \frac{1}{n!} \int_M \left\Vert S'_T \right\Vert^2_{ h }  \omega^n \leq \frac{ e^{\delta +\epsilon } }{n! } \int_M \left\Vert S'_T \right\Vert^2_{ h' }  \omega'^n = e^{\delta +\epsilon } ,
\end{eqnarray}
and then we see that
\begin{eqnarray}
\label{estimate4lmmdistancedifferencepeaksection}
 T \left( S_T \right) \geq T \left( \left\Vert S'_T \right\Vert^{-1}_{L^2 ; h,\omega } S'_T \right) = \left\Vert S'_T \right\Vert^{-1}_{L^2 ;h,\omega } T \left( S'_T \right) \geq e^{-\frac{\delta + \epsilon }{2}} T \left( S'_T \right) .
\end{eqnarray}
Combining (\ref{estimate2lmmdistancedifferencepeaksection}) with (\ref{estimate4lmmdistancedifferencepeaksection}), we get
\begin{eqnarray}
 e^{-\frac{\delta +\epsilon }{2}} T \left( S'_T \right) \leq T \left( S_T \right) \leq e^{\frac{\delta +\epsilon }{2} } T \left( S'_T \right) .
\end{eqnarray}

Let $\lambda \in \left[ e^{-\frac{\delta +\epsilon }{2}} , e^{\frac{\delta +\epsilon }{2}} \right]  $ be the unique constant such that $T \left( S'_T \right) =\lambda T \left( S_T \right) $. Then we have $ S'_T -\lambda S_T \in ker \mathrm{T} \cap X $. It follows that 
\begin{eqnarray}
\int_M \left\langle S'_T , S_T \right\rangle_{h} d\V_{\omega } -\lambda = \int_M \left\langle S'_T -\lambda S_T , S_T \right\rangle_{h} d\V_{\omega } =0 ,
\end{eqnarray}
and hence
\begin{eqnarray}
\left\Vert S_T -S'_T \right\Vert^2_{L^2 ;h,\omega } & = & \int_M \left\langle S'_T -\lambda S_T , S'_T -\lambda S_T \right\rangle_{h} d\V_{\omega } \nonumber \\
& = & \left\Vert S'_T \right\Vert^2_{L^2 ;h,\omega } -2\lambda \int_M \left\langle S'_T , S_T \right\rangle_{h} d\V_{\omega } +\lambda^2 \\
& = & \left\Vert S'_T \right\Vert^2_{L^2 ;h,\omega } - \lambda^2 \leq e^{ \delta +\epsilon  } - e^{ - \delta -\epsilon  } , \nonumber
\end{eqnarray}
and the proof is complete.
\end{proof}

\subsection{Peak sections on asymptotic complex hyperbolic cusps}
\hfill

We will prove Theorem \ref{thmcusplocalizationprinciple} in this part. 

In this subsection, we let $(M,\omega )$ be a K\"ahler manifold with $\Ric (\omega ) \geq -\Lambda \omega $ and let $( \linebundle ,h )$ be a Hermitian line bundle with $\Ric (h) \geq \epsilon \omega $, where $\epsilon ,\Lambda >0$ are constants. Suppose that $M$ is pseudoconvex if $(M,\omega )$ is not complete. Assume that there exists an asymptotic complex hyperbolic cusp $ \left( U,\omega ,\linebundle ,h \right) \cong \left( V_r /\Gamma_V ,\omega_V -\sqrt{-1}\partial\partialbar u ,\linebundle_V / \Gamma_V , e^u h_V \right) $ on $(M,\omega )$ such that $\bar{U}$ is complete, where $r\in (0,1)$ is a positive constant, and $u=o\left( 1 \right)$ as $h_D \to 0^+ $ to all orders with respect to $\omega_V$. Note that $ \left( \linebundle_V ,h_V \right) \cong \left( \mathbb{C} , \left| \log h_D \right| \right) $ on $V$. 

It will cause no confusion if we use the same letter to designate a subset of $V_r /\Gamma_V $ and the open subset in $U$ corresponding to it. Assume that $\bar{U} $ is complete in $(M,\omega )$. 

Since $u=o\left( 1 \right)$ as $h_D \to 0^+ $ to all orders with respect to $\omega_V$, we can estimate the lower bounds of injectivity radius on $U$. For any subset $W\subset M$, we write $\inj (W) =\inf_{x\in W} \inj_M (x) $.

\begin{lmm}
\label{lmminjectivityradiushyperbpliccusp}
Let $r$ be the constant defined in above, and $k\in\mathbb{N}$ be a constant. Then there exists a constant $C=C (k,M,\omega )>0 $ such that the curvature $ \sum_{j=0}^{5k+10n} \left\Vert \nabla^j \Ric (\omega ) \right\Vert_{\omega } \leq C $ on $U$, and for each $t\in (0,r) $, we have $ \inj \left( U\sq \pi_V \left( V_t \right) \right) \geq C^{-1} \left| \log t \right|^{-1} $, where $\pi_V : V\to V/\Gamma_V $ is the quotient map.
\end{lmm}

\begin{proof}
Since the holomorphic sectional curvature of $\omega_V$ is a constant, we can conclude that the sectional curvature of $\omega_V$ is bounded. Then $u=o\left( 1 \right)$ implies that the sectional curvature of $\omega$ is bounded on $U$. Similarly, for each $j\in\mathbb{Z}_{\geq 0} $, $\left\Vert \nabla^j \Ric (\omega ) \right\Vert_\omega $ is bounded on $U$. As in the proof of Lemma \ref{lmminjectivityradiusoncusp}, Cheeger's lemma now shows that there exists a constant $ \zeta >0$ depending only on $(M,\omega ,L,h)$, such that $\inj (x) \geq \zeta\left| \log t \right|^{-1} $ on $U\sq \pi_V \left( V_t \right) $, for any $t\in (0,r) $. Note that $\bar{U} $ is complete.
\end{proof}

Now we can construct suitable quasi-plurisubharmonic functions for using $L^2$ estimate.

\begin{lmm}
\label{lmmquasipshhyperboliccuspver2}
Let $r$ be the constant defined in above, and $\beta >0 $, $\delta \in \left( 0, e^{-1} \right) $ be constants. Suppose that $ \delta <\min\{ e^{-1} ,r \} $. Then for any point $y\in V_{ \delta^{1+2\beta } } $, there exists a quasi-plurisubharmonic function $\psi \in L^1 (M,\omega ) $ such that $\psi \leq 0 $, $ \lim_{x\to y} \left( \psi (x) - \log \left( \dist_{\omega} (x,y) \right) \right) = -\infty $, $\sqrt{-1} \partial\partialbar \psi \geq -C\left| \log \delta \right| \omega $, and $\mathrm{supp} \psi \subset V_\delta /\Gamma_V $, where $C>0$ is a constant dependent only on $(M,\omega ) $, and $\beta $.
\end{lmm}

\begin{proof}
We assume that $\Gamma_V =0 $ at first. By Theorem \ref{thmregularpart}, for sufficiently large integer $k $, we have $ (2\pi )^{n-2} \leq k^{1-n} \rho_{D,\omega_D ,k} \leq (2\pi )^{n } $. Choosing an $L^2$ orthonormal basis $\{ S_j \}_{j=1}^{N_k } $ of $\hl \left( D,\linebundle^{-k}_D \right) $ satisfying that $S_j \left( \pi_D (y) \right) =0 $, $ \forall j\geq 2 $. Note that we can construct holomorphic functions $ \widetilde{S}_j $ from $S_j $ as in Lemma \ref{holoontotal}. Let 
$$ f (x ) = \left| \widetilde{S}_1 \left( x \right) -\widetilde{S}_1 \left( y \right) \right|^2 + \sum_{j\geq 2} \left| \widetilde{S}_j \left( x \right)  \right|^2 ,\;\;  \forall x\in V .$$ 
Then we have $ f(x) = \rho_{D,k} \left( \pi_D (x) \right) h_D (x)^k + \left| \widetilde{S}_1 (y) \right|^2 - 2 \mathrm{Re} \left( \widetilde{S}_1 (x) \overline{\widetilde{S} }_1 (y) \right) $. Clearly, we see that $\left| \widetilde{S}_1 (y) \right|^2 \leq \delta^{k\left(1+2\beta \right) } (2\pi )^{n } $. Since $u=o(1)$, we can conclude that there exists a constant $k_0 >0$ depending only on $(M,\omega )$ and $\beta $, such that $ \left\Vert \nabla \log f (x) \right\Vert_{\omega} \leq C_1 \left| \log \delta \right| ,$ for any $ x\in V_\delta \sq V_{\delta^{1+\beta } } $ and $k\geq k_0 $, where $C_1 $ is a constant depending only on $(M,\omega )$, $k$ and $\beta $. Fix a large integer $k $.

By a straightforward calculation, we see that $ \dist_{\omega} \left( V_{\delta^{1+\beta} } ,M\sq V_\delta \right) \geq C_2^{-1} \log \left( 1+ \beta \right) $ for some constant $C_2 $ depending only on $(M,\omega )$. Then we can choose a cut-off function $\eta \in C^\infty \left( M \right) $ such that $\eta =1 $ on $V_{\delta^{1+\beta } } $, $\eta =0 $ on $M\sq V_\delta $, $0\leq \eta \leq 1$, and $ \sum_{i=0}^2 \left\Vert \nabla^i \eta \right\Vert_{\omega } \leq C_3 ,$ where $C_3 $ is a constant depending only on $(M,\omega )$. 

Let $C_f = \sup_{V_\delta} f $, and $\psi_0 (x) = \eta (x) \left( \log f(x) - \log C_f \right) $. It is easy to see that 
$$ \sup_{ V_\delta \sq V_{\delta^{1+\beta } } } \sum_{i=0}^1 \left\Vert \nabla^i \left( \log f(x) - \log C_f \right) \right\Vert_{\omega } \leq C_4 \left| \log \delta \right| ,$$
where $C_4 $ is a constant depending only on $(M,\omega )$ and $\beta $. Hence we have
\begin{eqnarray*}
\sqrt{-1} \partial\partialbar \psi_0 (x) & = & \sqrt{-1} \left( \log f(x) - \log C_f \right) \partial\partialbar \eta (x) + \sqrt{-1} \partial \log f(x) \wedge \partialbar \eta \\
& & + \sqrt{-1} \partial \eta (x) \wedge \partialbar \log f(x) + \sqrt{-1}  \eta (x)  \partial\partialbar \log f(x) \\
& & \geq -3C_3 C_4 \left| \log\delta \right| \omega ,\;\; \forall x\in V_\delta \sq V_{\delta^{1+\beta}} .
\end{eqnarray*}
Here we used $ \sqrt{-1} \partial\partialbar \log f(x) \geq 0 $ on $V$. 

Now we consider the case $ \Gamma_V \neq 0 $. Let $\psi_0 $ be the quasi-plurisubharmonic function constructed on $ \left( V_r , \pi_V^* \omega \right) $. Since $ \psi = \sum_{g\in \Gamma_V} \psi_0 \circ g $ is invariant under the action of $\Gamma_V $, we see that $\psi $ gives the function we need.
\end{proof}

\begin{rmk}
If we choose $f(x) = \sum_{j=1}^{N_k } \left| S_j \left( x \right)  \right|^2 $ in the proof of Lemma \ref{lmmquasipshhyperboliccuspver2}, then we can find a quasi-plurisubharmonic function $\psi \in L^1 (M,\omega ) $ such that $\psi \leq 0 $, $ \lim_{h_D (x) \to 0} \left( \psi (x) - \log \left( h_D (x) \right) \right) = -\infty $, $\sqrt{-1} \partial\partialbar \psi \geq -C\left| \log \delta \right| \omega $, $\mathrm{supp} \psi \subset V_\delta $, and $\psi$ is invariant under the action of $\Gamma_V $ on $V_r \cong U $, where $C>0$ is a constant dependent only on $(M,\omega ) $, and $\beta $.
\end{rmk}

Since $\left( U, \omega , \linebundle ,h \right) \cong \left( V_r /\Gamma_V , \omega_V -\sqrt{-1} \partial\partialbar u , \linebundle_V /\Gamma_V , e^u h_V \right) $ and $ \left( \linebundle_V ,h_V \right) \cong \left( \mathbb{C} , \left| \log h_D \right| \right) $ on $V$, we can conclude that any $\linebundle^m$-valued holomorphic section on $U$ can be represented by a $\Delta_{\Gamma_V ,m}$-invariant holomorphic function on $V_r $, $\forall m\in\mathbb{N} $. It follows that any finite linear combination of $\delta_x $ and its derivatives gives a bounded linear functional on $ H_{L^2}^0 \left( U, \linebundle^m \right) $, where $\delta_x $ is the Dirac function at $x$. By applying H\"ormander's $L^2 $ estimate in a similar way as in \cite[Proposition 1.1]{tg1}, we can give the following estimates for such peak sections. For abbreviation, we denote $e^{- \frac{\sqrt{m}}{\log m }}$ and $\frac{\log m}{\sqrt{m}} $ briefly by $r_m$ and $\varepsilon_m $, respectively, for any $m\geq 2$. 

\begin{prop}
\label{propcoraseestimate}
Let $\kappa $ and $ \delta $ be given positive constants. Then there are constants $C, m_0 $ that depend only on $\left( M,\omega ,\linebundle,h \right) , k , \kappa ,\delta $ and $l$, and satisfies the following property.

Let $m\geq n+1 $ be an integer satisfies that $r^{\kappa }_m < e^{-1} r $. Choosing a constant $\gamma_m \in \left[ r^{\kappa }_m , e^{-1} r \right) $, a point $x_m \in \pi_V \left( V_{\gamma_m } \right) $, and a linear functional $ T_m $ defined by finite linear combination of $\delta_{x_m} $ and its derivatives, where $\pi_V : V\to V/\Gamma_V $ is the quotient map. Assume that the order of derivatives contained in $T_m$ is at most $k$. Let $ U_{m} $ be an open neighborhood of $x_m \in M$. Suppose that $U_m \subset U $ and $\mathrm{dist}_{\omega } \left( \pi_V \left( V_{\gamma_m} \right) , M\sq U_m \right) \geq \delta \varepsilon_m $. Then
\begin{eqnarray}
\label{inequalitypropcoraseestimate}
\left\Vert T_m \big|_{\hl \left( M,\linebundle^m \right) } \right\Vert \geq \left( 1-\frac{C}{m^l } \right) \left\Vert T_m \big|_{\hl \left( U_m,\linebundle^m \right) } \right\Vert .
\end{eqnarray}
\end{prop}

\begin{proof}
Without loss of generality, we can assume that $\kappa =\delta =1 $. When $T_m \big|_{\hl \left( U_m,\linebundle^m \right) } = 0  ,$ this proposition is obvious. So we can assume that $ T_m \big|_{\hl \left( U_m,\linebundle^m \right) } \neq 0 $ from now.

The proof falls naturally into two parts by whether $x_m $ belongs to the set $\pi_V \left(V_{r^e_m} \right) $. 

\smallskip

\par {\em Part 1.} We start by prove the estimate (\ref{inequalitypropcoraseestimate}) when $x_m \in \pi_V \left( V_{\gamma_m } \right) \sq \pi_V \left( V_{r^e_m} \right) $.

By Lemma \ref{lmminjectivityradiushyperbpliccusp}, we see that $ \left\Vert \left( M,\omega ,x_m \right) \right\Vert^{holo}_{C^{ 5k+10,\frac{1}{2} } ,\varepsilon_m } \leq Q' $, where $ \left\Vert \cdot \right\Vert^{holo}_{C^{ k,\alpha } ,r } $ is the Cheeger-Gromov $ C^{k,\alpha } $-norm defined in Definition \ref{holonorm}, and $ Q' >0$ is a constant depending only on $\left( M,\omega ,\linebundle ,h \right)$, $k$ and $r $. Then we can find a holomorphic coordinate $\phi_{x_m } : \left( z_1 ,\cdots ,z_n \right) : W_{\xi ,x_m } \to B_{\xi \varepsilon_m } (0) \subset \mathbb{C}^n $ around $x_m $, such that the hermitian matrices $\left( g_{i\bar{j}} \right) $ satisfying that $g_{i\bar{j}} (0) = \delta_{ij} $ , $dg_{i\bar{j}} (0) = 0$, $ e^{-Q} \omega_{Euc} \leq \left( \phi_{x_m }^{-1} \right)^* \omega \leq e^Q  \omega_{Euc}  $, and there are holomorphic frames $e_{1,m}$ of $\linebundle_V $ on $ W_{\xi ,x_m } $ such that $a=\log h \left( e_{1,m} ,e_{1,m} \right) $ satisfying that $a(0)=0$, $\frac{\partial^{|I|} a}{\partial z^I }  (0) =0$ for each milti-index $I$ with $|I|\leq 3$, and $\left\Vert a \right\Vert^{* }_{5k+10,\frac{1}{2} ;W_{\xi ,x_m } }  \leq  Q\varepsilon_m^2 ,$ where $\xi ,Q >0$ are constants depending only on $\left( M,\omega ,\linebundle ,h \right)$, $k$ and $r $. For each $\lambda \in\left( 0,\xi \right) $, the inverse image of $B_{\lambda\varepsilon_m} (0)$ about the map $\phi_{x_m}$ will be denoted by $W_{\lambda ,x_m }$.

Let $\eta (t) \in C^\infty (\mathbb{R} ) $ be a cut-off function such that $\eta =1 $ on $(-\infty ,1]$, $\eta =0 $ on $[2,\infty )$, $-2\leq \eta'\leq 0$, $\left|\eta'' \right| \leq 8$. Set $\eta_{\lambda ,m} (z) =  \eta \left( \frac{4  |z|^2}{ \lambda^{2} \varepsilon_m^{2} } \right) $, and $\psi (z) =4(k+n) \log \left( \frac{4 |z|^2}{ \lambda^{2} \varepsilon_m^{2} } \right) \eta_{\lambda ,m} (z)  $. Then we have $\sqrt{1}\partial\partialbar \psi \geq -C_1 \lambda^{-2} \varepsilon_m^{-2} \omega_{Euc} $, where $C_1 =C_1 (n ) >0 $ is a constant. It will cause no confusion if we use the same letter to designate a function on $B_{\lambda\varepsilon_m} (0)$ and the composite of it with the map $\phi_{x_m}$.

Since $ T_m\big|_{\hl \left( U_m,\linebundle^m \right) } \neq 0 ,$ there exists a peak section $S_{T_m ,W_{\lambda ,x_m }}$ of $T_m $ on $W_{\lambda ,x_m }$. By the H\"ormander's $L^2$ estimate, we can find a smooth $\linebundle^m$-valued section $u_1\in L^2 $ on $M$ such that $\partialbar u_1 = \partialbar \left( \eta_{\lambda ,m} S_{T_m ,W_{\lambda ,x_m } } \right) =  \partialbar \eta_{\lambda ,m} \otimes S_{T_m ,W_{\lambda ,x_m } } $, and 
\begin{eqnarray}
\int_{M} \left\Vert u_1 \right\Vert_{h^m}^{2} e^{-\psi } d\V_\omega & \leq & \frac{C_2 }{\epsilon m -C_1 e^Q \lambda^{-2} \varepsilon_m^{-2} } \int_{M } \left\Vert \partialbar u_1 \right\Vert_{h^m ,\omega }^{2} e^{-\psi } d\V_\omega \nonumber \\
& \leq & C_2^2  \left( \log m \right)^{-2} \int_{\textrm{supp} \left\Vert \nabla\eta_{\lambda ,m} \right\Vert} \left\Vert S_{T_m ,W_{\lambda ,x_m } } \right\Vert_{h^m }^{2} e^{-\psi } d\V_\omega \\
& \leq & C^3_2  \left( \log m \right)^{-2} \int_{W_{\lambda ,x_m } } \left\Vert S_{T_m ,W_{\lambda ,x_m } } \right\Vert_{h^m }^{2} d\V_\omega = C^3_2  \left( \log m \right)^{-2}  , \nonumber
\end{eqnarray}
for any $ m\geq 1+e^{10 C_1 e^Q \lambda^{-2} \epsilon^{-1} } $, where $C_2 >1$ is a constant depending only on $\epsilon$, $C_1$, $\Lambda $, $\lambda$, $k$ and $r$. We assume that $m\geq 1+e^{10 C_1 e^Q \lambda^{-2} \epsilon^{-1} } $ from now. Then the integrability of $u_1 $ implies that $T_m \left( u_1 \right) =0$, and hence we have
\begin{eqnarray}
\left\Vert T_m \big|_{\hl \left( M,\linebundle^m \right) } \right\Vert & \geq & \left\Vert \eta_{\lambda ,m} S_{T_m ,W_{\lambda ,x_m }} - u_1 \right\Vert_{L^2 ; h^m ,\omega }^{-1} \left| T_m \left( \eta_{\lambda ,m} S_{T_m ,W_{\lambda ,x_m }} - u_1 \right) \right| \nonumber \\
& \geq & \left( \left\Vert \eta_{\lambda ,m} S_{T_m ,W_{\lambda ,x_m }}  \right\Vert_{L^2 ; h^m ,\omega } + \left\Vert u_1 \right\Vert_{L^2 ; h^m ,\omega } \right)^{-1} \left| T_m \left( S_{T_m ,W_{\lambda ,x_m }} \right) \right| \\
& \geq & \left( 1+ C^3_2  \left( \log m \right)^{-1} \right)^{-1} \left| T_m \left( S_{T_m ,W_{\lambda ,x_m }} \right) \right| . \nonumber
\end{eqnarray}

Analysis similar to that in the proof of Corollary \ref{coroagmontypeestimate} shows that for any $\lambda\in \left( 0,\frac{\xi}{10} \right) $,
\begin{eqnarray}
& & \left\Vert  T_m \left( S_{T_m ,W_{9\lambda ,x_m }} \right) S_{T_m ,W_{9\lambda ,x_m }} (y) \right\Vert = \left\Vert  T_m \left( S_{y,W_{9\lambda ,x_m} } \right) S_{y,W_{9\lambda ,x_m} } (y) \right\Vert \nonumber \\
& \leq & \left( \int_{W_{\lambda ,x_m}} \left\Vert  S_{y,W_{9\lambda ,x_m} }  \right\Vert^2 d\V_\omega \right)^{\frac{1}{2}} \left| T_m \left( S_{T_m,W_{\lambda ,x_m}} \right) \right| \left( \rho_{m,\omega ,W_{9\lambda ,x_m}} (y) \right)^{\frac{1}{2}} \\
& \leq & C_3 m^{ -2k - 4n-l } \left|  T_m \left( S_{T_m,W_{9\lambda ,x_m}} \right) \right| ,\;\; \forall y\in W_{8\lambda ,x_m} \sq W_{3\lambda ,x_m} , \nonumber
\end{eqnarray}
where $C_3 = C_3 \left( \lambda ,C_1 ,C_2 ,\epsilon ,\Lambda ,k,r \right) >0 $ is a constant, and $S_{y,W}$ is the peak section of $\delta_y $ on $W$. It follows that $\left\Vert S_{T_m,W_{3\lambda ,x_m}}  \right\Vert \leq C_3 m^{ -2k - 4n -l} $ on $W_{8\lambda ,x_m} \sq W_{3\lambda ,x_m}$. By applying the H\"ormander's $L^2$ estimate on $W_{8\lambda ,x_m} \sq W_{3\lambda ,x_m}$, we can conclude that
\begin{eqnarray}
\label{estimate1propcoraseestimatehyperbolic}
\left\Vert T_m \big|_{\hl \left( M,\linebundle^m \right) } \right\Vert  \geq  \left( 1- C_4  m^{-2k-2n-l} \right) \left| T_m \left( S_{T_m,W_{9\lambda ,x_m}} \right) \right| ,
\end{eqnarray}
where $C_4 = C_4 \left( \lambda ,C_1 ,C_2 ,\epsilon ,\Lambda ,k,r \right) >0 $ is a constant. Fix $\lambda = \frac{\xi}{100} $. It is easy to see that 
\begin{eqnarray}
\label{estimate2propcoraseestimatehyperbolic}
\left| T_m \left( S_{T_m,W_{9\lambda ,x_m}} \right) \right| = \left\Vert T_m \big|_{\hl \left( W_{9\lambda ,x_m} ,\linebundle^m \right) } \right\Vert \geq \left\Vert T_m \big|_{\hl \left( U_m,\linebundle^m |_{U_m} \right) } \right\Vert . 
\end{eqnarray}
Combining (\ref{estimate1propcoraseestimatehyperbolic}) with (\ref{estimate2propcoraseestimatehyperbolic}), we can conclude that the estimate (\ref{inequalitypropcoraseestimate}) holds for $x_m \in V_{\gamma_m } \sq V_{r^e_m} $ and $C= e^{ \left( 2k+2n+l \right) \left( C_4 + 10 C_1 e^Q \lambda^{-2} \epsilon^{-1} \right) } $.

\smallskip

\par {\em Part 2.} To complete the proof, we need to prove the estimate (\ref{inequalitypropcoraseestimate}) when $x_m \in \pi_V \left( V_{r^e_m} \right) $. 

By Lemma \ref{lmmquasipshhyperboliccuspver2}, we can find a quasi-plurisubharmonic function $\psi_m $ on $M$ by choosing $\delta_m = r^{2}_m $ and $\beta_m = 0.1 $. Then we have $\sqrt{-1} \partial\partialbar \psi_m \geq -C_5 \varepsilon_m^{-1} \omega $, where $C_5 >0$ is a constant depends only on $\left( M,\omega \right) $. Let $ \eta_m \in C^\infty \left( M \right) $ be a cut-off function such that $\eta_m =1$ on $ \pi_V \left( V_{\delta_m} \right) $, $\eta_m =0$ on $M\sq \pi_V \left( V_{r^{1.5}_m} \right) $, and $\sum_{i=1}^{i=2} \left\Vert \nabla^i \eta_m \right\Vert_{\omega } \leq C_6 $, where $C_6 >0$ is a constant depends only on $\left( M,\omega \right) $.

Choosing a large integer $m_0$ such that $ m_0 >2C_5 \sqrt{m_0} +1 $. By the H\"ormander's $L^2$ estimate, we can find a smooth $\linebundle^m $-valued section $u_2 $ on $M$ such that $\partialbar u_2 = \partialbar \left( \eta_{m} S_{T_m ,V_{r^{1.5}_m} / \Gamma_V } \right) $, and 
\begin{eqnarray}
& & \int_{M} \left\Vert u_2 \right\Vert_{h^m}^{2} e^{-5(n+k)\psi_m } d\V_\omega \nonumber \\
& \leq & \frac{1 }{m -C_3 \varepsilon_m^{-1} } \int_{M} \left\Vert \partialbar u_2 \right\Vert_{h^m ,\omega}^{2} e^{-5(n+k)\psi_m } d\V_\omega \\
& \leq &  \frac{C_4 }{m -C_3 \sqrt{m} } \int_{ \pi_V \left( V_{r^{1.5}_m } \right) \sq \pi_V \left( V_{\delta_m } \right) } \left\Vert S_{T_m ,V_{r^{1.5}_m} / \Gamma_V } \right\Vert_{h^m}^{2} d\V_\omega \leq \frac{C_7 }{m -C_3 \sqrt{m} } , \nonumber
\end{eqnarray}
where $ S_{T_m ,V_{r^{1.5}_m} / \Gamma_V  }$ is the peak section of $T_m$ on $V_{r^{1.5}_m} / \Gamma_V  $, and $C_7 >0$ is a constant depending only on $\left( M,\omega ,\linebundle,h \right) $ and $k$. Then there exists a constant $C_8 >0$ depending only on $\left( M,\omega ,\linebundle,h \right) $ and $k$, such that $\int_{M} \left\Vert u_2 \right\Vert_{h^m}^{2} d\V_\omega \leq C_8 m^{-1} $, $\forall m\geq m_0 $. By the integrability of $u_2 $, we see that $T_m \left( u_2 \right) =0$, and hence we have
\begin{eqnarray}
\left\Vert T_m \big|_{\hl \left( M,\linebundle^m \right) } \right\Vert & \geq & \left( \int_{M} \left\Vert \eta_{m} S_{T_m ,V_{r^{1.5}_m} / \Gamma_V  } - u_2 \right\Vert_{h^m}^2 d\V_\omega \right)^{-\frac{1}{2}} \left| T_m \left( \eta_{m} S_{T_m ,V_{r^{1.5}_m} / \Gamma_V  } - u_2 \right) \right| \nonumber \\
& \geq & \left( 1+ C_8 m^{-1} \right)^{-\frac{1}{2}} \left| T_m \left( S_{T_m ,V_{r^{1.5}_m}/ \Gamma_V  } \right) \right| . 
\end{eqnarray}
Analysis similar to that in the Part 1 now shows that for any integer $l\geq 0$,
\begin{eqnarray}
& & \left\Vert  T_m \left( S_{T_m ,U_m } \right) S_{T_m ,U_m } (y) \right\Vert = \left\Vert  T_m \left( S_{y,U_m } \right) S_{y,U_m } (y) \right\Vert \nonumber \\
& \leq & \left( \int_{V_{r^{1.5}_m} / \Gamma_V  } \left\Vert  S_{y,U_{m} }  \right\Vert_{h^m }^2 d\V_\omega \right)^{\frac{1}{2}} \left| T_m \left( S_{T_m,V_{r^{1.5}_m / \Gamma_V  }} \right) \right| \left( \rho_{m,\omega ,U_m} (y) \right)^{\frac{1}{2}}  \\
& \leq & C_9 m^{ -2k - 4n-l } \left|  T_m \left( S_{T_m,U_m} \right) \right| ,\;\; \forall y\in \pi_V \left( V_{r_m } \right) \sq \pi_V \left( V_{r^{1.4}_m } \right) , \nonumber
\end{eqnarray}
where $C_9 >0 $ is a constant depending only on $\left( M,\omega ,\linebundle,h \right) ,k$ and $l$, and $S_{y,W}$ is the peak section of $\delta_y $ on $W$. Since $ T_m\big|_{\hl \left( V_{\beta_m} / \Gamma_V  ,\linebundle^m \right) } \neq 0 $, we can see that $\left\Vert  S_{T_m ,U_m } (y) \right\Vert \leq C_7 m^{ -2k - 4n-l } $, for any $ y\in \pi_V \left( V_{r_m } \right) \sq \pi_V \left( V_{r^{1.4}_m } \right) $. 

Let $ \chi_m \in C^\infty \left( M \right) $ be a cut-off function such that $\chi_m =1$ on $ \pi_V \left( V_{r_m^{1.4}} \right) $, $\chi_m =0$ on $M\sq \pi_V \left( V_{r^{1.1}_m} \right) $, and $\sum_{i=1}^{2} \left\Vert \nabla^i \chi_m \right\Vert \leq C_{10} $, where $C_{10} >0$ is a constant depending only on $\left( M,\omega ,\linebundle,h \right) ,k$ and $l$. By using the H\"ormander's $L^2$ estimate to the weight function $e^{-5(n+k)\psi_m } $ and the smooth $\linebundle^m $-valued $(0,1)$-form $\partialbar \left( \chi_m S_{T_m ,U_m} \right) $, we can conclude that
\begin{eqnarray} 
\left\Vert T_m \big|_{\hl \left( M,\linebundle^m \right) } \right\Vert  \geq  \left( 1- C_{11}  m^{-2k-2n-l} \right) \left\Vert T_m \big|_{\hl \left( U_m ,\linebundle^m \right) } \right\Vert  ,
\end{eqnarray}
where $C_{11} >0 $ is a constant depending only on $\left( M,\omega ,\linebundle,h \right) ,k$ and $l$. This completes the proof.
\end{proof}

We assume that $u=O\left( \left| \log h_D \right|^{-\alpha } \right)$ as $h_D \to 0^+ $ to all orders with respect to $\omega_V$ for some constant $\alpha >0$ from now. Then we can give an estimate for the derivatives of holomorphic sections. 

\begin{lmm}
\label{lmmderiholosections}
We follow the above assumption about $u$. Given constants $n,k\in\mathbb{N}$ and $Q>0 $. There exists a constant $C>0$ with the following property.

Let $(M,\omega )$ be an $n$-dimensional K\"ahler manifold, $\linebundle $ be a holomorphic line bundle on $M$ equipped with a hermitian metric $h$, and $x\in M$. Assume that  $ \inj \left( B_r (x) \right) \geq r $, $\Ric (h) = \omega $ on $B_r (x) $, and $ \sum_{j=0}^{k} \sup_{B_r (x)} \left\Vert \nabla^j \Ric (\omega ) \right\Vert_{\omega} \leq Q $, where $r>0$ is a constant. Then for any $k,m\in\mathbb{N }$ and $S\in \hl \left( M,\linebundle^m \right) $, we have $ \left\Vert \nabla^l S (x) \right\Vert^2 + \left\Vert \nabla^{*l} S^* (x) \right\Vert^2  \leq C m^{n+k } \left( \min \left\lbrace \sqrt{m} r ,1 \right\rbrace \right)^{-2n-2k} \left\Vert S \right\Vert^2_{L^2 ; B_r (x) } .$
\end{lmm}

\begin{proof}
Without loss of generality, we can assume that $\left\Vert S \right\Vert_{L^2 ; B_r (x) } \neq 0 $. By replacing $(M, \omega )$ with $(M,m\omega )$, we can reduce this lemma to the case of $m=1$. By the theory of Cheeger-Gromov $C^{k,\alpha } $-convergence, we can find a coordinate $\left( z_1 ,\cdots ,z_n \right) : B_{r } (x) \to U_x \subset \mathbb{C}^n$ such that the image of $x$ is $ 0 $, $ e^{-Q'} \omega_{Euc} \leq \omega \leq e^{Q'} \omega_{Euc} $, and the hermitian matrix $\left( g_{i\bar{j}} \right) $ satisfies that $g_{i\bar{j}} (0) = \delta_{ij} $ , $dg_{i\bar{j}} (0) = 0$, and $\left\Vert g_{i\bar{j}} \right\Vert^{*}_{k+1,\frac{1}{2}; U_x } \leq Q' $, where $\left\Vert f \right\Vert^{*}_{k,\alpha } $ is the interior norms on a domain in $\mathbb{R}^{2n}$, and $Q' =Q' (n,k,Q) >0 $ is a constant. We further assume that there exists a holomorphic frame $e_x$ of $\linebundle $ on this coordinate such that the local representation function of $h$, $a=h\left( e_x ,e_x \right) $, satisfying that $a(0)=1$ and $\left\Vert a-1 \right\Vert^{* }_{k+1,\frac{1}{2} ;U_x }  \leq Q' r^2 $. 

Write $S=f e_x $, where $f\in\mathcal{O} \left( U_x \right) $ is a holomorphic function. Now we are reduced to establishing the estimate $ \sum_{l=0}^{k} \left\Vert \nabla^l f (0) \right\Vert_{\omega_{Euc}}^2 \leq C  \left( \min \left\lbrace r ,1 \right\rbrace \right)^{-2n-2k} \int_{U_x} |f|^2 d\V_{\mathbb{C}^n} $ on the Euclidean domain $U_x \subset \mathbb{C}^n $ containing $B_{e^{-Q'} r} (0)$. This is an immediate consequence of the standard interior estimates for derivatives of harmonic functions. This is our claim.
\end{proof}

As a corollary of Lemma \ref{lmmderiholosections}, we can give the following estimate of the derivatives of the holomorphic sections around the complex hyperbolic cusps. Here we denote $e^{- \frac{\sqrt{m}}{\log m }}$ and $\frac{\log m}{\sqrt{m}} $ briefly by $r_m$ and $\varepsilon_m $, respectively, for any $m\geq 2$. 

\begin{coro}
\label{corocoarseestimatecpxhpblccusp}
We follow the above assumption about $u$. Given $\delta ,\kappa ,\beta ,k >0$, there exists a constant $C>1$ depending only on $(M,\omega ,\linebundle ,h)$, $\delta $, $\kappa $, $\beta $ and $k$ satisfying the following property.

Let $m \geq C $ be an integer such that $ \pi_V \left( V_{r_m^{\kappa }} \right) \subset U $, $U_m $ be a sequence of open subsets of $U$ containing $\pi_V \left( V_{r_m^{\kappa }} \right) $, and $ \widetilde{U}_m $ be open subsets of $U_m$ such that $ \dist_{\omega } \left( \widetilde{U}_m ,M\sq U_m \right) \geq \delta \varepsilon_m $. Then for any $S\in \hl \left( U_m ,\linebundle^m \right) $, we have $ \left\Vert \nabla^k S (x) \right\Vert_{h^m ,\omega}^2 + \left\Vert \nabla^{*k} S^* (x) \right\Vert_{h^m ,\omega}^2 \leq C m^{C } \left| \log h_D (x) \right|^{-\beta } \left\Vert S \right\Vert^2_{L^2 ;h^m ,\omega} $, $ \forall x\in \widetilde{U}_m $, where $S^* $ is the smooth section of $\linebundle^{-m} $ given by $S^* (e) = h^m \left( e,S \right) $, and $\nabla^* $ is the Chern conncection of the Hermitian line bundle $\left( \linebundle^{-m} ,h^{-m} \right) $.
\end{coro}

\begin{proof}
Without loss of generality, we can assume that $\left\Vert S \right\Vert_{L^2 ; h^m , \omega } =1 $ and $\kappa =\delta =1$. 

Note that that $u=O\left( \left| \log h_D \right|^{-\alpha } \right)$ as $h_D \to 0^+ $ to all orders with respect to $\omega_V$ for some constant $\alpha >0$. Let $\tau_m =e^{-m^{2+2\alpha^{-1}}} $. By Lemma \ref{lmminjectivityradiushyperbpliccusp}, we see that there exists a constant $C_1 >1$ depending only on $(M,\omega) $ and $k$, such that for any $m\in\mathbb{N }$, we have $\inj \left( U\sq \pi_V \left( V_{\tau^m_m} \right) \right) \geq C_1^{-1} m^{-2-2\alpha^{-1} -1} ,$ and $ \sum_{j=0}^{2k} \left\Vert \nabla^j \Ric (\omega ) \right\Vert_\omega \leq C_1 $ on $U$. Then we can use Lemma \ref{lmmderiholosections} to find a constant $C_2 >0$ depending only on $(M,\omega) $ and $k$, such that
\begin{eqnarray}
 \sup_{ \widetilde{U}_m \sq \pi_V \left( V_{\tau^m_m} \right) } \left( \left\Vert \nabla^k S \right\Vert^2 +\left\Vert \nabla^{*k} S^* \right\Vert^2 \right) \leq C_2 m^{C_2 } .
\end{eqnarray}
Since $ \left| \log h_D \right| \leq m^{2+2\alpha^{-1} } $ on $\widetilde{U}_m \sq \pi_V \left( V_{\tau^m_m} \right) $, we can conclude that 
\begin{eqnarray}
 \left\Vert \nabla^k S (x) \right\Vert^2 +\left\Vert \nabla^{*k} S^* (x) \right\Vert^2  \leq C_2 m^{C_2 + 2 \beta +2\alpha^{-1} \beta } \left| \log h_D (x) \right|^{-\beta } ,\;\; \forall x\in \widetilde{U}_m \sq \pi_V \left( V_{\tau^m_m} \right) .
\end{eqnarray}
By the definition, we have $\left\Vert S \right\Vert^2 \leq \rho_{ U_m , \omega ,\linebundle ,h ,m}  $ on $ U_m $. Since $h=\left| \log h_D \right| e^u $, $\omega = \omega_V -\sqrt{-1} \partial\partialbar u $ and $u=O\left( \left| \log h_D \right|^{-\alpha} \right) $, we see that
\begin{eqnarray}
\sup_{V_{\tau_m } / \Gamma_V } \left( \left| \log \left( \frac{h^m}{\left| \log h_D \right|^m } \right) \right| +\left| \log \left( \frac{\omega^n}{\omega^n_V } \right) \right| \right) & \leq & C_3 m \left| \log \tau_{m} \right|^{- m \alpha } \leq C_3 m^{-1} ,
\end{eqnarray}
where $C_3 >1 $ is a constant depending only on $(M,\omega) $. Lemma \ref{lmmdistancedifferencepeaksection} now shows that 
\begin{equation}
\rho_{V_{\tau_m } / \Gamma_V ,\omega ,m } (x) \leq e^{C_3 } \rho_{V_{\tau_m } / \Gamma_V ,\omega_V ,m } (x) , \;\; \forall x\in V_{\tau_m } / \Gamma_V ,
\end{equation}
where $\rho_{V_{\tau_m } / \Gamma_V ,\omega ,m } $ is the $m$-th Bergman kernel on $\left( V_{\tau_m } /\Gamma_V ,\omega ,\linebundle ,h \right) $.

Analysis similar to that in the proof of Proposition \ref{propbergmanknv} shows that for any $m\geq n+1 $,
\begin{eqnarray}
\rho_{V_{\tau_m } / \Gamma_V ,\omega_V ,m } (x) & \leq & C_4 \left| \log h_D (x) \right|^m \sum_{q=1}^\infty q^{n+1} \tau^{-q}_m h_D (x)^q ,\;\; \forall x\in V_{\tau_m } /\Gamma_V ,
\end{eqnarray}
where $C_4 >1 $ is a constant depending only on $(M,\omega) $. Then we can conclude that
\begin{eqnarray}
\rho_{V_{\tau_m } / \Gamma_V ,\omega_V ,m } (x) & \leq & C_4 h_D (x)^{\frac{1}{2}} \left| \log h_D (x) \right|^m \sum_{q=1}^\infty q^n \left( \tau^{-q}_m h_D (x)^{\frac{1}{2}} \right)^q \nonumber \\
& \leq & C_5 \left| \log h_D (x) \right|^m h_D (x)^{\frac{1}{2}} \leq C^2_5 \left( \tau_m^{\frac{m}{4}} \left| m \log \tau_m \right|^m \right) h_D (x)^{\frac{1}{4}} \leq C^3_5 h_D (x)^{\frac{1}{4}},
\end{eqnarray}
for any $m\geq n+1 $ and $x \in V_{\tau_m^m} /\Gamma_V $, where $C_5 >1 $ is a constant depending only on $(M,\omega) $.

It follows that there exists a constant $C_6 >1 $ depending only on $(M,\omega) $, such that for any $x \in V_{\tau_m^m} /\Gamma_m $, we have $ \left\Vert S \right\Vert^2_{L^2 , B_{1 } (x) ; h^m ,\omega } \leq C_6 h_D (x)^{C_6^{-1}} $. Applying Lemma \ref{lmmderiholosections} again, we can find a constant $C_7$ depending only on $(M,\omega), k  $ and $\beta $ such that $\left\Vert \nabla^k S (x )\right\Vert^2 + \left\Vert \nabla^{*k} S^* (x) \right\Vert^2 \leq C_7 \left| \log h_D (x) \right|^{-\beta } $ on $V_{\tau_m^m}$. This completes the proof.
\end{proof}

Before proving Theorem \ref{thmcusplocalizationprinciple}, we now prove that $\rho_{V_{\sigma_m } /\Gamma_V , \omega ,m} > 0 $ everywhere. Hence the fraction in inequality (\ref{equation3thmcusplocalizationprinciple}) is well defined.

\begin{lmm}
\label{lmmconstructholomorphicfunctions}
Let $x \in V/\Gamma_V $, $m_0 \geq n+1 $ be an integer, and $ T\in H^0 \left( V/\Gamma_V, \left( \linebundle_V / \Gamma_V \right)^{m_0 } \right)^* $ be a linear functional defined by finite linear combination of $\delta_{x } $. Assume that $T\neq 0$ as a distribution, and the quotient space $V/\Gamma_V $ is a manifold. Then $T \big|_{ \hl \left( V/\Gamma_V,\left( \linebundle_V / \Gamma_V \right)^{m_0 } \right) } \neq 0$.
\end{lmm}

\begin{proof}
To shorten notation, we write $N$ instead of $\left| \Gamma_V \right| $.
 
Since $ \Theta_{\Gamma_V } (g^{ N} ) = 1 $, $\forall g\in\Gamma_V $, we see that $\left( V/\Gamma_V,\left( \linebundle_V / \Gamma_V \right)^{N } \right)  \cong  \left( \mathbb{C} , \left| \log h_D \right|^{ N } \right) $. Then there exists a holomorphic section $e \in H^0 \left( V/\Gamma_V,\left( \linebundle_V / \Gamma_V \right)^{N } \right) $ such that $ \left\Vert e \right\Vert_{h^{N}_V }^2 = \left| \log h_D \right|^{ N} $. For each $ p \in \mathbb{Z}_{\geq 0} $, we can define $ T_p \in H^0 \left( V/\Gamma_V,\left( \linebundle_V / \Gamma_V \right)^{m } \right)^* $ by 
$$ T_p \left( S \right) = T \left( e^{-p} (S) \right) e^{p } ,\;\; \forall S \in H^0 \left( V /\Gamma_V , \left( \linebundle_V /\Gamma_V \right)^{ m_0 + pN } \right) .$$

Since $ \left( V/\Gamma_V ,\omega_V , \linebundle_V / \Gamma_V , h_V \right) $ is a polarized complete K\"ahler manifold with constant Ricci curvature $\Ric \left( \omega_V \right) = -(n+1) \omega_V $, we see that Lemma \ref{peaksec} implies that there exist a large integer $p_1 $ and a section $S_{p_1} \in \hl \left( V/\Gamma_V, \left( \linebundle_V/\Gamma_V \right)^{m_0 + p_1 N } \right) $ such that $T_{p_1 } ( S_{p_1} ) \neq 0$. It follows that $ T \left( e^{-p_1} (S) \right) \neq 0 $.

Let $S' =f e^{m_0}_{\linebundle_V } \in \hl \left( V ,\linebundle_{V}^{m_0 } \right)_{\Gamma_V } $ be the pullback of $e^{-p_1} (S) $. By the statement above Proposition \ref{propbergmanknv}, we can find constants $a_{q,j} $ such that $ \sum_{q= 1 }^{\infty} \sum_{j=1}^{N'_{m_0 ,q} } a_{ q,j} \widetilde{S}'_{m_0 ,q,j} $ converge to $f $ uniformly on a neighborhood of $x$, where $\left\lbrace S'_{m_0 ,q,j} \right\rbrace_{j=1}^{N'_{m_0 ,q}} $ is an $L^2 $ orthonormal set in the Hilbert space $\hl \left( D,\linebundle^{-q}_D \right) $, and $\widetilde{S}'_{m_0 , q,j} $ denote the holomorphic function corresponding to $S'_{m_0 ,q,j} $ defined in Lemma \ref{holoontotal}. Now we see that there exists a holomorphic function $\widetilde{S}_{m_0 ,q,j} $ such that $ T \left( \widetilde{S}'_{m_0 ,q,j} e_{\linebundle_V }^{m_0 } \right) \neq 0 $, and the lemma follows.
\end{proof}

Now we are ready to prove Theorem \ref{thmcusplocalizationprinciple}.

\vspace{0.2cm}

\noindent \textbf{Proof of Theorem \ref{thmcusplocalizationprinciple}: }
Take a fixed $l\in\mathbb{N} $ at first. We only need to show this theorem in the case where $\xi =\kappa =1$, the more general case is then a fairly immediate consequence. 

The proof will be divided into three parts, proving estimates (\ref{equation1thmcusplocalizationprinciple}), (\ref{equation2thmcusplocalizationprinciple}) and (\ref{equation3thmcusplocalizationprinciple}) respectively. 

\smallskip

\par {\em Part 1:} First, we need to prove (\ref{equation1thmcusplocalizationprinciple}).

Set $\left( x_m ,y_m \right) \in \pi_V \left( V_{\gamma_m } \right) \times \left(U\sq \pi_V \left( V_{\sigma_m } \right) \right) $. Choosing vectors $e_{x_m} \in \linebundle^m \big|_{x_m} $ and $e_{y_m} \in \linebundle^m \big|_{y_m} $ such that $\left\Vert e_{x_m} \right\Vert_{h^m} = \left\Vert e_{y_m} \right\Vert_{h^m} = 1 $. Let $k_1 ,k_2 \in\mathbb{Z}_{\geq 0} $ such that $k_1 +k_2 \leq k $. Consider the unit tangent vectors $v_{1,m} ,\cdots ,v_{k_1 ,m} \in T_{x_m} M \oplus 0 \subset T_{\left( x_m ,y_m \right)} ( M\times M) $ and $v_{k_1 +1,m} ,\cdots ,v_{k_1 +k_2 ,m} \in 0 \oplus T_{y_m} M \subset T_{\left( x_m ,y_m \right)} (M\times M) $. Now we can define linear functionals on the space of holomorphic sections on small neighborhoods of $x_m $ and $y_m $ as following:
\begin{eqnarray}
\nabla^{k_1} S \left( v_{1,m} ,\cdots ,v_{k_1 ,m}  \right) & = & T_{1,m} (S) e_{x_m} ,\\
\nabla^{* k_2} S^* \left( v_{k_1 +1,m} ,\cdots ,v_{k_1 +k_2,,m}  \right) & = & \overline{ T_{2,m} (S) } e^*_{y_m} , 
\end{eqnarray}
for any holomorphic section $ S $ of $ \linebundle^m $ around $x_m $ and $y_m$, respectively. Then we have
\begin{eqnarray}
& & \left\Vert \nabla_x^{k_1} \nabla_y^{*k_2} B_{\omega ,m} \left( x_m ,y_m \right) \left( v_{ 1,m} ,\cdots ,v_{k_1 +k_2,m}  \right) \right\Vert_{h^m ,\omega } \nonumber \\
& = & \left\Vert \sum_{j} \nabla^{k_1} S_j \big|_{x_m} \left( v_{ 1,m} ,\cdots ,v_{k_1 ,m}  \right) \otimes \nabla^{*k_2} S^*_j \big|_{y_m} \left( v_{ k_1 +1 ,m} ,\cdots ,v_{k_1 +k_2 ,m}  \right) \right\Vert_{h^m ,\omega } \\
& = & \left\Vert \left( \sum_j T_{1,m} \left( S_j \right) \overline{T_{2,m} \left( S_j \right) } \right) e_{x_m} \otimes e_{y_m}^* \right\Vert_{h^m } = \left| T_{1,m} \left( S_{T_{1,m}} \right) \overline{T_{2,m} \left( S_{T_{1,m}} \right) } \right|  , \nonumber
\end{eqnarray}
where $\left\Vert \cdot \right\Vert_{h^m ,\omega }$ means the norms corresponding to the Hermitian metric $ \pi_1^* h^m \otimes \pi_2^* h^{-m} $ and K\"ahler metric $ \pi_1^* \omega + \pi_2^* \omega $ on $M\times M $, the family $\left\lbrace S_j \right\rbrace_{j\in J} $ is an $L^2$-orthonormal basis of $ \hl \left( M,\linebundle^m \right) $, and $S_{T_{1,m}} \in \hl \left( M,\linebundle^m \right) $ is the peak section of the functional $T_{1,m} $ on $M$. Since $u=o(1)$ as $h_D \to 0^+ $, we can find a constant $\epsilon' >0$ depending only on $(M,\omega ,\linebundle ,h ) $ such that 
\begin{eqnarray}
& & \min \left\lbrace \dist_{\omega } \left( \pi_V \left(V_{\gamma_m } \right) ,M\sq \pi_V \left( V_{\sqrt{\gamma_m \sigma_m }} \right) \right) , \right.  \\
& & \;\;\;\;\;\;\;\;\;\;\;\;\;\;\;\;\;\;\; \left. \dist_{\omega } \left( M\sq \pi_V \left(V_{\sigma_m } \right) ,  \pi_V \left( V_{\sqrt{\gamma_m \sigma_m }} \right) \right) \right\rbrace \geq \epsilon' \varepsilon_m , \nonumber
\end{eqnarray}
and Lemma \ref{lmminjectivityradiushyperbpliccusp} now shows that there exists a constant $\delta \in \left( 0,\epsilon' \right)$ depending only on $(M,\omega) $, such that $ \inj \left( U\sq \pi_V \left(V_{\gamma_m } \right) \right) \geq \delta \varepsilon_m  $. By Corollary \ref{corocoarseestimatecpxhpblccusp}, we see that there exists a constant $C_0 >0$ depending only on $(M,\omega ,\linebundle ,h ) ,k $ and $\beta $, such that
\begin{eqnarray}
& & \left| \log h_D \left( x_m \right) \right|^{2\beta} \left\Vert T_{1,m} \big|_{\hl \left( M,\linebundle^m \right) } \right\Vert  \\
& & \;\;\;\;\;\;\;\;\;\;\;\;\;\;\;\;\;\;\; + \left| \log h_D \left( y_m \right) \right|^{2\beta} \left\Vert T_{2,m} \big|_{\hl \left( B_{\delta \varepsilon_m} \left( y_m \right) ,\linebundle^m \right) } \right\Vert \leq C_0 m^{C_0}  .\nonumber
\end{eqnarray}
By using Lemma \ref{peaksectionrestriction} and Proposition \ref{propcoraseestimate} to $V_{\gamma_m }$ and $V_{\sqrt{\gamma_m \sigma_m } } $, we can conclude that for each integer $q\in\mathbb{N}$, there exists a constant $C_1 >0$ depending only on $(M,\omega ,\linebundle ,h ) ,k$ and $q $, such that 
\begin{equation}
\int_{M\sq \pi_V \left( V_{\sqrt{\gamma_m \sigma_m } } \right) } \left\Vert S_{T_{1,m}} \right\Vert_{h^m}^2 d\V_{\omega} \leq C_1 m^{-q} .
\end{equation}
Choosing $q= 2l+4+4\lfloor C_0 \rfloor $, where $\lfloor \cdot \rfloor$ is the greatest integer function. It follows that
\begin{eqnarray}
\left| T_{1,m} \left( S_{T_{1,m}} \right) \overline{T_{2,m} \left( S_{T_{1,m}} \right) } \right| & \leq & \left\Vert T_{1,m} \right\Vert \cdot \left\Vert T_{2,m} \big|_{\hl \left( B_{\delta \varepsilon_m} \left( y_m \right),\linebundle^m \right) } \right\Vert \cdot \left\Vert S_{T_{1,m}} \right\Vert_{L^2 ,B_{\delta \varepsilon_m} \left( y_m \right) ;h^m  } \nonumber \\
& \leq & C_0^2 C_1 m^{2C_0 } m^{-l-2-2\lfloor C_0 \rfloor} \left| \log h_D \left( x_m \right) \right|^{-\beta} \left| \log h_D \left( y_m \right) \right|^{-\beta} \\
& \leq & C_0^2 C_1 m^{-l } \left| \log h_D \left( x_m \right) \right|^{-\beta} \left| \log h_D \left( y_m \right) \right|^{-\beta} . \nonumber
\end{eqnarray}
Hence we can conclude that 
\begin{equation}
\sup_{ \pi_V \left( V_{\gamma_m } \right) \times \left(U\sq \pi_V \left( V_{\sigma_m } \right) \right) }   \left| \log h_D (x) \right|^\beta \left| \log h_D (y) \right|^\beta  \left\Vert B_{\omega ,m} (x,y) \right\Vert_{C^k ;h^m ,\omega } \leq C_2 m^{-l} .
\end{equation}
where a constant $C_2 >0$ depending only on $(M,\omega ,\linebundle ,h ) ,k,l $ and $\beta $. By exchanging $x_m $ and $y_m $, one can see that
\begin{equation}
\sup_{ \left(U\sq \pi_V \left( V_{\sigma_m } \right) \right) \times \pi_V \left( V_{\gamma_m } \right) }   \left| \log h_D (x) \right|^\beta \left| \log h_D (y) \right|^\beta  \left\Vert B_{\omega ,m} (x,y) \right\Vert_{C^k ;h^m ,\omega } \leq C_2 m^{-l} .
\end{equation} 

\smallskip

\par {\em Part 2.} Our task now is to prove (\ref{equation2thmcusplocalizationprinciple}).

Set $\left( x_m ,y_m \right) \in \pi_V \left( V_{\gamma_m } \right) \times \pi_V \left( V_{\gamma_m } \right) $. Let $e_{x_m} \in \linebundle^m \big|_{x_m} $ and $e_{y_m} \in \linebundle^m \big|_{y_m} $ such that $\left\Vert e_{x_m} \right\Vert_{h^m} = \left\Vert e_{y_m} \right\Vert_{h^m} = 1 $. Choosing $k_1 ,k_2 \in\mathbb{Z}_{\geq 0} $ such that $k_1 +k_2 \leq k $, unit tangent vectors $v_{1,m} ,\cdots ,v_{k_1 ,m} \in T_{x_m} M \oplus 0 \subset T_{\left( x_m ,y_m \right)} ( M\times M) $, and unit tangent vectors $v_{k_1 +1,m} ,\cdots ,v_{k_1 +k_2 ,m} \in 0 \oplus T_{y_m} M \subset T_{\left( x_m ,y_m \right)} (M\times M) $. Now we can define a linear functional on $\hl \left( \pi_V \left( V_{\gamma_m } \right) , \linebundle^m \right)$ as following:
\begin{eqnarray}
\nabla^{k_1} S \left( v_{1,m} ,\cdots ,v_{k_1 ,m}  \right) & = & T_{1,m} (S) e_{x_m} ,\\
\nabla^{* k_2} S^* \left( v_{k_1 +1,m} ,\cdots ,v_{k_1 +k_2,,m}  \right) & = & \overline{ T_{2,m} (S) } e^*_{y_m} , 
\end{eqnarray}
for any $ S \in\hl \left( \pi_V \left( V_{\gamma_m } \right) , \linebundle^m \right) $. Let $S_{T_{i,m ,\sigma_m}} \in \hl \left( M,\linebundle^m \right) $ and $S_{T_{i,m}} \in \hl \left( M,\linebundle^m \right) $ denote the peak sections of the functional $T_{i,m} $ on $\pi_V \left( V_{\sigma_m } \right) $ and $M$, respectively. As in the argument in Part 1, a straightforward computation shows that
\begin{eqnarray}
& & \left\Vert \nabla_x^{k_1} \nabla_y^{*k_2} \left( B_{\omega ,m} \left( x_m ,y_m \right) - B_{\omega ,m,\sigma_m } \right) \left( x_m ,y_m \right) \left( v_{ 1,m} ,\cdots ,v_{k_1 +k_2,m}  \right) \right\Vert_{h^m ,\omega}  \\
& = & \left\Vert \left( T_{1,m} \left( S_{T_{1,m}} \right) \overline{T_{2,m} \left( S_{T_{1,m}} \right) } -T_{1,m} \left( S_{T_{1,m,\sigma_m }} \right) \overline{T_{2,m} \left( S_{T_{1,m,\sigma_m }} \right) } \right) e_{x_m} \otimes e_{y_m}^* \right\Vert_{h^m} \nonumber \\
& \leq & \left| T_{1,m} \left( S_{T_{1,m}} - S_{T_{1,m,\sigma_m }} \right) \right| \cdot \left| T_{2,m} \left( S_{T_{1,m,\sigma_m }} \right)  \right| + \left| T_{1,m} \left( S_{T_{1,m}} \right) \right| \cdot \left| T_{2,m} \left( S_{T_{1,m }} - S_{T_{1,m,\sigma_m }} \right)  \right| . \nonumber
\end{eqnarray}
By Corollary \ref{corocoarseestimatecpxhpblccusp}, we see that there exists a constant $C_3 >0$ depending only on $(M,\omega ,\linebundle ,h ) $ and $k $, such that 
\begin{eqnarray}
 & & \left| \log h_D \left( x_m \right) \right|^{2\beta} \left\Vert T_{1,m} \big|_{\hl \left( \pi_V \left( V_{\sigma_m } \right) ,\linebundle^m \right) } \right\Vert \\
 & & \quad\quad\quad\quad\quad\;\;\;\;\; + \left| \log h_D \left( y_m \right) \right|^{2\beta} \left\Vert T_{2,m} \big|_{\hl \left( \pi_V \left( V_{\sigma_m } \right) ,\linebundle^m \right) } \right\Vert \leq C_3 m^{C_3} . \nonumber
\end{eqnarray}
Combining Lemma \ref{peaksectionrestriction} and Proposition \ref{propcoraseestimate}, we see that for each given integer $q\in\mathbb{N}$, there exists a constant $C_4 >0$ depending only on $(M,\omega ,\linebundle ,h ) ,k$ and $q $, such that
\begin{eqnarray}
\int_{ \pi_V \left( V_{\sigma_m } \right) } \left\Vert S_{T_{1,m}} - S_{T_{1,m,\sigma_m }} \right\Vert^2 d\V_{\omega} \leq C_3 m^{-q} .
\end{eqnarray} 
Choosing $q= 2l+4+4\lfloor C_2 \rfloor $, where $\lfloor \cdot \rfloor$ is the greatest integer function. Then we have
\begin{eqnarray}
& & \left| T_{1,m} \left( S_{T_{1,m}} - S_{T_{1,m,\sigma_m }} \right) \right| \cdot \left| T_{2,m} \left( S_{T_{1,m,\sigma_m }} \right)  \right| + \left| T_{1,m} \left( S_{T_{1,m}} \right) \right| \cdot \left| T_{2,m} \left( S_{T_{1,m }} - S_{T_{1,m,\sigma_m }} \right)  \right| \\
& \leq & 2 \left\Vert T_{1,m} \big|_{\hl \left( \pi_V \left( V_{\sigma_m } \right) ,\linebundle^m \right) } \right\Vert \cdot \left\Vert T_{2,m} \big|_{\hl \left( \pi_V \left( V_{\sigma_m } \right) ,\linebundle^m \right) } \right\Vert \cdot \left\Vert S_{T_{1,m}} - S_{T_{1,m,\sigma_m }} \right\Vert_{L^2 ; V_{ \sigma_m } } \nonumber \\
& \leq & 2C_2^2 C_3 m^{2C_0} m^{-l-2-2\lfloor C_2 \rfloor } \left| \log h_D \left( x_m \right) \right|^{-\beta} \left| \log h_D \left( y_m \right) \right|^{-\beta} \nonumber \\
& \leq & 2C_2^2 C_3 m^{-l} \left| \log h_D \left( x_m \right) \right|^{-\beta} \left| \log h_D \left( y_m \right) \right|^{-\beta}. \nonumber
\end{eqnarray}
Then we can obtain the estimate (\ref{equation2thmcusplocalizationprinciple}), and the theorem follows.

\smallskip

\par {\em Part 3.} It remains to prove (\ref{equation3thmcusplocalizationprinciple}).

Let $ x_m \in \pi_V \left( V_{\gamma_m } \right) $, and $e_{x_m} \in \linebundle^m \big|_{x_m} $ such that $\left\Vert e_{x_m} \right\Vert_{h^m} = 1 $. Then we can define a linear functional $T_{0,m} : \hl \left( \pi_V \left( V_{\gamma_m } \right) ,\linebundle^m \right) \to \mathbb{C} $ by $ S\left( x_m \right) = T_{0,m} \left(S\right) e_{x_m } $, $ \forall S\in \hl \left( \pi_V \left( V_{\gamma_m } \right) ,\linebundle^m \right) $. Then we have
\begin{eqnarray}
\frac{\rho_{M ,\omega ,m} (x)}{\rho_{V_{\sigma_m } /\Gamma_V , \omega ,m} \left( x \right) } = \frac{\left\Vert T_{0,m} \big|_{\hl \left( M,\omega ,\linebundle^m ,h^m \right) } \right\Vert }{\left\Vert T_{0,m} \big|_{\hl \left( V_{\sigma_m } /\Gamma_V ,\omega ,\linebundle^m ,h^m \right) } \right\Vert} ,
\end{eqnarray}
and (\ref{equation3thmcusplocalizationprinciple}) follows from Proposition \ref{propcoraseestimate}.
\qed

\subsection{Peak sections on Poincar\'e type cusps}
\hfill

Now we consider the peak sections on Poincar\'e type cusps. This subsection is an analogy of the previous subsection on Poincar\'e type cusp. Let $\bar{M} $ be a compact K\"ahler manifold, $D$ be a simple normal crossing divisor, and $(M,\omega )= \left( \bar{M} \sq D ,\omega \right) $ be a complete K\"ahler manifold with Poincar\'e type cusp along $D$. Fix a Riemannian metric $\bar{g}$ on $\bar{M}$. Write $\bar{d}(x) = \dist_{\bar{g}} \left(  x,D \right) $, $\forall x\in M $. Assume that $(M,\omega )$ is a complete K\"ahler manifold with $\Ric (\omega ) \geq -\Lambda \omega $ and $( \linebundle ,h )$ is a Hermitian line bundle with $\Ric (h) \geq \epsilon \omega $, where $\epsilon ,\Lambda >0$ are constants. 

The proof of the following estimate of the injectivity radius on $M$ is similar to that of Lemma \ref{lmminjectivityradiushyperbpliccusp}, so we omit it.

\begin{lmm}
\label{lmminjectivityradiuspoincarecusp}
Given an integer $k\in \mathbb{N} $. Then there exists a constant $C=C(k,M,\omega )>0 $ such that the curvature $ \sum_{j=0}^{5k+10n} \left\Vert \nabla^j \Ric (\omega ) \right\Vert_{\omega } \leq C $ on $M$, and for any $x\in M$, the injectivity radius $ \inj_M \left( x \right) \geq C^{-1} \left( \left| \log \bar{d} (x) \right| +1 \right)^{-1} $.
\end{lmm}

Choosing $x\in D$. By definition, we can find a neighborhooda $U_x $ of $x$ and a biholomorphic map 
$$ \varphi_x : U_x \sq D \to \prod_{i=1}^{k_x} \mathbb{D}_{r_{x,i}}^* \times B^{2n-2k_x }_{r_{x,k_x ,n}} (0) \subset \left( \mathbb{D}_{1}^* \right)^{k_x } \times \mathbb{C}^{ n- k_x }  ,$$
such that $ \left( \varphi^{-1}_x \right)^{*} \omega = -\sqrt{-1} \sum_{i=1}^{k_x} \partial\partialbar \log \left( \log \left| z_i \right|^2 \right) + \pi_{ k_x ,n }^* \omega_{x, k_x,n } + \sqrt{-1} \partial\partialbar u_x $, where $r_{x,j} $, $ r_{x,k_x ,n} $ are positive constants, $\omega_{x, k_x,n } $ is a K\"ahler metric on $B^{2n-2k_x }_{r_{x,k_x ,n}} (0)$, $u_x $ is a smooth function on $ U_x \sq D $, and there exists a constant $\alpha >0$ such that $u_x \circ \varphi_x^{-1} \in \cap_{i=1 }^{k_x} O\left( \left| \log \left| z_i \right| \right|^{-\alpha } \right) $ as $\inf_{1\leq i\leq k_x} \left|z_{j,i} \right| \to 0^+ $ to all orders with respect to $\omega_{\model}$ on $\prod_{i=1}^{k_x} \mathbb{D}_{r_{x,i}}^* \times B^{2n-2k_x }_{r_{x,k_x ,n}} (0) $. 

Let us denote by $U_{x,t,s} $ the domain $\left( \mathbb{D}^*_{t} \right)^{k_x } \times B^{2n-2k_x }_{s} (0) $ for any constants $t,s\in (0,1) $. It will cause no confusion if we use the same letter to designate a subset of $U_x \sq D$ and its image under $\varphi_x $. Now we can construct quasi-plurisubharmonic functions on $U_x \sq D$. This is an analogy of Lemma \ref{lmmquasipshhyperboliccuspver2} on Poincar\`e type cusp.

\begin{lmm}
\label{lmmquasipshpoincarecusp}
Let $\beta ,\delta ,\epsilon >0 $ be constants.  Assume that $ \delta < e^{-2}n^{-2} r_{x,j} $, $j=1,\cdots ,k_x $, and $\epsilon \leq e^{-2}n^{-2} r_{x,n} $. Then for any $p\in U_{x,\delta^{1+2\beta } ,\epsilon } $, there exists a quasi-plurisubharmonic function $\psi \in L^1 (M,\omega ) \cap C^\infty \left( M\sq \{ p \} \right) $ such that $\psi \leq 0 $, $\mathrm{supp} \psi \subset U_{x,\delta ,4\epsilon } $, $\sqrt{-1} \partial\partialbar \psi \geq -C \left( \left| \log \delta \right| + \epsilon^{-2} \right) \omega $ and $ \lim_{y\to p} \left( \psi (y) - \log \left( \dist_{\omega } (p,y) \right) \right) = -\infty $, where $C>0$ is a constant dependent only on $(M,\omega ) $, and $\beta $.
\end{lmm}

\begin{proof}
Fix a point $p=\left( p_1 ,\cdots ,p_n \right) \in  \left( \mathbb{D}^*_{\delta^{1+2\beta } } \right)^{k_x} \times B^{2n-2k_x }_{\epsilon} (0) = U_{x,\delta^{1+2\beta } ,\epsilon } $. Let 
$$ f(y) =\sum_{i=1}^{k_x } \delta^{-2} \left| y_i - p_i \right|^2 + \sum_{i= k_x +1}^{n} \epsilon^{-2} \left| y_i - p_i \right|^2 ,\;\; \forall y=\left( y_1,\cdots ,y_n \right) \in U_x .$$ 
It is easy to see that $ f(y) \geq \min \left\lbrace \delta^{2\beta } , e^{-2} \right\rbrace $, $ \forall y\in U_x \sq U_{x,\delta^{1+\beta } , 2\epsilon } $, and $f(y) \leq 16n $, $ \forall y\in U_{x,\delta , 4\epsilon } $. 

Since $u=o(1)$, we can conclude that $ \left| \nabla \log f (x) \right| \leq C_1 \left( \left| \log \delta \right| + \epsilon^{-1} \right) ,$ $\forall x\in U_{x,\delta , 4\epsilon } \sq U_{x,\delta^{1+\beta } , 2\epsilon } $, where $C_1 $ is a constant depending only on $(M,\omega )$ and $\beta $.

Now we consider the model metric $\omega_{\model } = -\sum_{i=1}^{k_x} \partial\partialbar \log \left| \log \left| z_i \right|^2 \right| $ on $\left( \mathbb{D}^*_{\delta } \right)^{k_x} $. By a straightforward calculation, we see that $ \dist_{\omega_{\model}} \left( \left( \mathbb{D}^*_{\delta^{1+\beta } } \right)^{k_x} ,\left( \mathbb{D}^*_{1 } \right)^{k_x} \sq \left( \mathbb{D}^*_{\delta } \right)^{k_x} \right) \geq C_2^{-1} \log \left( 1+ \beta \right) $ for some constant $C_2 $ depending only on $(M,\omega )$. Then we can choose a cut-off function $\eta_1 \in C^\infty \left( \left( \mathbb{D}^*_{1 } \right)^{k_x} \right) $ such that $\eta_1 =1 $ on $\left( \mathbb{D}^*_{\delta^{1+\beta } } \right)^{k_x} $, $\eta_1 =0 $ on $\left( \mathbb{D}^*_{1 } \right)^{k_x} \sq \left( \mathbb{D}^*_{\delta } \right)^{k_x} $, $0\leq \eta_1 \leq 1$, and $ \sum_{i=0}^2 \left\Vert \nabla^i \eta_1 \right\Vert_{\omega_{\model}} \leq C_3 ,$ where $C_3 >0 $ is a constant depending only on $n$ and $\beta $. Choosing another cut-off function $\eta_2 \in C_0^\infty \left( B^{2n-2k_x }_{4\epsilon } (0) \right) $ such that $\eta_2 =1 $ on $B^{2n-2k_x }_{2\epsilon} (0) $, $0\leq \eta_2 \leq 1$, and $ \sum_{i=0}^2 \epsilon^i \left\Vert \nabla^i \eta_2 \right\Vert_{\omega_{Euc}} \leq C_4 ,$ where $C_4 $ is a constant depending only on $n $. By the biholomorphic map $ \varphi_x $, we see that $\eta_1 $ and $\eta_2 $ can be viewed as functions on $U_x $.

Let $M = \sup_{U_{x,\delta ,4\epsilon}} f $, and $\psi (x) = \eta_1 (x) \eta_2 (x) \left( \log f(x) - \log M \right) $. Analysis similar to that in the proof of Lemma \ref{lmmquasipshhyperboliccuspver2} shows that $\varphi_x $ satisfies the properties we need.
\end{proof}

We will denote by $W_{r } $ the domain $\left\lbrace x\in M \big| \bar{d} (x,D) <r  \right\rbrace \subset M =\bar{M} \sq D $. Now we consider the peak sections around Poincar\'e type cusps. To apply Lemma \ref{lmmquasipshpoincarecusp} to points around $D$, we need the following decomposition result for the manifold $(M,\omega) $. For abbreviation, we denote $e^{- \frac{\sqrt{m}}{\log m}}$ and $ \frac{\log m }{ \sqrt{m} } $ briefly by $r_m$ and $ \varepsilon_m $, respectively, for any $m\geq 2$.

\begin{lmm}
There are constants $m_0 ,b_1 ,b_2 ,b_3 ,C_1 ,C_2 >0$, and a decreasing sequence $\left\lbrace \xi_m \right\rbrace_{m=1}^\infty $ that converges to $0$ as $m\to\infty $ satisfying the following property.

Let $m>m_0 $ be a constant. Then there are points $y_{i,m} \in D $, $i=1,\cdots , N_m $, positive constants $C_{1,i,m } \leq C_1 $, $C_{2,i,m } \leq C_2 $, open neighborhoods $U_{i,m} $ of $y_{i,m} $ in $\bar{M} $ and biholomorphic maps 
$$ \varphi_{i,m} : U_{i,m} \sq D \to \big( \mathbb{D}^*_{r_m^{b_1} } \big)^{k_{i,m}} \times B^{2n-2k_x }_{ b_{2} \varepsilon_m } (0) $$ 
such that $U_{i,m} \sq D \subset W_{r_m } $, $ W_{r^{b_3}_m} \subset \cup_{i=1}^{N_m} \varphi^{-1}_{i,m} \left( \big( \mathbb{D}^*_{r_m^{C_{1 ,i,m} b_1} } \big)^{k_{i,m}} \times B^{2n-2k_x }_{ \frac{ b_{2} \varepsilon_m }{C_{2,i,m} } } (0) \right) $, and the K\"ahler metrics $ \left(\varphi_{i,m }^{-1} \right)^* \omega = \omega_{\model , k_{i,m} ,n} +\sqrt{-1} \partial\partialbar u_{i,m} $ on $U_{i,m} \sq D $ for some real valued function $u_{i,m}$ satisfying that $ \left\Vert u_{i,m} \right\Vert_{C^2 ;\omega }  \leq \xi_m $ on $ U_{i,m} \sq D $, $i=1,\cdots ,N_m $.
\end{lmm}

\begin{proof}
By the compactness of the divisor $D\subset \bar{M} $, we can find points $\widetilde{y}_j \in D $, $j=1,\cdots ,N$, open neighborhoods $\widetilde{U}_j $ of $ \widetilde{y}_j $ in $\bar{M} $, and biholomorphic maps 
$$ \varphi_j = \left( z_{j,1} \cdots ,z_{j,n} \right) : \widetilde{U}_j \to \big( \mathbb{D}_{\epsilon_j } \big)^{k_{j}} \times B_{ \epsilon_j }^{2n-2k_{j}} \left( 0 \right) $$ 
such that $\varphi_j \left( \widetilde{y}_j \right) =0 $, $ \varphi_j \left( \widetilde{U}_j \sq D \right) = \big( \mathbb{D}^*_{\epsilon_j } \big)^{k_{j}} \times B_{ \epsilon_j }^{2n-2k_{j}} \left( 0 \right) $, $ D\subset \cup_{j=1}^N \varphi_j^{-1} \left( \big( \mathbb{D}_{\epsilon_j^2 } \big)^{k_{j}} \times B_{ \frac{\epsilon_j }{2} }^{2n-2k_{j}} \left( 0 \right) \right)  $, and $\left( \varphi_j^{-1} \right)^* \omega = -\sqrt{-1} \sum_{i=1}^{k_j} \partial\partialbar \log \left( \log \left| z_i \right|^2 \right) + \pi_{k_j ,n}^* \omega_{j} +\sqrt{-1} \partial\partialbar u_j $, where $\pi_i $ is the projection of $\mathbb{C}^n $ onto the $i$-th component, $\epsilon_j >0$ are constants, $\omega_{j} $ is a K\"ahler metric on $ B_{ \epsilon_j }^{2n-2k_{j}} \left( 0 \right) $, and $u_j = o (1) $ as $ \inf_{1\leq i\leq k_x} \left|z_{j,i} \right| \to 0^+ $ to all orders with respect to $\omega $. 

Fix $j\in \{ 1,\cdots ,N \}$. It is sufficient to construct the open subsets and biholomorphic maps on $ \widetilde{U}_j $. Without loss of generality, we can assume that $\omega = \varphi_j^* \omega_{\model } $. For each nonnegative integer $ q\leq k_j $, we consider the set
$$ Z_q = \bigcup_{\left\lbrace i_1 ,\cdots , i_q \right\rbrace \subset \left\lbrace 1,\cdots ,k_j \right\rbrace  } \left( \bigcap_{l=1}^{q} \left\lbrace z_{j,i_l} =0 \right\rbrace  \Bigg\sq \bigcup_{l\in \left\lbrace 1,\cdots ,k_j \right\rbrace \sq \left\lbrace i_1 ,\cdots , i_q \right\rbrace } \left\lbrace z_{j,l} =0 \right\rbrace \right) \subset \widetilde{U}_j . $$
Then we have $\cup_{q=0}^{k_j} Z_q = \widetilde{U}_j \cap D $, and $Z_q \cap Z_{q'} =\empty $ for each $q\neq q'$.

By Lemma \ref{lmminjectivityradiuspoincarecusp}, we can find a constant $\delta >0$ such that for any point $z\in\mathbb{D}_1^* $, there exist an open neighborhood $U_{z } $ of $z$ in $\mathbb{D}_1^* $, an increasing function $\varrho (t) $ on $(0,1)$, and a biholomorphic map $\phi_z : U_z \to \mathbb{D}_{\delta \left| \log |z| \right|^{-1} } $ such that $\phi_z (z)=0 $, $\lim_{t\to 0^+ } \varrho (t) =0 $, and $ \omega_{\model ,1,1} - \phi_z^* \omega_{Euc} = \sqrt{-1} \partial\partialbar u_{U_z } $ for some function $ u_{ U_z } $ satisfying that $ \left\Vert u_{U_z } \right\Vert_{C^2 ; \omega_{1,1} } \leq \varrho (|z|) $ on $U_z $, where $\omega_{Euc} $ is the Euclidean K\"ahler form on $\mathbb{C}^n $. 

Choosing $y\in Z_q \cap \varphi_j^{-1} \left( \big( \mathbb{D}_{\epsilon_j^2 } \big)^{k_{j}} \times B_{ \frac{\epsilon_j }{2} }^{2n-2k_{j}} \left( 0 \right) \right) $. Without loss of generality, we can assume that $y\in \cap_{l=1}^q \left\lbrace z_{j,l} =0 \right\rbrace  $. Let $t\leq n^{-2} e^{-2} \epsilon_j^2 $ be a positive constant. When $y\notin B^{Euc ,n}_{t} \left( \cup_{k=p+1}^n Z_k \right) $, then we can find an increasing function $\widetilde{\varrho } (t) $ on $(0,1)$, an open neighborhood $\widetilde{U}_{y,t} $ of $y$ in $\widetilde{U}_j $ and a biholomorphic map 
$$\widetilde{\phi}_{y,t} = \left( z_{y,t,1} , \cdots ,z_{y,t,n} \right) : \widetilde{U}_{y,t} \to \left( \mathbb{D}_{\delta t } \right)^{p} \times B_{\delta \left| \log t \right|^{-1} }^{2n-2p} \left( 0 \right) $$ 
such that $\lim_{t\to 0^+ } \widetilde{\varrho } (t) =0 $, $\widetilde{\phi}_{y,t} (y)=0 $, $ \widetilde{U}_{y,t} \sq D = \cup_{l=1}^{p} \left\lbrace z_{y,t,l} =0 \right\rbrace $, and $ \widetilde{\phi}_{y,t}^* \omega_{\model } - \omega = \sqrt{-1} \partial\partialbar u_{\widetilde{U}_{y,t} } $ for some function $ u_{\widetilde{U}_{y,t} } $ satisfying that $ \big\Vert u_{\widetilde{U}_{y,t} } \big\Vert_{C^2 ; \omega  } \leq \varrho (t) $ on $U_{y,t} $, where $B^{Euc ,n}_t (x) $ is the Euclidean ball in $\mathbb{C}^n $ with radius $t$ and center $x\in\mathbb{C}^n $, and $B^{Euc ,n}_t (A)=\cup_{x\in A} B^{Euc ,n}_t (x) $.

Now we see that for each constant $\mu >0 $, there are constants $C_j >0 $ and $\epsilon_{l,j} >0 $, $l=1,2,\cdots ,k_j $, such that 
\begin{eqnarray}
 \varphi^{-1}_{j } \left( \big( \mathbb{D}_{r_m^{C_j } } \big)^{k_{j}} \times B_{ \frac{ \epsilon_j}{2} }^{2n-2k_{j}} \left( 0 \right) \right) & \subset & \bigcup_{l=1}^{k_j } \left( \bigcup_{y\in Z_l \sq B^{Euc ,n}_{\epsilon_{l,j} r_m} \left( \bigcup_{q>l } Z_q \right) } \widetilde{U}_{y,\epsilon_{l,j} r_m } \right) \label{constructionLemmadecompositionpoincaretype} \\
& \subset & \varphi^{-1}_{j } \left( \big( \mathbb{D}_{\mu r_m } \big)^{k_{j}} \times B_{ \epsilon_j }^{2n-2k_{j}} \left( 0 \right) \right) , \nonumber
\end{eqnarray}
for sufficiently large $m$. 

Clearly, we can find a small constant $\mu >0 $ such that $\varphi^{-1}_{j } \left( \left( \mathbb{D}^*_{\mu r_m } \right)^{k_{j}} \times B_{ \epsilon_j }^{2n-2k_{j}} \left( 0 \right) \right) \subset W_{r_m } $, for sufficiently large $m$. Then the open $ \widetilde{U}_{y,\epsilon_{l,j} r_m } $ subsets in (\ref{constructionLemmadecompositionpoincaretype}) are the open subsets we need.

Then we can prove the lemma by choosing $j=1,\cdots ,N$ in the above argument.
\end{proof}

This is the Poincar\`e type cusp version of Proposition \ref{propcoraseestimate}. The proof of this proposition is similar to the proof of Proposition \ref{propcoraseestimate}, except that Lemma \ref{lmmquasipshhyperboliccuspver2} is replaced by Lemma \ref{lmmquasipshpoincarecusp}.

\begin{prop}
\label{prop2coraseestimate}
Given constants $k,\delta ,\kappa ,l>0 $, there exists a constant $C>0$ depending only on $(M,\omega ,\linebundle ,h ),k,\delta ,\kappa $ and $l$, satisfying the following property.

For any integer $m\in\mathbb{N} $, choose a constant $\gamma_m \geq r^\kappa_m $, a point $p_m \in W_{\gamma_m } $, and a linear functional $ T_m $ defined by finite linear combination of $\delta_{p_m} $ and its derivatives. Assume that the order of derivatives contained in $T_m$ is at most $k$, $\forall m\in\mathbb{N} $. Suppose that there exists a sequence of open subsets $ \left\lbrace U_{m} \right\rbrace_{m=1}^\infty $ of $M$ such that $\mathrm{dist}_\omega \left( W_{\gamma_m} , M\sq U_m \right) \geq \delta \varepsilon_m $, $ \forall m\in\mathbb{N} $. Then we have
\begin{eqnarray}
\label{inequality2propcoraseestimate}
\left\Vert T_m \big|_{\hl \left( M,\linebundle^m \right) } \right\Vert \geq \left( 1-\frac{C}{m^l } \right) \left\Vert T_m \big|_{\hl \left( U_m,\linebundle^m |_{U_m} \right) } \right\Vert .
\end{eqnarray}
\end{prop}

\begin{rmk}
Note that for any point $p$ on the manifold $M$, we can choose a trivialization of $\linebundle $ around $p$, so $\delta_p $ can still represent a linear functional on $\hl \left( M,\linebundle \right) $.
\end{rmk}

\begin{proof}
Without loss of generality, we can assume that $\kappa =\delta =1 $. When $T_m \big|_{\hl \left( U_m,\linebundle^m \right) } = 0  ,$ this proposition is obvious. So we can assume that $ T_m \big|_{\hl \left( U_m,\linebundle^m \right) } \neq 0 $ from now.

Let $\beta >0 $ be a large constant. When $p_m \notin W_{\gamma^{\beta }_m } $, analysis similar to that in the first part of the proof of Proposition \ref{propcoraseestimate} shows that (\ref{inequality2propcoraseestimate}) holds for some constant $C$ depending only on $(M,\omega ,\linebundle ,h ),k,\delta ,\beta ,\kappa $ and $l$. 

Now we assume that $p_m \in W_{\gamma^{\beta }_m } $. In this case, by replacing Lemma \ref{lmmquasipshhyperboliccuspver2} with Lemma \ref{lmmquasipshpoincarecusp}, we can apply the argument in the second part in the proof of Proposition \ref{propcoraseestimate} to give a proof of this proposition, and the proof is complete.
\end{proof}

The next result is an analogy of Corollary \ref{corocoarseestimatecpxhpblccusp} on K\"ahler manifolds with Poincar\`e type cusps, and it is also a corollary of Lemma \ref{lmmderiholosections}.

\begin{coro}
\label{corocoarseestimatepoincarecusp}
Given $\epsilon ,\kappa ,\beta ,k >0$, there exists a constant $C>0$ depending only on $(M,\omega)$, $\epsilon$, $\kappa $, $\beta $ and $k$ satisfying the following property.

Let $U_m $ be an open subset of $M$ containing $W_{r_m^{\kappa }} $ for some integer $m\geq C$, and $ \widetilde{U}_m $ be an open subset of $U_m$ such that $ \dist_{\omega } \left( \widetilde{U}_m ,M\sq U_m \right) \geq \epsilon \varepsilon_m $. Then for any $S\in \hl \left( U_m ,\linebundle^m \right) $, we have $ \left\Vert \nabla^k S (x) \right\Vert^2 + \left\Vert \nabla^{*k} S^* (x) \right\Vert^2 \leq C m^{C } \left| \log \bar{d} (x) \right|^{-\beta } \left\Vert S \right\Vert^2_{L^2 ;U_m} $, $\forall x\in \widetilde{U}_m $, where $S^* $ is the smooth section of $\linebundle^{-m} $ given by $S^* (e) = h^m \left( e,S \right) $, and $\nabla^* $ is the Chern conncection of the Hermitian line bundle $\left( \linebundle^{-m} ,h^{-m} \right) $.
\end{coro}

\begin{proof}
By the compactness of the divisor $D\subset \bar{M} $, we can find points $\widetilde{y}_j \in D $, $j=1,\cdots ,N$, open neighborhoods $\widetilde{U}_j $ of $ \widetilde{y}_j $ in $\bar{M} $, and biholomorphic maps 
$$ \varphi_j = \left( z_{j,1} \cdots ,z_{j,n} \right) : \widetilde{U}_j \to \left( \mathbb{D}_{\epsilon_j } \right)^{k_{j}} \times B_{ \epsilon_j }^{2n-2k_{j}} \left( 0 \right) $$ 
such that $\varphi_j \left( \widetilde{y}_j \right) =0 $, $ \varphi_j \left( \widetilde{U}_j \sq D \right) = \big( \mathbb{D}^*_{\epsilon_j } \big)^{k_{j}} \times B_{ \epsilon_j }^{2n-2k_{j}} \left( 0 \right) $, $ D\subset \cup_{j=1}^N \varphi_j^{-1} \left( \left( \mathbb{D}_{\epsilon_j^2 } \right)^{k_{j}} \times B_{ \epsilon^2_j }^{2n-2k_{j}} \left( 0 \right) \right) $, and $\left( \varphi_j^{-1} \right)^* \omega = -\sqrt{-1} \sum_{i=1}^{k_j} \partial\partialbar \log \left( \log \left| z_i \right|^2 \right) + \pi_{k_j ,n}^* \omega_{ j } +\sqrt{-1} \partial\partialbar u_j $, where $\epsilon_j >0$ are constants, $\omega_j $ is a K\"ahler metric on $ B_{ \epsilon_j }^{2n-2k_{j}} \left( 0 \right) $, and $u_j = \cap_{i=1}^{k_j } O \left( \left| \log \left|z_{j,i} \right| \right|^{-\alpha } \right) $ as $ \inf_{1\leq i\leq k_x} \left|z_{j,i} \right| \to 0^+ $ to all orders with respect to $\omega_{ \model } $. 

Fix $j\in \{ 1,\cdots ,N \} $. Let $\gamma_m = e^{-m^{2+2\alpha^{-1}}} $. Now we consider the Bergman kernel functions on $ \widetilde{U}_{\widetilde{y}_j , \gamma_m ,\epsilon ,l } = \varphi_j^{-1} \left( \left( \mathbb{D}^*_{\epsilon } \right)^{l-1 }  \times \mathbb{D}^*_{\gamma_m }  \times \left( \mathbb{D}^*_{\epsilon } \right)^{k_{j} -l} \times B_{ \epsilon }^{2n-2k_{j}} \left( 0 \right) \right) \subset \widetilde{U}_j \sq D $, $l=1,\cdots ,k_j $. 

As in the proof of Corollary \ref{corocoarseestimatecpxhpblccusp}, we only need to find a constant $C_{j,l,\beta } >0 $ independent of $m$, such that for sufficiently large integer $m $, 
$$ \rho_{\widetilde{U}_{\widetilde{y}_j , \gamma_m ,\epsilon_j ,l } ,\varphi_j^* \omega_{\model } ,m } (x) \leq C_{j,l,\beta } m^{C_{j,l,\beta } } \left| \log \left| z_{j,l } (x) \right| \right|^{-\beta -2n} ,\;\; \forall x \in \widetilde{U}_{\widetilde{y}_j , \gamma^e_m ,\epsilon_j^2 ,l } .$$

By a straightforward calculation, one can obtain
\begin{eqnarray*}
\rho_{\widetilde{U}_{\widetilde{y}_j , \gamma_m ,\epsilon_j ,l } ,\varphi_j^* \omega_{\model } ,m } (x) & = & \left( \prod_{ \begin{subarray}{l}
1\leq  i  \leq k_j\\
\;\;\; i \neq l
\end{subarray}} \rho_{\mathbb{D}_{\epsilon_j }^* ,\omega_{\model } ,m } \left( z_{j,i} (x) \right) \right) \cdot \rho_{\mathbb{D}_{\gamma_m }^* ,\omega_{\model } ,m } \left( z_{j,l} (x) \right) \\
& & \cdot \; \rho_{B^{2n-2k_{j} }_{ \epsilon_j } (0) ,\omega_{Euc } ,m } \left( z_{j,k_j +1} (x) ,\cdots , z_{j,n} (x) \right) .
\end{eqnarray*}
Combining Corollary \ref{corobgmkernelsup} with Tian-Yau-Zelditch expansion, we see that there exists a constant $C_1 =C_1 (n) >0 $ such that $\rho_{\widetilde{U}_{\widetilde{y}_j , \gamma_m ,\epsilon_j ,l } ,\varphi_j^* \omega_{\model } ,m } (x) \leq C_1 m^{\frac{3n}{2}} \rho_{\mathbb{D}_{\gamma_m }^* ,\omega_{\model } ,m } \left( z_{j,l} (x) \right) $ for sufficiently large $m$. It is sufficient to show that $ \rho_{\mathbb{D}_{\gamma_m }^* ,\omega_{\model } ,m } \left( z \right) \leq C_2 m^{C_2 } \left| \log \left| z \right| \right|^{-\beta -2n} $, $ \forall z \in \mathbb{D}_{\gamma_m^e }^* ,$ which is shown in the proof of Corollary \ref{corocoarseestimatecpxhpblccusp}. This completes the proof.
\end{proof}

We conclude this section by proving Proposition \ref{proppoincarelocalizationprinciple}.

\vspace{0.2cm}

\noindent \textbf{Proof of Theorem \ref{proppoincarelocalizationprinciple}: }
Fix $l\in\mathbb{N} $ at first. We only need to show this theorem in the case where $\xi =\kappa =1$, the more general case is then a fairly immediate consequence. 

Similar to the proof of Proposition \ref{thmcusplocalizationprinciple}, we also split the proof of this theorem into two parts, proving estimates (\ref{equation1thmpoincarelocalizationprinciple}) and (\ref{equation2thmpoincarelocalizationprinciple}) respectively. 

\smallskip

\par {\em Part 1:} First, we need to prove (\ref{equation1thmpoincarelocalizationprinciple}).

Set $\left( x_m ,y_m \right) \in V_{m } \times \left(M\sq U_{m } \right) $. Choosing vectors $e_{x_m} \in \linebundle^m \big|_{x_m} $ and $e_{y_m} \in \linebundle^m \big|_{y_m} $ such that $\left\Vert e_{x_m} \right\Vert = \left\Vert e_{y_m} \right\Vert = 1 $. Let $k_1 ,k_2 \in\mathbb{Z}_{\geq 0} $ such that $k_1 +k_2 \leq k $. Consider the unit tangent vectors $v_{1,m} ,\cdots ,v_{k_1 ,m} \in T_{x_m} M \oplus 0 \subset T_{\left( x_m ,y_m \right)} ( M\times M) $, and $v_{k_1 +1,m} ,\cdots ,v_{k_1 +k_2 ,m} \in 0 \oplus T_{y_m} M \subset T_{\left( x_m ,y_m \right)} (M\times M) $. Now we can define bounded linear functionals on small neighborhoods of $x_m $ and $y_m $ as following:
\begin{eqnarray*}
\nabla^{k_1} S \left( v_{1,m} ,\cdots ,v_{k_1 ,m}  \right) & = & T_{1,m} (S) e_{x_m} ,\\
\nabla^{* k_2} S^* \left( v_{k_1 +1,m} ,\cdots ,v_{k_1 +k_2,,m}  \right) & = & \overline{ T_{2,m} (S) } e^*_{y_m} , 
\end{eqnarray*}
for any holomorphic section $ S $ of $ \linebundle^m $ around $x_m $ and $y_m$, respectively. Then we have
$$ \left\Vert \nabla_x^{k_1} \nabla_y^{*k_2} B_{\omega ,m} \left( x_m ,y_m \right) \left( v_{ 1,m} ,\cdots ,v_{k_1 +k_2,m}  \right) \right\Vert = \left| T_{1,m} \left( S_{T_{1,m}} \right) \overline{T_{2,m} \left( S_{T_{1,m}} \right) } \right|  ,$$
where $\left\lbrace S_j \right\rbrace_{j\in J} $ is an $L^2$-orthonormal basis of $ \hl \left( M,\linebundle^m \right) $, and $S_{T_{1,m}} \in \hl \left( M,\linebundle^m \right) $ is the peak section of the functional $T_{1,m} $ on $M$. Since $u=o(1)$ as $h_D \to 0^+ $, we can find a sequence of open subsets $\left\lbrace \widetilde{U}_m \right\rbrace_{m=2}^\infty $ and a constant $\epsilon >0$ independent of $m$ such that 
$$\inf \left( \dist_{\omega } \left( V_{m } ,M\sq \widetilde{U}_m \right) , \dist_{\omega } \left( \widetilde{U}_m ,M\sq U_{m} \right) \right) \geq \epsilon \frac{\log m}{\sqrt{m}} ,\;\; \forall m\in\mathbb{N} ,$$ 
and Lemma \ref{lmminjectivityradiuspoincarecusp} now shows that there exists a constant $\delta \in (0,\epsilon )$ independent of $m$ such that $ \inj \left( M\sq V_{m } \right) \geq \delta \frac{\log (m)}{\sqrt{m}} $. By Corollary \ref{corocoarseestimatepoincarecusp}, we see that there exists a constant $C_0 >0$ depending only on $(M,\omega ,\linebundle ,h ) $ and $k $, such that 
$$ \left| \log \bar{d} \left( x_m \right) \right|^{2\beta} \left\Vert T_{1,m} \big|_{\hl \left( M,\linebundle^m \right) } \right\Vert + \left| \log \bar{d} \left( y_m \right) \right|^{2\beta} \left\Vert T_{2,m} \big|_{\hl \left( B_{\frac{\delta \log m}{\sqrt{m}}},\linebundle^m \right) } \right\Vert \leq C_0 m^{C_0} ,\;\; \forall m\in\mathbb{N} ,$$ 
where $\epsilon >0$ is a constant independent of $m$. By using Lemma \ref{peaksectionrestriction} and Proposition \ref{prop2coraseestimate} to $V_{m }$ and $ \widetilde{U}_{m} $, we can conclude that for each integer $q\in\mathbb{N}$, there exists a constant $C_1 >0$ depending only on $(M,\omega ,\linebundle ,h ),k$ and $q $, such that 
$$ \int_{M\sq \widetilde{U}_{m} } \left\Vert S_{T_{1,m}} \right\Vert^2 d\V_{\omega} \leq C_1 m^{-q} ,\;\; \forall m\in\mathbb{N} .$$
Choosing $q= 2l+4+4\lfloor C_0 \rfloor $, where $\lfloor \cdot \rfloor$ is the greatest integer function. It follows that
\begin{eqnarray*}
\left| T_{1,m} \left( S_{T_{1,m}} \right) \overline{T_{2,m} \left( S_{T_{1,m}} \right) } \right| & \leq & \left\Vert T_{1,m} \right\Vert \cdot \left\Vert T_{2,m} \big|_{\hl \left( B_{\frac{\delta \log m}{\sqrt{m}}},\linebundle^m \right) } \right\Vert \cdot \left\Vert S_{T_{1,m}} \right\Vert_{L^2 ;B_{\frac{\delta \log m}{\sqrt{m}}} } \\
& \leq & C_0^2 C_1 m^{2C_0 } m^{-l-2-2\lfloor C_0 \rfloor} \left| \log \bar{d} \left( x_m \right) \right|^{-\beta} \left| \log \bar{d} \left( y_m \right) \right|^{-\beta} \\
& \leq & C_0^2 C_1 m^{-l } \left| \log \bar{d} \left( x_m \right) \right|^{-\beta} \left| \log \bar{d} \left( y_m \right) \right|^{-\beta} .
\end{eqnarray*}
Hence we can conclude that 
$$\sup_{ V_{m } \times \left(M\sq U_{m } \right) }   \left| \log \bar{d} (x) \right|^\beta \left| \log \bar{d} (y) \right|^\beta  \left\Vert B_{\omega ,m} (x,y) \right\Vert_{C^k  } \leq C m^{-l} .$$
By exchanging $x_m $ and $y_m $, one can see that the proof of 
$$\sup_{ \left(M\sq U_{m } \right) \times V_{m } }   \left| \log \bar{d} (x) \right|^\beta \left| \log \bar{d} (y) \right|^\beta  \left\Vert B_{\omega ,m} (x,y) \right\Vert_{C^k  } \leq C m^{-l} $$ 
is similar to the argument above.

\smallskip

\par {\em Part 2.} It remains to prove (\ref{equation2thmpoincarelocalizationprinciple}).

Set $\left( x_m ,y_m \right) \in V_{m } \times V_{m }$. Let $e_{x_m} \in \linebundle^m \big|_{x_m} $ and $e_{y_m} \in \linebundle^m \big|_{y_m} $ such that $\left\Vert e_{x_m} \right\Vert = \left\Vert e_{y_m} \right\Vert = 1 $. Choosing $k_1 ,k_2 \in\mathbb{Z}_{\geq 0} $ such that $k_1 +k_2 \leq k $, unit tangent vectors $v_{1,m} ,\cdots ,v_{k_1 ,m} \in T_{x_m} M \oplus 0 \subset T_{\left( x_m ,y_m \right)} ( M\times M) $, and $v_{k_1 +1,m} ,\cdots ,v_{k_1 +k_2 ,m} \in 0 \oplus T_{y_m} M \subset T_{\left( x_m ,y_m \right)} (M\times M) $. Now we can define bounded linear functionals on $V_{\gamma_m}$ as following:
\begin{eqnarray*}
\nabla^{k_1} S \left( v_{1,m} ,\cdots ,v_{k_1 ,m}  \right) & = & T_{1,m} (S) e_{x_m} ,\\
\nabla^{* k_2} S^* \left( v_{k_1 +1,m} ,\cdots ,v_{k_1 +k_2,,m}  \right) & = & \overline{ T_{2,m} (S) } e^*_{y_m} , 
\end{eqnarray*}
for any $ S \in\hl \left( V_{m} , \linebundle^m \right) $. Let $S_{T_{U_m ,1,m }} \in \hl \left( U_m ,\linebundle^m \right) $ and $S_{T_{1,m}} \in \hl \left( M,\linebundle^m \right) $ denote the peak sections of the functional $T_{1,m} $ on $U_{m}$ and $M$, respectively. As in the argument in Part 1, a straightforward computation shows that
\begin{eqnarray*}
& & \left\Vert \nabla_x^{k_1} \nabla_y^{*k_2} \left( B_{\omega ,m} \left( x_m ,y_m \right) - B_{U_m ,\omega ,m } \right) \left( x_m ,y_m \right) \left( v_{ 1,m} ,\cdots ,v_{k_1 +k_2,m}  \right) \right\Vert \\
& = & \left\Vert \left( T_{1,m} \left( S_{T_{1,m}} \right) \overline{T_{2,m} \left( S_{T_{1,m}} \right) } -T_{1,m} \left( S_{T_{U_m ,1,m }} \right) \overline{T_{2,m} \left( S_{T_{U_m ,1,m }} \right) } \right) e_{x_m} \otimes e_{y_m}^* \right\Vert \\
& \leq & \left| T_{1,m} \left( S_{T_{1,m}} - S_{T_{U_m ,1,m }} \right) \right| \cdot \left| T_{2,m} \left( S_{T_{U_m ,1,m }} \right)  \right| + \left| T_{1,m} \left( S_{T_{1,m}} \right) \right| \cdot \left| T_{2,m} \left( S_{T_{1,m }} - S_{T_{U_m ,1,m }} \right)  \right| .
\end{eqnarray*}
By Corollary \ref{corocoarseestimatepoincarecusp}, we see that there exists a constant $C_2 >0$ depending only on $(M,\omega ,\linebundle ,h ) $ and $k $, such that 
$$ \left| \log \bar{d} \left( x_m \right) \right|^{2\beta} \left\Vert T_{1,m} \big|_{\hl \left( U_m ,\linebundle^m \right) } \right\Vert + \left| \log \bar{d} \left( y_m \right) \right|^{2\beta} \left\Vert T_{2,m} \big|_{\hl \left( U_m ,\linebundle^m \right) } \right\Vert \leq C_2 m^{C_2} ,\;\; \forall m\in\mathbb{N} ,$$ 
where $\epsilon >0$ is a constant independent of $m$. Combining Lemma \ref{peaksectionrestriction} and Proposition \ref{prop2coraseestimate}, we see that for each given integer $q\in\mathbb{N}$, there exists a constant $C_3 >0$ depending only on $(M,\omega ,\linebundle ,h ) ,k$ and $q $, such that 
$$ \int_{ V_{ \sigma_m } } \left\Vert S_{T_{1,m}} - S_{T_{U_m ,1,m }} \right\Vert^2 d\V_{\omega} \leq C_3 m^{-q} ,\;\; \forall m\in\mathbb{N} .$$
Choosing $q= 2l+4+4\lfloor C_2 \rfloor $, where $\lfloor \cdot \rfloor$ is the greatest integer function. Then we have
\begin{eqnarray*}
& & \left| T_{1,m} \left( S_{T_{1,m}} - S_{T_{U_m ,1,m }} \right) \right| \cdot \left| T_{2,m} \left( S_{T_{U_m ,1,m  }} \right)  \right| + \left| T_{1,m} \left( S_{T_{1,m}} \right) \right| \cdot \left| T_{2,m} \left( S_{T_{1,m }} - S_{T_{U_m ,1,m }} \right)  \right| \\
& \leq & 2 \left\Vert T_{1,m} \big|_{\hl \left( U_{M },\linebundle^m \right) } \right\Vert \cdot \left\Vert T_{2,m} \big|_{\hl \left( U_{m },\linebundle^m \right) } \right\Vert \cdot \left\Vert S_{T_{1,m}} - S_{T_{U_m ,1,m  }} \right\Vert_{L^2 ; U_{m } } \\
& \leq & 2C_2^2 C_3 m^{2C_0} m^{-l-2-2\lfloor C_2 \rfloor } \left| \log \bar{d} \left( x_m \right) \right|^{-\beta} \left| \log \bar{d} \left( y_m \right) \right|^{-\beta} \\
& \leq & 2C_2^2 C_3 m^{-l} \left| \log \bar{d} \left( x_m \right) \right|^{-\beta} \left| \log \bar{d} \left( y_m \right) \right|^{-\beta}.
\end{eqnarray*}
Then we can obtain the estimate (\ref{equation2thmpoincarelocalizationprinciple}), and the theorem follows.\qed

\section{Bergman kernels in two special cases}
\label{bergmanperturbedmanifoldssection}

In this section, we consider the Bergman kernels on complex hyperbolic cusps with K\"ahler-Einstein metrics and quotients of complex ball.

\subsection{Bergman kernels on complex hyperbolic cusps with K\"ahler-Einstein metrics}
\hfill

Let $(M,\omega )$ be a K\"ahler manifold and let $(\linebundle ,h)$ be a Hermitian line bundle on $M$ such that $\Ric (h) \geq \epsilon \omega $ and $\Ric (\omega ) \geq -\Lambda \omega $ for some constants $\epsilon ,\Lambda >0$. Suppose that $M$ is pseudoconvex if $(M,\omega )$ is not complete. Assume that $U \cong V_r /\Gamma_V $ is an asymptotic complex hyperbolic cusp on $(M,\omega , \linebundle ,h )$ such that $ \left( U,\omega \right) $ is a K\"ahler-Einstein manifold with $\Ric (\omega) = -(n+1) \omega $, and $ \bar{U} $ is complete. By definition, $\omega = \omega_V -\sqrt{-1} \partial\partialbar f $ for some function $f\in C^\infty \left( \bar{V}_r /\Gamma_V \right) $.

First, we show that $U$ is an asymptotic complex hyperbolic cusp of order $\infty$. Since $\pi_V^* \omega $ is a K\"ahler-Einstein metric on $V_r $, we only need to consider the case $\Gamma_V =0 $. 

For each constant $t>0 $, we denote the restriction of the scalar multiplication $v\to tv $ on the domain $V_{t^{-1} } \subset \linebundle_D $ as $\Phi_t $. By definition, $ \Phi_t : V_{t^{-1} } \to V_1 =V $ is a biholomorphic map. The pullbacks $\Phi_t^* \linebundle_V $ and $\Phi_t^* \linebundle^{-1}_V $ are holomorphic line bundles on $V_{t^{-1}} $, and the Bergman kernels satisfying that $ \Phi_t^* B_{V ,\omega_V ,m } = B_{V_{t^{-1}} ,\Phi_t^* \omega_V ,m } $ on $ V_{t^{-1}} \times V_{t^{-1} } $. Note that $ \Phi_t^* e_{\linebundle_V } $ is a non-vanishing holomorphic section on $V_{t^{-1}}$ and $ \left\Vert  \Phi_t^* e_{\linebundle_V } \right\Vert = \left| \log h_D \circ \Phi_t \right| $, so we see that $\Phi_t^* \linebundle_V $ and $\Phi_t^* \linebundle^{-1}_V $ are trivial line bundles on $V_{t^{-1}}$. Now we regard $ \left( \Phi_t^* \linebundle_V , \Phi_t^* h_V \right) $ and $ \left( \Phi_t^* \linebundle^{-1}_V , \Phi_t^* h^{-1}_V \right) $ as the restriction of holomorphic line bundles $\linebundle_V $ and $ \linebundle^{-1}_V $ on $V \cap V_{t^{-1}} $ respectively, just replace the Hermitian metrics $h_V $ and $h_V^{-1}$ with the Hermitian metrics $\left\Vert  \Phi_t^* e_{\linebundle_V } \right\Vert \cdot \left\Vert e_{\linebundle_V } \right\Vert^{-1} h_V $ and $ \left\Vert   e_{\linebundle_V } \right\Vert \cdot \left\Vert \Phi_t^* e_{\linebundle_V } \right\Vert^{-1} h^{-1}_V $, respectively. Therefore, the pull back of the Bergman kernels, $ \Phi_t^* B_{V ,\omega_V ,m } $, can be regarded as the $\pi_1^* \linebundle_V \otimes \pi_2^* \linebundle_V^{-1} $-valued functions on $ \left( V \cap V_{t^{-1}} \right) \times \left( V \cap V_{t^{-1}} \right) $. By considering the change of the Hermitian line bundle under the mapping $ \Phi_t $, we can prove the following lemma. 

\begin{lmm}
\label{lmmdatarfusongmanifold1}
Let $\left( V,\omega_V ,\linebundle_V ,h_V \right) $ be an $n$-dimensional complex hyperbolic cusp, $t>0 $ be a constant, and $ \Phi_t :V_{t^{-1}} \to V $ be the restriction of the scalar multiplication $v\mapsto tv $ on $V_{t^{-1}} \subset \linebundle_D $.

Then we have $ \Phi_t^* \omega_V = \omega_V -\sqrt{-1} \partial\partialbar\log \left( 1-2 \log t \left| \log h_D \right|^{-1} \right) $ on $ V\cap V_{t^{-1}} $.
\end{lmm}

\begin{proof}
By definition, we have $\omega_V = -\sqrt{-1} \partial\partialbar \log \left| \log h_D \right| $, and hence
\begin{eqnarray*}
\Phi_t^* \omega_V & = & -\sqrt{-1} \Phi_t^* \left( \partial\partialbar \log \left| \log h_D \right| \right) = -\sqrt{-1} \partial\partialbar \log \left| \log h_D \circ \Phi_t \right| \\
& = & -\sqrt{-1} \partial\partialbar \log \left| \log t^2 h_D \right| = -\sqrt{-1} \partial\partialbar \log \left( \left| \log h_D \right| - 2\log t \right) \\
& = & -\sqrt{-1} \partial\partialbar \log \left| \log h_D \right| - \sqrt{-1} \partial\partialbar \log\left( 1-2 \log t \left| \log h_D \right|^{-1} \right) \\
& = & \omega_V -\sqrt{-1} \partial\partialbar \log\left( 1-2 \log t \left| \log h_D \right|^{-1} \right) ,\;\; \textrm{on $ V\cap V_{t^{-1}} $.}
\end{eqnarray*}
This establishes the formula.
\end{proof}

Combining Lemma \ref{lmmdatarfusongmanifold1} with Fu-Hein-Jiang's estimate (\ref{fhjasymptoticbehavior}), we see that $U$ is an asymptotic complex hyperbolic cusp of order $\infty$ by choosing $t= e^{ -\frac{  c_{r,f} }{2} }  $, where $c_{r,f}$ is the constant in (\ref{fhjasymptoticbehavior}).

We now come to the proof of Theorem \ref{thmdatarfusongmanifold}. Here $r_m $ denotes $e^{-\frac{\sqrt{m}}{\log m}}$ for $m\geq 2$.

\vspace{0.2cm}

\noindent \textbf{Proof of Theorem \ref{thmdatarfusongmanifold}: }
Without loss of generality, we can assume that $\Gamma_V =0$. By definition, there exists a smooth function $u$ on $U$ such that $ \left( U,\omega ,\linebundle ,h \right) \cong \left( V_r ,\omega_V -\sqrt{-1}\partial\partialbar u ,\linebundle_V , e^u h_V \right) $ and for any $\alpha >0$, $u=O\left( \left| \log h_D \right|^{-\alpha } \right)$ as $h_D \to 0^+ $ to all orders with respect to $\omega_V$. 

We divide our proof into three parts.

\smallskip

\par {\em Part 1.} At first, we prove estimate (\ref{inequalitydatarfusongmanifold1}). 

By definition, we see that $ U $ is an asymptotic complex hyperbolic cusp on $(M,\omega ,\linebundle ,h )$, and hence the estimate (\ref{inequalitydatarfusongmanifold1}) is a consequence of the localization principle (\ref{equation1thmcusplocalizationprinciple}). 

\smallskip

\par {\em Part 2.}
We now start to consider estimate (\ref{inequalitydatarfusongmanifold2}). 

By the localization principle (\ref{equation2thmcusplocalizationprinciple}), one can see that the estimate (\ref{inequalitydatarfusongmanifold2}) can be localized, and hence we only need to show that there are constants $m_0 ,C>0$ independent of $m$, such that
$$ \sup_{V_{r_m } \times V_{r_m } } \left| \log h_D (x) \right|^\beta \left| \log h_D (y) \right|^\beta \left\Vert B_{V_{\sqrt{r_m } } ,\omega ,m } (x,y) - B_{V_{\sqrt{r_m } } , \omega_{V} ,m } (x,y) \right\Vert_{C^k ; h_{V}^m , \omega_{V} } \leq C m^{-l} , $$
for any $m\geq m_0 $. Choosing a large integer $m_1 $ such that $r_m \leq r^4 $, $\forall m\geq m_1 $. Fix an integer $m\geq m_1$ and a point $\left( x_m ,y_m \right) \in V_{r_m } \times V_{r_m }$. Let $e_{x_m} \in \linebundle_V^m \big|_{x_m} $ and $e_{y_m} \in \linebundle_V^m \big|_{y_m} $ such that $\left\Vert e_{x_m} \right\Vert_{h_{V}^m} = \left\Vert e_{y_m} \right\Vert_{h_{V}^m} = 1 $. Choosing nonnegative integers $k_1 ,k_2 $ such that $k_1 +k_2 \leq k $, unit tangent vectors $v_{1,m} ,\cdots ,v_{k_1 ,m} \in T_{x_m} V \oplus 0 \subset T_{\left( x_m ,y_m \right)} ( V\times V) $, and unit tangent vectors $v_{k_1 +1,m} ,\cdots ,v_{k_1 +k_2 ,m} \in 0 \oplus T_{y_m} V \subset T_{\left( x_m ,y_m \right)} (V\times V) $. Here the metric we consider is $ \omega_{V} $. It is easy to see that $\hl \left( V_{\sqrt{\gamma_m }},\omega , \linebundle_V^m ,h^m \right) = \hl \left( V_{\sqrt{\gamma_m }}, \omega_{V} , \linebundle_V^m , h^m_{V} \right) $ as vector subspaces of $H^0 \left( V_{\sqrt{\gamma_m }} ,\linebundle^m_V \right) $, and the $L^2$-norms are equivalent. Now we can define bounded linear functionals on $V_{\gamma_m}$ as following:
\begin{eqnarray*}
\nabla^{k_1} S \left( v_{1,m} ,\cdots ,v_{k_1 ,m}  \right) & = & T_{1,m} (S) e_{x_m} ,\\
\nabla^{* k_2} S^* \left( v_{k_1 +1,m} ,\cdots ,v_{k_1 +k_2,,m}  \right) & = & \overline{ T_{2,m} (S) } e^*_{y_m} , 
\end{eqnarray*}
for any $ S \in\hl \left( V_{\gamma_m  } , \linebundle_V^m \right) $. Let $S_{T_{i,m }} \in \hl \left( V_{\sqrt{\gamma_m }},\linebundle_V^m \right) $ and $S'_{T_{i,m}} \in \hl \left( V_{\sqrt{\gamma_m }},\linebundle_V^m \right) $ denote the peak sections of the functional $T_{i,m} $ on $ V_{\sqrt{\gamma_m }} $ associated with $\left( h, \omega \right)$ and $ \left( h_V , \omega_V \right) $, respectively. By a straightforward computation, we can conclude that
\begin{eqnarray*}
& & \left\Vert \nabla_x^{k_1} \nabla_y^{*k_2} \left( B_{V_{\sqrt{r_m } } ,\omega ,m } (x,y) - B_{V_{\sqrt{r_m } } , \omega_{V} ,m } (x,y) \right) \left( x_m ,y_m \right) \left( v_{ 1,m} ,\cdots ,v_{k_1 +k_2,m}  \right) \right\Vert_{h_V^m } \\
& = & \left\Vert \left( T_{1,m} \left( S_{T_{1,m}} \right) \overline{T_{2,m} \left( S_{T_{1,m}} \right) } -T_{1,m} \left( S'_{T_{1,m }} \right) \overline{T_{2,m} \left( S'_{T_{1,m }} \right) } \right) e_{x_m} \otimes e_{y_m}^* \right\Vert_{h_V^m } \\
& \leq & \left| T_{1,m} \left( S_{T_{1,m}} - S'_{T_{1,m }} \right) \right| \cdot \left| T_{2,m} \left( S'_{T_{1,m }} \right)  \right| + \left| T_{1,m} \left( S_{T_{1,m}} \right) \right| \cdot \left| T_{2,m} \left( S_{T_{1,m }} - S'_{T_{1,m }} \right)  \right| .
\end{eqnarray*}
By Corollary \ref{corocoarseestimatecpxhpblccusp}, we see that there exists a constant $C_1 >0$ independent of $m $, such that 
$$ \left| \log h_D \left( x_m \right) \right|^{2\beta} \left\Vert T_{1,m} \big|_{\hl \left( V_{\sqrt{r_m} },\linebundle_V^m \right) } \right\Vert + \left| \log h_D \left( y_m \right) \right|^{2\beta} \left\Vert T_{2,m} \big|_{\hl \left( V_{\sqrt{r_m} },\linebundle_V^m \right) } \right\Vert \leq C_1 m^{C_1} ,\;\; \forall m\in\mathbb{N} ,$$ 
where $\epsilon >0$ is a constant independent of $m$. The norm we used here is the norm associated with the metric $\omega_{V}$ and the Hermitian metric $h^m_{V}$. Combining Fu-Hein-Jiang's estimate (\ref{fhjasymptoticbehavior}) with Lemma \ref{lmmdistancedifferencepeaksection}, we see that for each given integer $q\in\mathbb{N}$, there exists a constant $C_2 >0$ depending only on $V,r $ and $q $, such that 
$$ \int_{ V_{ \sqrt{r_m} } } \left\Vert S_{T_{1,m}} - S'_{T_{1,m }} \right\Vert_{h^m_{V}  }^2 d\V_{\omega_{V}} \leq C_2 m^{-q} ,\;\; \forall m\in\mathbb{N} .$$
Choosing $q= 2l+4+4\lfloor C_2 \rfloor $, where $\lfloor \cdot \rfloor$ is the greatest integer function. Then we have
\begin{eqnarray*}
& & \left| T_{1,m} \left( S_{T_{1,m}} - S'_{T_{1,m }} \right) \right| \cdot \left| T_{2,m} \left( S'_{T_{1,m }} \right)  \right| + \left| T_{1,m} \left( S_{T_{1,m}} \right) \right| \cdot \left| T_{2,m} \left( S_{T_{1,m }} - S'_{T_{1,m }} \right)  \right| \\
& \leq & 2 \left\Vert T_{1,m} \big|_{\hl \left( V_{\sqrt{r_m} },\linebundle^m \right) } \right\Vert \cdot \left\Vert T_{2,m} \big|_{\hl \left( V_{\sqrt{r_m} },\linebundle^m \right) } \right\Vert \cdot \left\Vert S_{T_{1,m}} - S'_{T_{1,m }} \right\Vert_{L^2 ; V_{ \sqrt{r_m } } ,h_{V}^m ,\omega_{V} } \\
& \leq & 2C_1^2 C_2 m^{2C_0} m^{-l-2-2\lfloor C_2 \rfloor } \left| \log h_D \left( x_m \right) \right|^{-\beta} \left| \log h_D \left( y_m \right) \right|^{-\beta} \\
& \leq & 2C_1^2 C_2 m^{-l} \left| \log h_D \left( x_m \right) \right|^{-\beta} \left| \log h_D \left( y_m \right) \right|^{-\beta}.
\end{eqnarray*}
Then we can obtain the estimate (\ref{inequalitydatarfusongmanifold2}).

\smallskip

\par {\em Part 3.}
In this part, we prove the inequality (\ref{inequalitydatarfusongmanifold3}). 

For any point $x\in V_{r } $, we choose a vector $e_{x} \in \linebundle_V \big|_x $ such that $ \left\Vert e_x \right\Vert_{h_V } =1 $. Hence we can define bounded linear functionals $T_{x,m} \in \hl \left( V_{r_m } ,\linebundle_{V}^m \right)^* $ by $ S\big|_{x} = T_{x,m} \left( S \right) e^m_{x} $, $\forall S\in\hl \left( V_{r_m } ,\linebundle_{V}^m \right) $, where $m>0 $ is an integer such that $r_m \leq r^4 $, and $x\in V_{r_m }$. For any open subset $U'$ of $V $ containing $x$, we denote the peak sections of $T_{x,m}$ on $ \hl \left( U',\omega_V ,\linebundle_V^m ,h_V^m \right) $ as $S'_{U',x,m}$. Similarly, for any open subset $U'$ of $M $ containing $x$, we will denote by $S_{U',x,m}$ the peak sections of $T_{x,m}$ on $ \hl \left( U',\omega ,\linebundle^m ,h^m  \right) $. It is easy to see that 
$$ \frac{\rho_{M ,\omega ,m} (x)}{\rho_{V, \omega_{V} ,m} \left( x \right) } = \left| \frac{T_{x,m} \left( S_{M ,x,m} \right)}{T_{x,m} \left( S'_{V ,x,m} \right) } \right|^2 e^{mu } .$$

By definition, there are constants $m_2 ,C_3 >0$ depending only on $V,r$ and $l$, such that
$$ \sup_{V_{r_m} } \left| e^{mu } -1 \right| \leq C_3 m^{-l},\;\; \forall m\geq m_2 .$$

Applying Proposition \ref{propcoraseestimate} to $ T_{x,m} \left( S_{M ,x,m} \right) $ and $ T_{x,m} \left( S'_{V ,x,m} \right) $, one can see that there are constants $m_3 ,C_4 >0$ depending only on $V,r,f$ and $l$, such that
$$ \sup_{x\in V_{r_m} } \left( \left| \frac{T_{x,m} \left( S_{M ,x,m} \right) }{T_{x,m} \left( S_{V_{\sqrt{r_m}} ,x,m} \right) } -1 \right| + \left| \frac{T_{x,m} \left( S'_{V ,x,m} \right) }{T_{x,m} \left( S'_{V_{\sqrt{r_m}} ,x,m} \right) } -1 \right| \right) \leq C_4 m^{-l},\;\; \forall m\geq m_3 .$$

By definition, $ \omega = \omega_V -\sqrt{-1} \partial\partialbar u $, and $h = e^{u} h_V $. Lemma \ref{lmmdistancedifferencepeaksection} now implies that that there are constants $m_4 ,C_5 >0$ depending only on $V,r$ and $l$, such that
$$ \sup_{x\in V_{r_m} } \left| \frac{T_{x,m} \left( S_{V_{\sqrt{r_m}} ,x,m} \right) }{T_{x,m} \left( S'_{V_{\sqrt{r_m}} ,x,m} \right) } -1 \right| \leq C_5 m^{-l},\;\; \forall m\geq m_4 .$$
It follows that
\begin{eqnarray*}
& & \sup_{x\in V_{r_m} } \left| \frac{\rho_{M ,\omega ,m} (x)}{\rho_{V, \omega_{V} ,m} \left( x \right) } -1 \right| = \sup_{x\in V_{r_m} } \left| \left| \frac{T_{x,m} \left( S_{M ,x,m} \right)}{T_{x,m} \left( S'_{V ,x,m} \right) } \right|^2 e^{mu } -1 \right| \\
& = & \sup_{x\in V_{r_m} } \left| \left| \frac{T_{x,m} \left( S_{M ,x,m} \right) }{T_{x,m} \left( S_{V_{\sqrt{r_m}} ,x,m} \right) }  \right|^2 \cdot \left| \frac{T_{x,m} \left( S_{V_{\sqrt{r_m}} ,x,m} \right) }{ T_{x,m} \left( S'_{V_{\sqrt{r_m}} ,x,m} \right) }  \right|^2 \cdot \left| \frac{ T_{x,m} \left( S'_{V_{\sqrt{r_m}} ,x,m} \right) }{ T_{x,m} \left( S'_{V ,x,m} \right) }  \right|^2 \cdot e^{mu } -1 \right| \\
& \leq & \left( 1+2C_4 m^{-l} \right)^4 \left( 1+2C_5 m^{-l} \right)^2 \left( 1+2C_3 m^{-l} \right) -1  \leq \left( 1+2C_4 \right)^4 \left( 1+2C_5 \right)^2 \left( 1+2C_3 \right) m^{-l} ,
\end{eqnarray*}
for any integer $m\geq \sum_{j=2}^4 \left( m_j +10 C_{j+1} \right) $.
\qed

\smallskip

Before considering the $C^k$-estimate of the quotients of Bergman kernel functions, we need the following lemma.

\begin{lmm}
\label{lmmdatarfusongmanifold2}
Let $ \left( V,\omega_V ,\linebundle_V ,h_V \right) $ be a complex hyperbolic cusp defined in Section \ref{intro}, and $S\in \hl \left( V,\linebundle_V \right) $ such that $ \left\Vert S \right\Vert_{L^2 ; h_V^m , \omega_V }^2 =1 $. Then for each $k\in\mathbb{N}$, there exists a constant $C>0$ depending only on $\left( V,\omega_V ,\linebundle_V ,h_V \right) $ and $k$, such that
\begin{eqnarray}
\left\Vert S (x) \right\Vert_{C^k ; h^m_V , \omega_V }^2  \leq Cm^C \left| \log h_D (x) \right|^C \rho_{\mathbb{D}^*_1 ,m} \left( \sqrt{ h_D (x) } \right) ,
\end{eqnarray}
for any $x\in V_{e^{-2}} $ and integer $m\geq C$, where $\rho_{\mathbb{D}^*_1 ,m} $ is $m$-th Bergman kernel on the polarized punctured unit disc $ \left( \mathbb{D}_1^* ,\omega_{\mathbb{D}^*} ,\mathbb{C} , \left| \log \left| z \right|^2 \right| \right) $.
\end{lmm}

\begin{proof}
Combining Lemma \ref{lmminjectivityradiushyperbpliccusp} with Lemma \ref{lmmderiholosections}, we can find a constant $C_1 >0$ depending only on $\left( V,\omega_V ,\linebundle_V ,h_V \right) $ and $k$, such that
$$ \left\Vert \nabla^k S(x) \right\Vert_{C^k ; h^m_V , \omega_V }^2 \leq C_1 m^{C_1} \left| \log h_D (x) \right|^{C_1 } \int_{B_{C_1^{-1} \left| \log h_D (x) \right|^{-10} } (x) } \rho_{\omega_V ,m} (y) d\V_{\omega_V} ,$$
for any $x\in V_{e^{-2}} $ and integer $m\geq C_1 $. By a direct computation, we can conclude that there exists a constant $C_2 >0$ depending only on $\left( V,\omega_V ,\linebundle_V ,h_V \right) $, such that
$$ B_{C_1^{-1} \left| \log h_D (x) \right|^{-10} } (x) \subset \left\lbrace  y\in V \bigg| h_D (x)^{1+C_2 \left| \log h_D (x) \right|^{-10} } \leq h_D (y) \leq h_D (x)^{1-C_2 \left| \log h_D (x) \right|^{-10} }  \right\rbrace , $$
for any $ x\in V_{e^{-2} } $. Combining Theorem {thmregularpart} with Lemma \ref{lmminjectivityradiushyperbpliccusp}, we see that there exists a constant $\epsilon \in \left( 0,e^{-1} \right) $ depending only on $\left( V,\omega_V ,\linebundle_V ,h_V \right) $, such that
$$ \frac{(m+n)!}{(2\pi )^{n+1} m! } \leq \rho_{\omega_V ,m} (x) \leq \frac{(m+n)!}{(2\pi )^{n-1} m! }  ,$$
for any $x\in V_{e^{-2}} \sq V_{e^{-\sqrt{\epsilon m} }} $ and integer $m\geq \epsilon^{-1} $. Hence there exists a constant $C_3 >0$ depending only on $\left( V,\omega_V ,\linebundle_V ,h_V \right) $, such that
$$ \left\Vert \nabla^k S(x) \right\Vert_{C^k ; h^m_V , \omega_V }^2 \leq C_3 m^{C_3 } \left| \log h_D (x) \right|^{C_3 } \leq C^3_2 m^{2C_3 } \left| \log h_D (x) \right|^{2C_3 } \rho_{\mathbb{D}^*_1 ,m} \left( \sqrt{ h_D (x) } \right) ,$$
for any $x\in V_{e^{-2}} \sq V_{e^{-\epsilon \sqrt{ m} }} $ and integer $m\geq C_3 $.

Now we will give an estimate of $ \rho_{\omega_V ,m} (y) $ for $ y\in B_{C_1^{-1} \left| \log h_D (x) \right|^{-10} } (x) $ and $x\in V_{e^{-\epsilon \sqrt{ m} }} $. By Proposition \ref{propbergmanknv}, we can conclude that
\begin{eqnarray}
\rho_{\omega_V ,m} (y) & = & \frac{n\left| \log h_D (y) \right|^m }{2\pi (m-n-1)! } \sum_{l=1}^\infty l^{m-n} \rho_{\omega_D ,l} \left( \pi_D (y) \right) h_D (y)^l \nonumber \\
& \leq & C_4 \frac{n\left| \log h_D (x) \right|^m \left( 1+C_2 \left| \log h_D (x) \right|^{-10} \right)^m }{2\pi (m-n-1)! } \sum_{l=1}^\infty l^{m-1} h_D (x)^{l \left( 1-C_2 \left| \log h_D (x) \right|^{-10} \right) } \\
& \leq & C^2_4 \frac{n\left| \log h_D (x) \right|^m }{2\pi (m-n-1)! } \sum_{l=1}^{\infty } l^{m-1} h_D (x)^{l \left( 1-C_2 \left| \log h_D (x) \right|^{-10} \right) } \nonumber \\
& \leq & C^3_4 m^n \rho_{\mathbb{D}^*_1 ,m} \left( \sqrt{h_D (x)} \right), \nonumber
\end{eqnarray}
for any $m\geq C_4 $, where $C_4 >0$ is a constant depending only on $\left( V,\omega_V ,\linebundle_V ,h_V \right) $. Combining these inequalities, we can prove this lemma.
\end{proof}

Now we need to estimate the $C^k $ norm on $\left( V ,\omega_V +\pi^*_D \omega_D \right) $. Let $\dya \in D $ and
 $z=\left( z_1 ,\cdots ,z_{n-1} \right) :U_\dya \to \mathbb{C}^{n-1} $ be a holomorphic coordinate around $\dya \in D $ such that $z ( \dya ) =0 $ and $ \omega_D = z^* \omega_{Euc} $. Fix a holomorphic section $e_{\linebundle_D } \in H^0 \left( U_\dya ,\linebundle_D \right) $ satisfying that $ \left\Vert e_{\linebundle_D } \right\Vert_{h_D}^2 = e^{|z|^2 } $ on $U_\dya $. 
Now we can construct a holomorphic map $\varphi : U_\dya \times \mathbb{R}_{+} \times \mathbb{R} \to \pi_D^{-1} \left( U_\dya \right) $ by $ \varphi (z,r,\theta ) = re^{\sqrt{-1} \theta } e_{\linebundle_D } (z) $, $\forall (z,r,\theta ) \in U_\dya \times \mathbb{R}_{+} \times \mathbb{R} $. It is clear that $ rank_{(z,r,x)} \left( d\varphi \right) = 2n $, $\forall (z,r,\theta ) \in U_\dya \times \mathbb{R}_{+} \times \mathbb{R} $, and
$$ \varphi^{-1} \left( \pi_D^{-1} \left( U_\dya \right) \cap V \right) = \left\lbrace (z,r,\theta ) \in U_\dya \times \mathbb{R}_{+} \times \mathbb{R} \big| r^2 < e^{-|z|^2 } \right\rbrace . $$
Write $x_{2i-1} = \re z_i $ and $x_{2i} =\im z_i $, $i=1,\cdots ,n-1$. Let
\begin{eqnarray}
\mathcal{X}_j & = & \left| \log h_D \right|^{ \frac{1}{2}} \mathcal{Y}_j = \mathcal{Z}_j = \frac{\partial}{\partial x_{j}} ,\;\; j=1,\cdots ,2n-2 ,\nonumber \\
\mathcal{X}_{2n-1} & = & \mathcal{Y}_{2n-1} =  r\left| |z|^2 +\log r^2 \right| \mathcal{Z}_{2n-1} = r\left| |z|^2 +\log r^2 \right| \frac{\partial}{\partial r} , \label{vectorfieldscusp} \\
\mathcal{X}_{2n } & = & \mathcal{Y}_{2n } =  r\left| |z|^2 +\log r^2 \right| \mathcal{Z}_{2n } = \left| |z|^2 +\log r^2 \right| \frac{\partial}{\partial \theta } \nonumber .
\end{eqnarray}
Then we can obtain the following estimate.

\begin{lmm}
\label{lmmgradientomegaomega+D}
Under the above assumptions, for any given $k\in\mathbb{N} $, there exists a constant $C>0$ depending only on $\left( V,\omega_V \right) $ and $U_{\dya} $, such that for any point $x\in V_{e^{-2}} \cap \pi_D^{-1} \left( U_\dya \right) $ and function $f\in C^k \left( V;\mathbb{R} \right) $, we have 
\begin{eqnarray}
\left\Vert \nabla^k f(x) \right\Vert_{\omega_V +\pi^*_{D} \omega_D } & \leq & C \left| \log h_D (x) \right|^C \sum_{l=0}^k \sum_{1\leq i_1 , \cdots , i_l \leq 2n } \left| \mathcal{X}_{i_1} \left( \cdots \left( \mathcal{X}_{i_l} \left( f \right)  \right) \cdots \right) (x) \right| \\
& \leq & C^2 \left| \log h_D (x) \right|^{2 C } \sum_{p=0}^k \left\Vert \nabla^p f(x) \right\Vert_{\omega_V +\pi^*_{D} \omega_D } , \nonumber
\end{eqnarray}
\begin{eqnarray}
& & \left\Vert \nabla^k f(x) \right\Vert_{\omega_V +\pi^*_{D} \omega_D } + \left\Vert \nabla^k f(x) \right\Vert_{\omega_V } \nonumber \\
& \leq & C \left| \log h_D (x) \right|^C \sum_{p=0}^k \sum_{1\leq i_1 , \cdots , i_p \leq 2n } \left| \mathcal{Y}_{i_1} \left( \cdots \left( \mathcal{Y}_{i_p} \left( f \right)  \right) \cdots \right) (x) \right| \\
& \leq & C^2 \left| \log h_D (x) \right|^{2C} \sum_{p=0}^k \left\Vert \nabla^p f(x) \right\Vert_{\omega_V } , \nonumber
\end{eqnarray}
and
\begin{eqnarray}
\sum_{p=0}^k \left\Vert \nabla^p f(x) \right\Vert_{ \pi^*_{D} \omega_D +\sqrt{-1} \partial\partialbar h_D } & \leq & C \sum_{p=0}^k \sum_{1\leq i_1 , \cdots , i_p \leq 2n } \left| \mathcal{Z}_{i_1} \left( \cdots \left( \mathcal{Z}_{i_p} \left( f \right)  \right) \cdots \right) (x) \right| \\
& \leq & C^2 \sum_{p=0}^k \left| h_D (x) \right|^{-C} \left\Vert \nabla^p f(x) \right\Vert_{\omega_V } . \nonumber
\end{eqnarray}

Moreover, if $f=f_D \circ \pi_D $ for some function $f_D \in C^k \left( D;\mathbb{R} \right) $, then
\begin{eqnarray}
\left\Vert \nabla^k f(x) \right\Vert_{\omega_V +\pi^*_{D} \omega_D } \leq C \sum_{l=0}^k \left\Vert \nabla^l f_D \left( \pi_D (x) \right) \right\Vert_{ \omega_D } .
\end{eqnarray}
\end{lmm}

The proof of this lemma is not difficult but is too long to give it.

Then we can give an estimate for the $C^k $-norm of the quotient of Bergman kernel functions on $V_{r_m }$ with respect to the K\"ahler metric $\omega_V $.

\begin{prop}
\label{propestimatequotientCkdatarfusong1}
Let $(M,\omega ,\linebundle ,h )$ and $U\cong V_r $ be as in Theorem \ref{thmdatarfusongmanifold}. Assume that $\Gamma_V =0$ and $\linebundle_D^{-1} $ is globally generated. Then there exists a constant $b >0 $, such that for any given integer $l\geq 0$,
\begin{eqnarray}
 \left\Vert \frac{\rho_{M ,\omega ,m} (x)}{\rho_{V, \omega_{V} ,m} \left( x \right) } -1 \right\Vert_{C^k ;\omega_{V} } & \leq & C \left| \log h_D \right|^{b} m^{-l}  ,\;\; \forall m\geq C,\;\; \forall x\in V_{r_m} , \label{inequalitypropestimatequotientCkdatarfusong1}
\end{eqnarray}
where $C >0$ is a constant depending only on $(M,\omega ,\linebundle ,h)$, $V$, $k$ and $l$.
\end{prop}

\begin{proof}
By choosing open subsets $U_\dya $ of $D $ as stated before Lemma \ref{lmmgradientomegaomega+D}, we can find a finite covering $ \left\lbrace \pi_D^{-1} \left( U_\dya \right) \right\rbrace_{\dya \in \mathfrak{A}} $ of $V$. Fix $x \in V_{r_m }$. Then there are vector fields $\mathcal{X}_i $ and $\mathcal{Y}_i $ around $x$ satisfying the estimates in Lemma \ref{lmmgradientomegaomega+D}. Let $ \left\lbrace S_{M,m,j} \right\rbrace_{j\in I_m } $ and $ \left\lbrace S'_{M,m,j} \right\rbrace_{j\in I'_m } $ be $L^2 $ orthonormal bases in the spaces $\hl \left( M, \omega , \linebundle^m ,h^m \right) $ and $\hl \left( V, \omega_V , \linebundle_V^m ,h_V^m \right) $, respectively. Then we have
\begin{eqnarray}
\frac{\rho_{M,\omega ,m} (x) }{\rho_{V ,\omega_V ,m} (x)} = \frac{\sum_{j\in I_m} \left\Vert S_{M,m,j} (x) \right\Vert_{h^m}^2 }{\sum_{j\in I'_m} \left\Vert S'_{M,m,j} (x) \right\Vert_{h_V^m}^2} = e^{mu} \frac{\sum_{j\in I_m} \left\Vert S_{M,m,j} (x) \right\Vert_{h_V^m}^2 }{\sum_{j\in I'_m} \left\Vert S'_{M,m,j} (x) \right\Vert_{h_V^m}^2 } ,\; \forall x\in U\cong V_r .
\end{eqnarray}

By Lemma \ref{lmmgradientomegaomega+D}, we can conclude that for any constant $l>0$, there exists a constant $C_1 >0$ depending only on $\left(M,\omega ,\linebundle ,h \right) $, $V$, $l$ and $k$, such that
$$ \sum_{p=0}^k \sup_{V_{r_m} } \left\Vert \nabla^p \left( e^{mu} -1 \right) (x) \right\Vert_{\omega_V } \leq C_1 m^{-l} ,\;\; \forall m\geq C_1 . $$

For abbreviation, we write $\mathfrak{Y}_{i_1 ,\cdots ,i_p }$ instead of $\mathcal{Y}_{i_{1} } \left( \cdots \left( \mathcal{Y}_{i_{p}} \left( \sum_{j\in I_m} \left\Vert S_{M,m,j} \right\Vert^{2}_{h_V^m} \right) \right) \cdots \right) (x) $, and we write $\mathfrak{Y}'_{i_1 ,\cdots ,i_p }$ instead of $\mathcal{Y}_{i_{1} } \left( \cdots \left( \mathcal{Y}_{i_{p}} \left( \sum_{j\in I'_m} \left\Vert S'_{M,m,j} \right\Vert^{2}_{h_V^m} \right) \right) \cdots \right) (x) $. By induction on $k$, for any integer $k\geq 1$, one can conclude that
\begin{eqnarray}
& & \sum_{1\leq i_1 , \cdots , i_k \leq 2n } \mathcal{Y}_{i_1} \left( \cdots \left( \mathcal{Y}_{i_k} \left( \frac{\sum_{j\in I_m} \left\Vert S_{M,m,j}  \right\Vert_{h_V^m}^2 }{\sum_{j\in I'_m} \left\Vert S'_{M,m,j} \right\Vert_{h_V^m}^2 } \right)  \right) \cdots \right) (x) \\
& = & \sum_{p=0}^k \sum_{ 1\leq i_1 , \cdots , i_p \leq 2n } \left(  \sum_{j\in I'_m} \left\Vert S'_{M,m,j} (x) \right\Vert^{2 }_{h_V^m} \right)^{-k} P_{ k,p } \cdot \left( \mathfrak{Y}_{i_1 ,\cdots ,i_p} - \widetilde{\mathfrak{Y}}_{i_1 ,\cdots ,i_p} \right) \nonumber
\end{eqnarray}
where $p\leq k $ is a nonnegative integer, $P_{ k,p} $ are polynomials in variables $\mathfrak{Y}_{i_1 ,\cdots ,i_p} $, $ \mathfrak{Y}'_{i_1 ,\cdots ,i_p}$ and of degree $k-1$, and $1\leq i_1 ,\cdots , i_p \leq 2n $ are integers. Since $\linebundle_D^{-1} $ is globally generated, we can find a constant $C_2 >0$ depending only on $V$, such that $ \rho_{\mathbb{D}_1^* ,m} \left( \sqrt{h_D (x)} \right) \leq C_2 \rho_{V ,\omega_V ,m} (x) $, $ \forall x \in V $, $\forall m\geq n+1$. By Lemma \ref{lmmdatarfusongmanifold2}, we see that for any integer $k\in\mathbb{N}$ and point $x\in V_{e^{-2}}$,
$$ \left\Vert \nabla^k \left( \frac{\sum_{j\in I_m} \left\Vert S_{M,m,j} \right\Vert_{h_V^m}^2 }{\sum_{j\in I'_m} \left\Vert S'_{M,m,j}  \right\Vert_{h_V^m}^2 } -1 \right) \right\Vert_{\omega_V } \leq C_3 m^{C_3} \left| \log h_D \right|^{C_3 } \sum_{p=0}^k \sum_{ 1\leq i_1 , \cdots , i_p \leq 2n } \frac{ \left| \mathfrak{Y}_{i_1 ,\cdots ,i_p} - \mathfrak{Y}'_{i_1 ,\cdots ,i_p} \right| }{\sum_{j\in I'_m} \left\Vert S'_{M,m,j}  \right\Vert^{2k }_{h_V^m}} , $$
for any $ m\geq C_3 $, where $C_3 >0$ is a constant depending only on $(M,\omega ,\linebundle ,h)$, $V$ and $k$. Now we estimate the function $\left| \mathfrak{Y}_{i_1 ,\cdots ,i_p} - \mathfrak{Y}'_{i_1 ,\cdots ,i_p} \right| $. Since Chern connection is compatible with the Hermitian metric $h^m_V $, we have
\begin{eqnarray}
& & \sum_{p=0}^k \sum_{ 1\leq i_1 , \cdots , i_p \leq 2n } \left| \mathfrak{Y}_{i_1 ,\cdots ,i_p} - \mathfrak{Y}'_{i_1 ,\cdots ,i_p} \right| \nonumber \\
& \leq & \sum_{p=0}^k \sum_{p' =0 }^p \sum_{ 1\leq i_1 , \cdots , i_p \leq 2n } \left|  \sum_{j\in I_m } \left\langle \nabla_{\mathcal{Y}_{i_1} } \cdots \nabla_{\mathcal{Y}_{i_{p'}} } S_{M,m,j} , \nabla_{\mathcal{Y}_{i_{p' +1}} } \cdots \nabla_{\mathcal{Y}_{i_{p }} } S_{M,m,j} \right\rangle_{h_V^m} \right. \\
& &\;\;\;\;\;\;\;\;\;\;\;\;\;\;\;\;\;\;\;\;\;\;\;\;\;\;\;\;\;\;\;\;\;\;\; \left. - \sum_{j\in I'_m } \left\langle \nabla_{\mathcal{Y}_{i_1} } \cdots \nabla_{\mathcal{Y}_{i_{p'}} } S'_{M,m,j} , \nabla_{\mathcal{Y}_{i_{p' +1}} } \cdots \nabla_{\mathcal{Y}_{i_{p }} } S'_{M,m,j} \right\rangle_{h_V^m} \right| . \nonumber
\end{eqnarray}
For any point $x\in U\cong V_r $ and finite sequence of integers $1\leq i_1 ,\cdots ,i_p \leq 2n $, we can define a linear functional $T_{x,m,i_1 ,\cdots ,i_p } $ on $ \hl \left( V_r , \linebundle_V^m \right) $ by 
$$ \nabla_{\mathcal{Y}_{i_1} } \cdots \nabla_{\mathcal{Y}_{i_{p }} } S (x) = T_{x,m,i_1 ,\cdots ,i_p } (S) e^m_x , \;\forall S\in \hl \left( V_r , \linebundle_V^m \right) , $$
where $e_x \in \linebundle_V |_{x} $ is a vector such that $\left\Vert e_x \right\Vert_{h_V} =1 $. Let $S_{x,m,i_1 ,\cdots ,i_p }$ and $ S'_{x,m,i_1 ,\cdots ,i_p } $ be the peak sections of $T_{x,m,i_1 ,\cdots ,i_p } $ in $\hl \left( M, \omega , \linebundle^m ,h^m \right) $ and $\hl \left( V, \omega_V , \linebundle_V^m ,h_V^m \right) $, respectively. Then we have
\begin{eqnarray}
& & \left|  \sum_{j\in I_m } \left\langle \nabla_{\mathcal{Y}_{i_1} } \cdots \nabla_{\mathcal{Y}_{i_{p'}} } S_{M,m,j} , \nabla_{\mathcal{Y}_{i_{p' +1}} } \cdots \nabla_{\mathcal{Y}_{i_{p }} } S_{M,m,j} \right\rangle_{h_V^m} \right. \nonumber \\
& &\;\;\;\;\;\;\;\;\;\;\;\;\;\;\;\;\;\;\;\;\;\;\;\;\;\;\;\;\;\;\;\;\;\;\; \left. - \sum_{j\in I'_m } \left\langle \nabla_{\mathcal{Y}_{i_1} } \cdots \nabla_{\mathcal{Y}_{i_{p'}} } S'_{M,m,j} , \nabla_{\mathcal{Y}_{i_{p' +1}} } \cdots \nabla_{\mathcal{Y}_{i_{p }} } S'_{M,m,j} \right\rangle_{h_V^m} \right| \\
& = & \left| \left\langle T_{x,m,i_1 ,\cdots ,i_{p'} } \left( S_{x,m,i_1 ,\cdots ,i_{p'} } - S'_{x,m,i_1 ,\cdots ,i_{p'} } \right) e_x^m , T_{x,m,i_{p' +1} ,\cdots ,i_{p'} } \left( S_{x,m,i_1 ,\cdots ,i_{p'} } \right) e_x^m \right\rangle_{h_V^m} \right. \nonumber \\
& & \;\;\;\; \left. - \left\langle T_{x,m,i_1 ,\cdots ,i_{p'} } \left( S'_{x,m,i_1 ,\cdots ,i_{p'} } \right) e_x^m , T_{x,m,i_{p' +1} ,\cdots ,i_{p'} } \left( S_{x,m,i_1 ,\cdots ,i_{p'} } -S'_{x,m,i_1 ,\cdots ,i_{p'} } \right) e_x^m \right\rangle_{h_V^m} \right| \nonumber \\
& \leq & 2 \left\Vert T_{x,m,i_1 ,\cdots ,i_{p'} } \right\Vert \cdot \left\Vert T_{x,m,i_{p' +1} ,\cdots ,i_{p'} } \right\Vert \cdot \left\Vert S_{x,m,i_1 ,\cdots ,i_{p'} } - S'_{x,m,i_1 ,\cdots ,i_{p'} } \right\Vert , \nonumber
\end{eqnarray}
where $\left\Vert \cdot \right\Vert$ are norms corresponding to the Hilbert space $ \hl \left( V_{\sqrt{r_m} } ,\omega_V ,\linebundle_V^m ,h^m_V \right) $. Combining Lemma \ref{peaksectionrestriction}, Proposition \ref{propcoraseestimate} and Lemma \ref{peaksecperturbedmanifoldssubsection}, we can assert that for any integer $l>0$, there exists a constant $C_4 >0 $ depending only on $(M,\omega ,\linebundle ,h)$, $V$, $k$ and $l$, such that
\begin{eqnarray}
\int_{V_{\sqrt{r_m }} } \left\Vert S_{x,m,i_1 ,\cdots ,i_p } - S'_{x,m,i_1 ,\cdots ,i_p } \right\Vert_{h_V^m}^2 d\V_{\omega_V} \leq C_4 m^{-l} ,\;\; \forall m\geq C_4 ,\forall x\in V_{r_m } .
\end{eqnarray}
Hence for any integers $k,l\in\mathbb{N}$,
$$ \left\Vert \nabla^k \left( \frac{\sum_{j\in I_m} \left\Vert S_{M,m,j} \right\Vert_{h_V^m}^2 }{\sum_{j\in I'_m} \left\Vert S'_{M,m,j} \right\Vert_{h_V^m}^2 } -1 \right) \right\Vert_{\omega_V } \leq C_6 m^{-l} \left| \log h_D \right|^{C_5 } ,\;\; \forall m\geq C_6 ,\forall x\in V_{r_m } ,$$
where $C_5 >0$ is a constant depending only on $(M,\omega ,\linebundle ,h)$, $V$ and $k$, and $C_6 >0$ is a constant depending only on $(M,\omega ,\linebundle ,h)$, $V$, $k$ and $l$. This is our claim.
\end{proof}

Now we are ready to prove the $C^k$-estimate of the quotients of Bergman kernel functions.

\begin{prop}
\label{propdatarfusongmanifold2}
Let $(M,\omega ,\linebundle ,h )$ and $U\cong V_r $ be as in Theorem \ref{thmdatarfusongmanifold}. Assume that $\Gamma_V =0$ and $\linebundle_D^{-1} $ is globally generated. Then for any given integer $l\geq 0$, there exists a constant $C >0$ depending only on $(M,\omega ,\linebundle ,h)$, $V$, $k$ and $l$, such that
\begin{eqnarray}
 \sup_{V_{r_m } } \left\Vert \frac{\rho_{M ,\omega ,m} (x)}{\rho_{V, \omega_{V} ,m} \left( x \right) } -1 \right\Vert_{C^k ;\omega_{V} + \pi_D^* \omega_D } & \leq & C m^{-l} ,\;\; \forall m\geq C. \label{inequalitydatarfusongmanifold4}
\end{eqnarray}
\end{prop}

\begin{proof}
By Proposition \ref{propestimatequotientCkdatarfusong1}, we only need to prove it when $x \in V_{e^{-m^{10}}} $. Without loss of generality, we can assume that $m $ is sufficiently large such that $r_m <r^{2} $. 

Let $X_m =\hl \left( V_{r^4_m } ,\omega_V ,\linebundle_V^m ,h_V^m \right)_{2} $ be the subspace of $\hl \left( V_{r^4_m } ,\omega_V ,\linebundle_V^m ,h_V^m \right) $ generated by the orthogonal set $\left\lbrace  f_{S_{q,j}} e^m_{\linebundle_V} \bigg| \; 1\leq j \leq N_q ,\; q\geq 2 \right\rbrace ,$ where $\left\lbrace S_{q,j} \right\rbrace_{j=1}^{N_q} $ is an orthonormal basis of the Hilbert space $\hl \left( D,\omega_D ,\linebundle^{-q}_D ,h_D^{-q} \right) $, and $f_{S_{q,j}} $ denote the holomorphic function corresponding to $S_{q,j} $ defined in Lemma \ref{holoontotal}. Then we define $Y_m =X_m \cap \hl \left( M,\omega ,\linebundle^m ,h^m \right) $, and $Z_m =X_m \cap \hl \left( V ,\omega_V ,\linebundle_V^m ,h_V^m \right) $. Note that $\left( V_r ,\linebundle_V \right) \cong \left( U ,\linebundle \right) $, and hence $\hl \left( V_r ,\omega_V ,\linebundle_V^m ,h_V^m \right) = \hl \left( U ,\omega ,\linebundle^m ,h^m \right) $ as subspaces of $ H^0 \left( V_r ,\linebundle_V^m \right) \cong H^0 \left( U ,\linebundle^m \right) $, $ \forall m\in\mathbb{N} $. Let $X_m^{ \perp } $, $Y_m^{ \perp } $ and $Z_m^{ \perp } $ be the orthogonal complements of $X_m $, $Y_m $ and $Z_m  $ in $ \hl \left( V_{r_m } ,\omega_V ,\linebundle_V^m ,h_V^m \right) $, $ \hl \left( M,\omega ,\linebundle^m ,h^m \right) $ and $ \hl \left( V ,\omega_V ,\linebundle_V^m ,h_V^m \right) $, respectively. By applying Lemma \ref{lmmorthonormalfver} to $\hl \left( V,\linebundle^m_V \right) $, $\forall m\geq n+1 $, we see that $\left\lbrace  f_{S_{q,j}} e^m_{\linebundle_V} \bigg| \; 1\leq j \leq N_q \right\rbrace $ is an orthogonal basis of $X_m^{ \perp } $ and $Z_m^{ \perp } $ at the same time. Let $a_{m} = \left( \frac{n }{ 2\pi (m-n-1)! } \right)^{\frac{1}{2}} $. It is easy to see that for any integer $l >0 $, we have
\begin{eqnarray}
\left| \int_{V_{r_m^4}} \left\Vert a_{m} f_{S_{1,j}} e^m_{\linebundle_V}  \right\Vert_{h_V^m}^2 d\V_{\omega_V} -1 \right| +  \sup_{V_{r_m } \sq V_{r_m^4} } \left\Vert a_{m} f_{S_{1,j}} e^m_{\linebundle_V}  \right\Vert_{h_V^m}^2 \leq C_1 m^{-l } ,
\end{eqnarray} 
where $C_1 >0 $ is a constant depending only on $V $ and $l$. 

Analysis similar to that in the proof of Lemma \ref{lmmquasipshhyperboliccuspver2} shows that there are quasi-plurisubharmonic functions $\psi_m \in L^1 \left( V,\omega_V \right) $ such that $\psi_m \leq 0 $, $\mathrm{supp} \psi_m \subset V_{r_m^4} $, $\sqrt{-1} \partial\partialbar \psi_m \geq -C_2   \sqrt{m} \omega_V $, and $ \lim_{h_D (x) \to 0} \left( \psi_m (x) - \log \left( h_D (x) \right) \right) = -\infty $, where $C_2 >0$ is a constant dependent only on $V $. See also the remark below Lemma \ref{lmmquasipshhyperboliccuspver2}. Let $ \eta_{m} \in C^\infty \left( V \right) $ such that $ \eta_m =1 $ on $V_{r_m^4 } $, $ \eta_m =0 $ on $V \sq V_{r_m^2 } $, and $ \left\Vert \nabla \eta_m \right\Vert_{\omega_V } \leq C_{3}  $, where $C_{3} >0$ is a constant dependent only on $V $. By H\"ormander's $L^2 $ estimate, we can find a smooth $\linebundle^m $-valued function $ \zeta_{j,m } $ on $M$ such that $ \partialbar \left( a_{m} \eta_{m} f_{S_{1,j}} e^m_{\linebundle } - \zeta_{j,m } \right) =0 $, and for each given integer $l>0 $,
$$ \int_M \left\Vert \zeta_{j,m } \right\Vert_{h^m }^2 e^{-10 \psi_m } d\V_{\omega } \leq \frac{C_{4 }}{m } \int_M \left\Vert \partialbar \left( a_{m} \eta_{m} f_{S_{1,j}} e^m_{\linebundle } \right) \right\Vert_{h^m ,\omega }^2 e^{-10 \psi_m } d\V_{\omega } \leq C_{4 } m^{-l } , $$
for any integer $m\geq C_{4}$, where $C_{4} >0 $ is a constant depending only on $\left( M,\omega ,\linebundle ,h \right) $ and $l$. It follows that $\zeta_{j,m} \big|_{V_{r^4_m} } \in X_m $, and $ \left\Vert \zeta_{j,m } \right\Vert_{L^2 ;h^m ,\omega }^2 \leq C_{4 } m^{-l } $. Similarly, there are holomorphic sections $\zeta'_{j,m} \in Y_m $ such that $ a_{m} \eta_{m} f_{S_{1,j}} e^m_{\linebundle } - \zeta_{j,m } - \zeta'_{j,m} \in Y_m^{ \perp } $, and for any given integer $l$, we have $ \left\Vert \zeta_{j,m } \right\Vert_{L^2 ;h^m ,\omega }^2 \leq C_{5 } m^{-l } $, $ \forall m\geq C_{5} $, where $C_{5} >0 $ is a constant depending only on $\left( M,\omega ,\linebundle ,h \right) $ and $l$. By Schmidt orthonormalization process, we can find constants $b_{i,j,m} $ such that the set $ \left\lbrace \sum_{j=1}^{N_1} b_{i,j,m } \left( a_{m} \eta_{m} f_{S_{1,j}} e^m_{\linebundle } - \zeta_{j,m } - \zeta'_{j,m} \right) \right\rbrace_{i=1}^{N_1} $ becomes a orthonormal basis of $Y_m^{\perp } $, and for each given integer $l$, there exists a constant $C_{6} >0 $ depending only on $\left( M,\omega ,\linebundle ,h \right) $ and $l$, such that $ \left| b_{i,j,m} -\delta_{i,j} \right| \leq C_{6 } m^{-l } $, $ \forall m\geq C_{6} $.

An argument similar to the one used in Lemma \ref{lmmdatarfusongmanifold2} shows that for each integer $k >0 $, there exists a constant $C_{7} >0 $ depending only on $V$ and $k$, such that 
$$ \left\Vert S (x) \right\Vert_{C^k ; h^m_V , \omega_V }^2 \leq C_{7} m^{C_{7}} \left| \log h_D (x) \right|^{C_{7} } \rho_{\mathbb{D}^*_1 ,m} \left( \sqrt{ h_D (x) } \right)^2 \left\Vert S \right\Vert_{L^2 ; h^m_V , \omega_V }^2 , $$
for any $ m\geq C_{7} $, $ S\in X_m $ and $ x\in V_{r^{10}_m } $. As in Proposition \ref{propestimatequotientCkdatarfusong1}, this estimate shows that for any integer $k \geq 0 $, we have
\begin{eqnarray}
& & \left\Vert \nabla^k \left( \frac{\sum_{j=1}^{N_1 } \left\Vert a_{m} f_{S_{1,j}} e^m_{\linebundle_V} \right\Vert_{h_V^m}^2 }{ \rho_{V,\omega_V ,m } } -1 \right) \right\Vert_{\omega_V } + \left\Vert \nabla^k \left( \frac{\sum_{j=1}^{N_1 } \left\Vert \sum_{j=1}^{N_1} b_{i,j,m } a_{m} f_{S_{1,j}} e^m_{\linebundle } \right\Vert_{h^m}^2 }{ \rho_{M,\omega ,m}  } -1 \right) \right\Vert_{\omega_V } \\
& & \leq C_{8} m^{-l} \left| \log h_D \right|^{C_{8} } \rho_{\mathbb{D}^*_1 ,m} \left( \sqrt{ h_D (x) } \right) ,\;\; \forall m\geq C_{8} ,\;\; \forall x\in V_{r^{10}_m } , \nonumber
\end{eqnarray}
where $C_{8} >0 $ is a constant depending only on $(M,\omega ,\linebundle)$ and $k$. Since $\left| b_{i,j,m} -\delta_{i,j} \right| \leq C_{6 } m^{-l } $, we can apply Lemma \ref{lmmgradientomegaomega+D} to show that for any given integers $k,l>0 $, there exists a constant $C_{9} >0 $ depending only on $(M,\omega ,\linebundle)$, $k$ and $l$, such that
$$ \left\Vert \nabla^k \left( \frac{\sum_{j=1}^{N_1 } \left\Vert \sum_{j=1}^{N_1} b_{i,j,m } a_{m} f_{S_{1,j}} e^m_{\linebundle_V } \right\Vert_{h_V^m}^2 }{ \sum_{j=1}^{N_1 } \left\Vert a_{m} f_{S_{1,j}} e^m_{\linebundle_V} \right\Vert_{h_V^m}^2  } -1 \right) \right\Vert_{\omega_V +\pi^*_D } \leq C_{9} m^{-l} ,\;\; \forall m\geq C_{9 } ,\;\; \forall x\in V_{r^{10}_{m}} .$$
Combining these inequalities, one can conclude that there exists a constant $C_{10} >0 $ depending only on $(M,\omega ,\linebundle ,h)$, $k$ and $l$, such that
$$ \left\Vert \frac{\rho_{M,\omega ,m} (x) }{\rho_{V ,\omega_V ,m} (x)} -1 \right\Vert_{C^k ;\omega_V +\pi^* \omega_D } \leq C_{10} m^{-l} ,\;\; \forall m \geq C_{10} ,\; \forall x\in V_{e^{-m^{10}}} , $$
which completes the proof.
\end{proof}

\subsection{Bergman kernels on quotients of complex ball}
\hfill

Let $\left(  {\mathbb{B}}^n ,\omega_{\mathbb{B}} \right)$ be the $n$-dimensional complex hyperbolic ball with constant bisectional curvature $-1$, and let $\Gamma $ be a torsion free lattice in $ \Aut \left( {\mathbb{B}}^n ,\omega_{\mathbb{B}} \right) \cong PU (n,1) $. Then the quotient space $ {\mathbb{B}}^n / \Gamma $ is a manifold, and there exist a finite collection of disjoint open subsets $ U_i \subset \mathbb{B}^n / \Gamma $, $i=1,\cdots , N$, and complex hyperbolic cusps $\left( V_i /\Gamma_{V_i } ,\omega_{V_i }  ,\linebundle_{V_i } /\Gamma_{V_i } ,h_{V_i } \right) $ such that $ \left( U_i ,\omega_{\mathbb{B}^n} \right) \cong \left( V_i /\Gamma_{V_i } ,\omega_{V_i } \right) $, and $ \left( \mathbb{B}^n / \Gamma \right) \sq  \left( \cup_{i=1}^N U_i \right) $ is compact.

Now we consider the Bergman kernels on $U_i$. Since $U_i$ and $ \mathbb{B}^n / \Gamma $ satisfy the conditions of Theorem \ref{thmballquotient}, we see that Theorem \ref{thmballquotient} describes the asymptotic behavior of the Bergman kernels on $U_i$.

\vspace{0.2cm}

\noindent \textbf{Proof of Theorem \ref{thmballquotient}: }
Let $\gamma_m = r^2 $ and $\sigma_m =r$ in Theorem \ref{thmcusplocalizationprinciple}. Note that $U\cong V_{ r} /\Gamma_{V } $. By applying Theorem \ref{thmcusplocalizationprinciple} to $U\subset M$ and $V_{ r} /\Gamma_{V } \subset V  /\Gamma_{V } $, respectively, we can prove the theorem.
\qed

\vspace{0.2cm}

Our next goal is to prove the $C^k$-estimate of the quotients of Bergman kernel functions.

By combining Lemma \ref{lmmgradientomegaomega+D} with Proposition \ref{propestimatequotientCkdatarfusong1}, one can get the following proposition.

\begin{prop}
\label{propestimatequotientCkballquotient1}
Let $(M,\omega ,\linebundle ,h )$ and $U\cong V_r $ be as in Theorem \ref{thmballquotient}. Assume that $\Gamma_V =0$ and $\linebundle_D^{-1} $ is globally generated. Then there exists a constant $b >0 $ depending only on $ (M,\omega ,\linebundle ,h ) $, such that for any given integer $l\geq 0$,
\begin{eqnarray}
 \left\Vert \frac{\rho_{M ,\omega ,m} (x)}{\rho_{V, \omega_{V} ,m} \left( x \right) } -1 \right\Vert_{C^k ;\omega_0 } & \leq & C \left| h_D (x) \right|^{-b} m^{-l}  ,\;\; \forall x\in V_{r^2 } , \;\; \forall m\geq C , \label{inequalitypropestimatequotientCkballquotient1}
\end{eqnarray}
where $ \omega_0 = \pi_D^* \omega_D + \sqrt{-1} \partial\partialbar h_D $ is a K\"ahler metric on the total space of $\linebundle_D $, and $C >0$ is a constant depending only on $(M,\omega ,\linebundle ,h)$, $V$, $k$ and $l$.
\end{prop}

When $h_D $ is small, we cannot derive the estimate we need from Proposition \ref{propestimatequotientCkballquotient1}. In this case, we need to make some finer estimates by approximating the $L^2$ orthonormal basis of $ \hl \left( M,\omega ,\linebundle^m ,h^m \right) $ using the $L^2$ orthonormal basis of $ \hl \left( V/\Gamma_V ,\omega_V  ,( \linebundle_V / \Gamma_V )^m ,h_V^m \right) $. Suppose that $\Gamma_V =0 $. The following lemma gives an estimate about the norm of holomorphic sections on $V $.

\begin{lmm}
\label{lmmcutholomorphicsectiononVr}
Let $\left( V ,\omega_V  , \linebundle_V^m ,h_V^m \right) $ be a complex hyperbolic cusp, and $\kappa\in \left( 0,e^{-1} \right) $ be a constant. Then for any given constant $\mathfrak{K} >0 $, there exists a constant $C>0 $ depending only on $\kappa$ and $ \left( V ,\omega_V ,\linebundle_V ,h_V \right) $, such that
\begin{equation}
 \left\Vert a_{m,q} f_{S_q} e_{\linebundle_V}^m \right\Vert_{L^2 ,V\sq V_\kappa ; h_V^m ,\omega_V }^2  \leq Ce^{-4\mathfrak{K} m } \left\Vert S_q \right\Vert_{L^2 ,D ; h_D^{-q} ,\omega_D }^2 ,
\end{equation}
for any positive integers $m\geq n+1$, $q \leq e^{-4\left( \mathfrak{K} + \left| \log \left| \log \kappa \right| \right| +n+4 \right)} m $ and holomorphic sections $S_p \in H^0 \left( D,\linebundle_D^{-p} \right) $, where $a_{m,q} = \sqrt{ \frac{n q^{m-n}}{2\pi (m-n-1)!} } $, and $f_{S_q}$ is the holomorphic function on $V$ defined in Lemma \ref{holoontotal}.
\end{lmm}

\begin{rmk}
Since $q\in\mathbb{N}$, we see that Lemma \ref{lmmcutholomorphicsectiononVr} is trivial when $ m \leq e^{ 4\left( \mathfrak{K} + \left| \log \left| \log \kappa \right| \right| +n+4 \right)} $.
\end{rmk}

\begin{proof}
Without loss of generality, we can assume that $ \left\Vert S_q \right\Vert_{L^2 ,D ; h_D^{-q} ,\omega_D } =1 $. Fix $x\in V \sq V_{\kappa } $. Let $t= h_D (x) \in \left( \kappa  ,1 \right) $. By Theorem \ref{thmregularpart}, there exists a constant $C_1 >0$ depending only on $ \left( D ,\omega ,\linebundle_D ,h_D \right) $, such that $\rho_{D,\omega_D ,\linebundle_D^{-1} , h_D^{-1} ,q} \leq C_1 q^{n-1}  $, $\forall q\in \mathbb{N} $. Hence
\begin{eqnarray}
\left| f_{S_q} (x) \right|^2 = h_D (x)^q \left\Vert S_q \left( \pi_D (x) \right) \right\Vert_{h_D^{-q}}^2 \leq t^{q}  \rho_{D,\omega_D ,\linebundle_D^{-1} , h_D^{-1} ,q} \left( \pi_D (x) \right)  \leq C_1 t^{q} q^{n-1} .
\end{eqnarray}
By Stirling's formula, we can conclude that
\begin{eqnarray}
\left\Vert a_{m,q} f_{S_q} (x) e_{\linebundle_V}^m (x) \right\Vert_{h_V^m}^2  & \leq & C_1 \frac{n q^{m-1}}{2\pi (m-n-1)!} t^{q} \left| \log t \right|^m \leq C_1 \frac{n m^{n+1} q^{m-1} e^m}{2\pi \sqrt{2\pi} m^m} t^{q} \left| \log t \right|^m  \nonumber\\
& = & \frac{n C_1 }{2\pi \sqrt{2\pi} } e^{ (n+1) \log m + (m-1) \log q +m -m\log m +q\log t +m\log \left| \log t \right| } \\
& \leq & \frac{n C_1 }{2\pi \sqrt{2\pi} } e^{ n \log m - 4(m-1) \left( \mathfrak{K} + \left| \log \left| \log r \right| \right| +n+4 \right) + m\log \left| \log t \right| } \nonumber \\
& \leq & C_1 e^{ -4\mathfrak{K} m - m\left|\log \left| \log t \right| \right| } , \nonumber
\end{eqnarray}
for any positive integers $m\geq n+1$ and $q \leq e^{-4\left( \mathfrak{K} + \left| \log \left| \log r \right| \right| +n+4 \right)} m $. Now (\ref{calculationvolumeformhyperboliccusp}) implies that
\begin{eqnarray}
\left\Vert a_{m,q} f_{S_q} e_{\linebundle_V}^m \right\Vert_{L^2 ,V\sq V_\kappa ; h_V^m ,\omega_V }^2  & \leq & C_1 e^{ -4\mathfrak{K} m} \left( \V_{\omega_V} \left( V_{e^{-1}}\sq V_\kappa \right) + \int_{V \sq V_{e^{-1} } } \left| \log h_D \right|^{m } d\V_{\omega_V} \right) \nonumber \\
& \leq & C_2 e^{ -4\mathfrak{K} m} \left( 1+ \int_{\mathbb{D} \sq \mathbb{D}_{e^{-1} } } \sqrt{-1} |z|^{-2} \left| \log |z| \right|^{m-n-1 } dz\wedge d\zbar \right) \\
& \leq & C_3 e^{ -4\mathfrak{K} m} , \nonumber
\end{eqnarray}
where $C_2$, $C_3$ are positive constants depending only on $r$ and $ \left( V ,\omega_V ,\linebundle_V ,h_V \right) $. This is our assertion
\end{proof}

The following basic fact of linear algebra will also be used in the proof of the $C^k$-estimate of the quotients of Bergman kernel functions.

\begin{lmm}
\label{lmmlinearalgebramatrix}
Let $N\in\mathbb{N}$, and let $A= \left( a_{ij} \right)_{i,j=1}^N$ be an $N\times N$ Hermitian matrix. Suppose that there exists a positive constant $\delta \leq \frac{1}{100N} $ such that $\max_{1\leq i,j\leq N}\left| a_{ij} -\delta_{i,j} \right| \leq \delta $. Then there exists an $N\times N$ Hermitian matrix $B= \left( b_{ij} \right)_{i,j=1}^N$ such that $A=B \bar{B}^T $, and $\max_{1\leq i,j\leq N}\left| b_{ij} -\delta_{i,j} \right| \leq 2 N \delta $.
\end{lmm}

\begin{proof}
Recall that any Hermitian matrix can be diagonalized by a unitary matrix, so we can find $Q\in U(N)$ such that $A=Q \Lambda \bar{Q}^T $, where $\Lambda = \mathrm{diag} \left( \lambda_1 ,\cdots ,\lambda_N \right) $ be a diagonal matrix. Since $\max_{1\leq i,j\leq N}\left| a_{ij} -\delta_{i,j} \right| \leq \delta $, the characteristic polynomial of $A-I$, $p_A (\lambda ) = \lambda^n + c_1 \lambda^{n-1} +\cdots +c_n $, satisfying that $ c_i \leq N^i \delta^i $, $i=1,\cdots ,n$. Note that $ p_A (\lambda_1 ) =\cdots =p_A (\lambda_N ) =0 $. Then we can conclude that $\max_{1\leq j\leq N} \left| \lambda_j -1\right| \leq 2N\delta \leq \frac{1}{50} $. Clearly, $\Lambda $ is real matrix. Let $B = Q \sqrt{\Lambda} \bar{Q}^T $, where $ \sqrt{\Lambda} =\mathrm{diag} \left( \sqrt{\lambda_1} ,\cdots , \sqrt{\lambda_N } \right) $ is a diagonal matrix. Then we have $A=B \bar{B}^T $, and $ \left| \sqrt{\lambda_i } -1 \right| \leq \left| \lambda_i -1 \right| $ implies that $\max_{1\leq i,j\leq N}\left| b_{ij} -\delta_{i,j} \right| \leq 2 N \delta $.
\end{proof}

Now we are ready to prove the $C^k$-estimate of the quotients of Bergman kernel functions.

\begin{prop}
\label{propballquotient2}
Let $(M,\omega ,\linebundle ,h )$ and $U\cong V_r $ be as in Theorem \ref{thmballquotient}. Assume that $\Gamma_V =0$ and $\linebundle_D^{-1} $ is globally generated. Then
\begin{eqnarray}
 \sup_{V_{r^2 } } \left\Vert \frac{\rho_{M ,\omega ,m} (x)}{\rho_{V , \omega_{V} ,m} \left( x \right) } -1 \right\Vert_{C^k ;\omega_{0} } & \leq & C m^{-l} ,\;\; \forall m\geq C, \label{inequalityballquotient3}
\end{eqnarray}
where $ \omega_0 = \pi_D^* \omega_D + \sqrt{-1} \partial\partialbar h_D $ is a K\"ahler metric on the total space of $\linebundle_D $, and $C >0$ is a constant depending only on $(M,\omega ,\linebundle ,h)$, $\omega_0 $, $V$, $k$ and $l$.
\end{prop}

\begin{proof}
For each integer $ \mathfrak{q} \geq 1 $, we use $X_{m,\mathfrak{q}}$ to denote the linear subspace of $\hl \left( V_r ,\omega_V ,\linebundle_V^m ,h_V^m \right) $ generated by $\left\lbrace f_{S_{q,j}} e_{\linebundle_V}^m \big| q\geq \mathfrak{q} , 1\leq j\leq N_q \right\rbrace $, where $\left\lbrace S_{q,j} \right\rbrace_{j=1}^{N_q} $ is an $L^2 $ orthonormal basis in the Hilbert space $\hl \left( D,\linebundle^{-q}_D \right) $, and $f_{S_{q,j} } $ denote the holomorphic function corresponding to $S_{q,j} $ defined in Lemma \ref{holoontotal}. Write $Y_{m,\mathfrak{q}} =X_{m,\mathfrak{q}} \cap  \hl \left( V ,\omega_V ,\linebundle_V^m ,h_V^m \right) $, and $Z_{m,\mathfrak{q}} =X_{m,\mathfrak{q}} \cap  \hl \left( M ,\omega ,\linebundle^m ,h^m \right) $. Let $\mathfrak{K} \geq 10 $ be a constant that will be determined later. Set $q_{ m,\mathfrak{K} } = \left\lfloor e^{-4\left( \mathfrak{K} + \left| \log \left| \log r \right| \right| +n+6 \right)} m \right\rfloor $, for any integer $ m\geq e^{ 4\left( \mathfrak{K} + \left| \log \left| \log r \right| \right| +n+6 \right)} $. Then Lemma \ref{lmmcutholomorphicsectiononVr} shows that there exists a constant $C_1 \geq 1 $ depending only on $ r $ and $\left( V,\omega_V ,\linebundle_V ,h_V \right) $, such that 
\begin{equation}
 \left\Vert a_{m,q} f_{S_{q,j}} e_{\linebundle_V}^m \right\Vert_{L^2 ,V\sq V_{r^2 } ; h_V^m ,\omega_V }^2  \leq C_1 e^{-4\mathfrak{K} m } ,\;\; \forall q\leq q_{m,\mathfrak{K}} ,\;\; 1\leq j \leq N_q ,
\end{equation}
where $a_{m,q} = \sqrt{ \frac{n q^{m-n}}{2\pi (m-n-1)!} } $. Let $\eta \in C^\infty (V) $ such that $ 0\leq \eta \leq 1 $, $ \left| \nabla \eta \right|_{\omega_V } \leq 10 $, $ \eta =1 $ on $V_{r^2}$ and $\eta =0$ on $V\sq V_r $. Then we can apply H\"ormander's $L^2$ estimate to show that there are $\linebundle^m $-valued smooth functions $u_{m,q,j} $ on $M$, such that $\partialbar u_{m,q,j} = \partialbar \left( a_{m,q} \eta f_{S_{q,j}} e_{\linebundle_V}^m \right) $, and 
\begin{eqnarray}
\int_M \left\Vert u_{m,q,j} \right\Vert_{h^m}^2 d\V_{\omega} \leq \int_M \left\Vert \partialbar \left( a_{m,q} \eta f_{S_{q,j}} e_{\linebundle }^m \right) \right\Vert_{h^m ,\omega }^2 d\V_{\omega} \leq 100C_1 e^{-4\mathfrak{K} m} ,
\end{eqnarray}
for any integers $q\leq q_{m,\mathfrak{K}} $ and $ 1\leq j \leq N_q$. Write $\widetilde{S}'_{m,q,j} = a_{m,q} \eta f_{S_{q,j}} e_{\linebundle }^m - u_{m,q,j} $ for any integers $q\leq q_{m,\mathfrak{K}} $ and $ 1\leq j \leq N_q$. Since $ a_{m,q} f_{S_{q,j}} e_{\linebundle_V }^m \perp X_{m,q_{m,\mathfrak{K}} +1} $, $\forall q \leq q_{m,\mathfrak{K}} $, we see that the image of $\widetilde{S}'_{m,q,j}$ under the orthogonal projection of $\hl \left( M ,\linebundle^m \right) $ onto $Z_{m,q_{m,\mathfrak{K}} +1} $, $ \widetilde{S}'^{Z}_{m,q,j} $, satisfying that
\begin{eqnarray}
\;\;\;\;\;\;\left\Vert \widetilde{S}'^{Z}_{m,q,j} \right\Vert^2_{L^2 ,M; h^m ,\omega } & \leq & \left\Vert u_{m,q,j} \right\Vert_{L^2 ,M; h^m ,\omega } + \left\Vert a_{m,q} f_{S_{q,j}} e_{\linebundle }^m \right\Vert_{L^2 ,M \sq V_{r^2} ; h^m ,\omega } \leq 200 C_1 e^{-2\mathfrak{K} m} .
\end{eqnarray}
Set $\widetilde{S}''_{m,q,j} = \widetilde{S}'_{m,q,j} - \widetilde{S}'^{Z}_{m,q,j} $, and $N_{\mathrm{sum} , m,\mathfrak{K} } = \sum_{q=1}^{q_{m,\mathfrak{K}}} N_q $. Then $\upsilon_{m,q,j,q',j'} = \left\langle \widetilde{S}''_{m,q,j} , \widetilde{S}''_{m,q',j'} \right\rangle $ gives an $N_{\mathrm{sum} , m,\mathfrak{K} } \times N_{\mathrm{sum} , m,\mathfrak{K} } $ Hermitian matrix $\Upsilon_m $ such that $\left| \upsilon_{m,q,j,q',j'} -\delta_{q,q'} \delta_{j,j'} \right| \leq 400C_1 e^{-\mathfrak{K} m} $. Note that $N_{\mathrm{sum} , m,\mathfrak{K} } \leq C_2 q^n $ for some constant $C_2 >0$ depending only on $\left( D,\omega_D ,\linebundle_D ,h_D \right) $. So we can assume that $ e^{ \mathfrak{K} m} \geq e^{ 400 (C_1 + C_2 )} q^{4n}  $ by increasing $m$ if necessary. Lemma \ref{lmmlinearalgebramatrix} now shows that there exists a constant $C_3 \geq 1 $ depending only on $ r $ and $\left( V,\omega_V ,\linebundle_V ,h_V \right) $, such that we can find constants $b_{m,q,j,q',j'} $ satisfying that $ |b_{m,q,j,q',j'} | \leq C_3 q^n e^{-\mathfrak{K} m} $, and
\begin{eqnarray}
\left\lbrace \widetilde{S}_{m,q,j} = \widetilde{S}''_{m,q,j} + \sum_{q' =1}^{q_{m,\mathfrak{K}}} \sum_{j=1}^{N_{q'}} b_{m,q,j,q',j'} \widetilde{S}''_{m,q',j'} \bigg| 1\leq q\leq q_{m,\mathfrak{K}} , 1\leq j \leq N_q \right\rbrace
\end{eqnarray}
forms an orthonormal basis of $X_{m,q_{m,\mathfrak{K}} +1}^{\perp} $. It follows that 
\begin{eqnarray}
\widetilde{S}_{m,q,j} = a_{m,q} f_{S_{q,j}} e_{\linebundle_V }^m + \sum_{q' =1}^{q_{m,\mathfrak{K}}} \sum_{j=1}^{N_{q'}} c_{m,q,j,q',j'} a_{m,q'} f_{S_{q',j'}} e_{\linebundle_V }^m + \widetilde{S}^{Z}_{m,q,j} , \label{estimateholomorphicMSmqjclosetoafe}
\end{eqnarray}
where $c_{m,q,j,q',j'} $ are constants satisfying that $ |c_{m,q,j,q',j'}| \leq 200 C_3 q^{3n} e^{-\mathfrak{K} m} $, and $\widetilde{S}^{Z}_{m,q,j} \in Z_{m,q_{m,\mathfrak{K}} +1} $ satisfying that $\left\Vert \widetilde{S}^{Z}_{m,q,j} \right\Vert^2_{L^2 ,M; h^m ,\omega } \leq 200 C_3 q^{3n} e^{-\mathfrak{K} m} $.

Fix $k\in\mathbb{N}$ and $x\in V_{r^3} $. Then there are vector fields $\mathcal{Z}_i $ around $x$ satisfying the estimates in Lemma \ref{lmmgradientomegaomega+D}. For any finite sequence of integers $1\leq i_1 ,\cdots ,i_p \leq 2n $, we can define a linear functional $T_{x,m,i_1 ,\cdots ,i_p } $ on $ \hl \left( V_r , \linebundle_V^m \right) $ by
\begin{equation}
T_{x,m,i_1 ,\cdots ,i_p } (S) = \mathcal{Z}_{i_1} \left( \cdots \left( \mathcal{Z}_{i_p} \left( h_D^{-\frac{1}{2}} e_{\linebundle_V}^{-m} (S) \right) \right) \cdots \right) , \;\forall S\in \hl \left( V_r , \linebundle_V^m \right) .
\end{equation} 
By a straightforward calculation, one can obtain
\begin{eqnarray}
\left| T_{x,m,i_1 ,\cdots ,i_p } \left( f_{S_{q,j}} e_{\linebundle_V}^m \right) \right|^2 \leq C_4 q^{p+n-1} h_D (x)^{\max \lbrace 0, q-p-1 \rbrace } ,
\end{eqnarray}
for any integers $q\in \mathbb{N}$ and $1\leq j\leq N_q$, where $C_4 >0$ is a constant depending only on $ r $, $p$ and $\left( V,\omega_V ,\linebundle_V ,h_V \right) $. Suppose that $q_{m,\mathfrak{K}} \geq p+1 $. Then 
\begin{eqnarray}
\left\Vert T_{x,m,i_1 ,\cdots ,i_p } \big|_{ Y_{m,q_{m,\mathfrak{K}} +1} } \right\Vert^2 & = & \sum_{q> q_{m,\mathfrak{K}} } \sum_{j=1}^{N_q} \left| T_{x,m,i_1 ,\cdots ,i_p } \left( a_{m,q} f_{S_{q,j}} e_{\linebundle_V }^m  \right) \right|^2 \nonumber \\
& \leq & \sum_{q> q_{m,\mathfrak{K}} } C_4 N_q a_{m,q}^2 q^{p+n+1} h_D (x)^{ q-p-1 } \label{estimatehigerdegreetermonV} \\
& \leq & \sum_{q> q_{m,\mathfrak{K}} } \frac{n C_2 C_4 }{2\pi (m-n-1)! } q^{m+p+ n } h_D (x)^{ q-p-1 } \nonumber .
\end{eqnarray}
When $\log h_D (x) \leq -10 -\frac{m+p+n}{ q_{m,\mathfrak{K}} } $, (\ref{estimatehigerdegreetermonV}) implies that
\begin{eqnarray}
\left\Vert T_{x,m,i_1 ,\cdots ,i_p } \big|_{ Y_{m,q_{m,\mathfrak{K}} +1} } \right\Vert^2 \leq \frac{n C_2 C_4 }{ \pi (m-n-1)! } q_{m,\mathfrak{K}}^{m+p+ n } h_D (x)^{ q_{m,\mathfrak{K}}-p-1 }  .\label{estimatehigerdegreetermonVsmallhD}
\end{eqnarray}

Note that $ \rho_{\omega_D ,1} >0 $ on $D$. Then we can conclude that
\begin{eqnarray}
 \left| T_{x,m,i_1 ,\cdots ,i_p } \left( a_{m,q} f_{S_{q,j}} e_{\linebundle_V}^m \right) \right|^2 \leq C_4 e^{km} q_{m,\mathfrak{K}}^{2k+2n} \sum_{q' =1}^{q_{m,\mathfrak{K} }} \sum_{j'=1}^{N_{q' }} \left| a_{m,q'} f_{S_{q',j'}} h_D^{-\frac{1}{2}} \right|^2 , \label{estimatesmallqTimage}
\end{eqnarray}
for any integers $q\leq q_{m,\mathfrak{K}}$ and $1\leq j\leq N_q $.

By the construction in much the same way as in the proof of Lemma \ref{lmmquasipshhyperboliccuspver2}, we see that there exist a quasi-plurisubharmonic function $\psi \in L^1 \left( V,\omega_V \right) $ and a constant $C_5 $ depending only on $ r $ and $\left( V,\omega_V ,\linebundle_V ,h_V \right) $, such that $\psi \leq 0 $, $\sqrt{-1} \partial\partialbar \psi \geq -C_5 \omega_V $, $\mathrm{supp} \psi \subset V_{r^2} $, and
$$ \lim_{h_D (y) \to 0 } \left( \psi (y) - \log \left( \dist_{\omega_V} (x,y) \right) \right) = \lim_{y\to x} \left( \psi (y) - \log \left( h_D (y) \right) \right) = -\infty .$$
Assume that $m-n-2 \geq C_5 \left( q_{m,\mathfrak{K}} +4n +4p \right) $. Then $ \sqrt{-1} \left( 2n+2k + q_{m,\mathfrak{K}} \right) \partial\partialbar \psi + m\omega + \Ric (\omega ) \geq \omega $. Let $S_{ x,m,Y,i_1 ,\cdots ,i_p } $ be the peak section of $ T_{x,m,i_1 ,\cdots ,i_p } \big|_{ Y_{m,q_{m,\mathfrak{K}} +1} } $. Hence we can use H\"ormander's $L^2$ estimate to show that there exists a smooth $\linebundle^m$-values $u_{ x,m,Y,i_1 ,\cdots ,i_p } $ on $M$ such that $\partialbar u_{ x,m,Y,i_1 ,\cdots ,i_p } = \partialbar \left( \eta S_{ x,m,Y,i_1 ,\cdots ,i_p } \right) $, and
\begin{eqnarray}
\int_M \left\Vert u_{ x,m,Y,i_1 ,\cdots ,i_p } \right\Vert_{h^m}^2 d\V_{\omega} & \leq & \int_M \left\Vert \partialbar \left( \eta S_{ x,m,Y,i_1 ,\cdots ,i_p } \right) \right\Vert_{h^m ,\omega }^2 d\V_{\omega} \\
& = & \int_M \left\Vert \partialbar \eta \otimes S_{ x,m,Y,i_1 ,\cdots ,i_p }  \right\Vert_{h^m ,\omega }^2 d\V_{\omega} \leq 100 . \nonumber
\end{eqnarray}
Since $S_{ x,m,Y,i_1 ,\cdots ,i_p } $ is the peak section of $ T_{x,m,i_1 ,\cdots ,i_p } \big|_{ Y_{m,q_{m,\mathfrak{K}} +1} } $, we can conclude that
\begin{eqnarray}
T_{x,m,i_1 ,\cdots ,i_p } \left( \eta S_{ x,m,Y,i_1 ,\cdots ,i_p } - u_{ x,m,Y,i_1 ,\cdots ,i_p } \right) = \left\Vert T_{x,m,i_1 ,\cdots ,i_p } \big|_{ Y_{m,q_{m,\mathfrak{K}} +1} } \right\Vert .
\end{eqnarray}
It follows that
\begin{eqnarray}
\left\Vert T_{x,m,i_1 ,\cdots ,i_p } \big|_{ Y_{m,q_{m,\mathfrak{K}} +1} } \right\Vert^2 & \leq & 200 \left\Vert T_{x,m,i_1 ,\cdots ,i_p } \big|_{ Z_{m,q_{m,\mathfrak{K}} +1} } \right\Vert^2 . \label{estimatehigerdegreetermonMsmallhDYZ} 
\end{eqnarray}
Similarly, we have
\begin{eqnarray}
\left\Vert T_{x,m,i_1 ,\cdots ,i_p } \big|_{ Z_{m,q_{m,\mathfrak{K}} +1} } \right\Vert^2 & \leq & 200 \left\Vert T_{x,m,i_1 ,\cdots ,i_p } \big|_{ Y_{m,q_{m,\mathfrak{K}} +1} } \right\Vert^2 . \label{estimatehigerdegreetermonMsmallhDZY} 
\end{eqnarray}

Let $\left\lbrace \widetilde{S}_{i} \right\rbrace_{i\in I_{m,\mathfrak{K}}} $ be an $L^2$ orthonormal basis of $Z_{m, \mathfrak{K} }$. Then we have
\begin{eqnarray}
\frac{\rho_{M ,\omega ,m} }{\rho_{V , \omega_{V} ,m} } & = & \left( \sum_{q =1}^{q_{m,\mathfrak{K}}} \sum_{j=1}^{N_{q }} \left\Vert \widetilde{S}_{m,q,j} \right\Vert_{h^m}^2 + \sum_{ i\in I_{m,\mathfrak{K}} } \left\Vert \widetilde{S}_{i} \right\Vert_{h^m}^2 \right) \cdot \left( \sum_{q =1}^{\infty } \sum_{j=1}^{N_{q }} \left\Vert a_{m,q} f_{S_{q,j}} e_{\linebundle_V }^m \right\Vert_{h^m}^2 \right)^{-1} \nonumber \\
& = & \left( \sum_{q =1}^{q_{m,\mathfrak{K}}} \sum_{j=1}^{N_{q }} \left| e_{\linebundle_V}^{-m} \left( \widetilde{S}_{m,q,j} \right) h_D^{-\frac{1}{2}} \right|^2 \right) \cdot \left( \sum_{q =1}^{\infty } \sum_{j=1}^{N_{q }} \left| a_{m,q} f_{S_{q,j}} h_D^{-\frac{1}{2}} \right|^2 \right)^{-1} \\
& & + \left( \sum_{ i\in I_{m,\mathfrak{K}} } \left| e_{\linebundle_V}^{-m} \left( \widetilde{S}_{i} \right) h_D^{-\frac{1}{2}} \right|^2 \right) \cdot \left( \sum_{q =1}^{\infty } \sum_{j=1}^{N_{q }} \left| a_{m,q} f_{S_{q,j}} h_D^{-\frac{1}{2}} \right|^2 \right)^{-1} . \nonumber
\end{eqnarray}

By Lemma \ref{holoontotal}, there exists a constant $C_6 $ depending only on $k $, $ r $ and $\left( V,\omega_V ,\linebundle_V ,h_V \right) $, such that
\begin{eqnarray}
\sum_{p=0}^k \left\Vert \nabla^p f (x) \right\Vert_{ \omega_0 } & \leq & C_6 \sum_{p=0}^k \sum_{1\leq i_1 , \cdots , i_p \leq 2n } \left| \mathcal{Z}_{i_1} \left( \cdots \left( \mathcal{Z}_{i_p} \left( f \right)  \right) \cdots \right) (x) \right| ,
\end{eqnarray}
for any smooth function around $x$. As in the proof of Proposition \ref{propestimatequotientCkdatarfusong1}, we see that the estimates (\ref{estimatehigerdegreetermonMsmallhDZY}), (\ref{estimatehigerdegreetermonVsmallhD}) and (\ref{estimatesmallqTimage}) imply that
\begin{eqnarray}
& & \left\Vert \left( \sum_{q =q_{m,\mathfrak{K}} +1}^{\infty } \sum_{j=1}^{N_{q }} \left| a_{m,q} f_{S_{q,j}} h_D^{-\frac{1}{2}} \right|^2 \right) \cdot \left( \sum_{q =1}^{q_{m,\mathfrak{K}} } \sum_{j=1}^{N_{q }} \left| a_{m,q} f_{S_{q,j}} h_D^{-\frac{1}{2}} \right|^2 \right)^{-1} \right\Vert_{ C^k ; \omega_0 } (x) \nonumber \\
& \leq & C_6 \sum_{p=0}^k \sum_{1\leq i_1 , \cdots , i_p \leq 2n } \left| \mathcal{Z}_{i_1} \left( \cdots \left( \mathcal{Z}_{i_p} \left( \left( \sum_{q =q_{m,\mathfrak{K} +1}}^{\infty } \sum_{j=1}^{N_{q }} \left| a_{m,q} f_{S_{q,j}} h_D^{-\frac{1}{2}} \right|^2 \right) \right. \right. \right. \right.  \nonumber \\
& & \;\;\;\;\;\;\;\;\;\;\;\;\;\;\;\;\;\;\;\;\;\;\;\;\;\;\;\;\;\;\;\;\;\;\;\;\;\;\;\;\;\;\;\;\;\;\;\;\;\;\;\;\;\; \left. \left. \left. \left. \cdot \left( \sum_{q =1}^{q_{m,\mathfrak{K} }} \sum_{j=1}^{N_{q }} \left| a_{m,q} f_{S_{q,j}} h_D^{-\frac{1}{2}} \right|^2 \right)^{-1} \right)  \right) \cdots \right) (x) \right| \\
& \leq & C_7 e^{k(k-1) m } q_{m,\mathfrak{K}}^{2k(k+n)} \left( \sum_{p=0}^k \sum_{1\leq i_1 , \cdots , i_p \leq 2n } \left\Vert T_{x,m,i_1 ,\cdots ,i_p } \big|_{ Y_{m,q_{m,\mathfrak{K}} +1} } \right\Vert^2 \right. \nonumber \\
& & \left. \;\;\;\;\;\;\;\;\;\;\;\;\;\;\;\;\;\;\;\;\;\;\;\;\;\;\;\;\;\;\;\;\;\;\;\;\;\;\;\;\;\;\;\;\;\;\;\; \cdot \left( \sum_{q =1}^{q_{m,\mathfrak{K} }} \sum_{j=1}^{N_{q }} \left| a_{m,q} f_{S_{q,j}} (x) h_D (x)^{-\frac{1}{2}} \right|^2 \right)^{-1} \right) \nonumber \\
& \leq & C^2_7 \left( q_{m,\mathfrak{K}}^{m+k+ n } h_D (x)^{ q_{m,\mathfrak{K}}-k-1 } \right)^k ,\nonumber
\end{eqnarray}
where $C_7 >0$ is a constant depending only on $k $, $ r $ and $\left( V,\omega_V ,\linebundle_V ,h_V \right) $. Similarly, we have
\begin{eqnarray}
& & \left\Vert \left( \sum_{ i\in I_{m,\mathfrak{K}} } \left| e_{\linebundle_V}^{-m} \left( \widetilde{S}_{i} \right) h_D^{-\frac{1}{2}} \right|^2 \right) \cdot \left( \sum_{q =1}^{\infty } \sum_{j=1}^{N_{q }} \left| a_{m,q} f_{S_{q,j}} h_D^{-\frac{1}{2}} \right|^2 \right)^{-1} \right\Vert_{ C^k ; \omega_0 } (x) \label{estimatekeyCkestimatequotientballquotient1} \\
& \leq & e^{200 k} C^2_7 \left( q_{m,\mathfrak{K}}^{m+k+ n } h_D (x)^{ q_{m,\mathfrak{K}}-k-1 } \right)^k . \nonumber 
\end{eqnarray}

Then we can use (\ref{estimateholomorphicMSmqjclosetoafe}) to find a constant $C_8 >0$ depending only on $k $, $ r $ and $\left( V,\omega_V ,\linebundle_V ,h_V \right) $, such that
\begin{eqnarray}
& & \left\Vert \left( \sum_{q =1}^{q_{m,\mathfrak{K}}} \sum_{j=1}^{N_{q }} \left| e_{\linebundle_V}^{-m} \left( \widetilde{S}_{m,q,j} \right) h_D^{-\frac{1}{2}} \right|^2 \right) \cdot \left( \sum_{q =1}^{\infty } \sum_{j=1}^{N_{q }} \left| a_{m,q} f_{S_{q,j}} h_D^{-\frac{1}{2}} \right|^2 \right)^{-1} \right\Vert_{ C^k ; \omega_0 } (x) \label{estimatekeyCkestimatequotientballquotient2} \\
& \leq & C_8 q^{2n} e^{k^2 m -\mathfrak{K} m} + C_8 \left( q_{m,\mathfrak{K}}^{m+k+ n } h_D (x)^{ q_{m,\mathfrak{K}}-k-1 } \right)^k . \nonumber 
\end{eqnarray}

Let $\mathfrak{K} > 100\left( n+k+C_5 \right)^2 $ be a large constant, let $m_1 \geq e^{4\left( \mathfrak{K} +\left| \log | \log r | \right| + | \log r | +2n+2k +6 \right)} $ be an integer, and let $ \mathcal{C} > m_1^2 $ be a constant. Then for any $m\geq m_1$ and $x\in V_{m^{-\mathcal{C}}}$, we have 
\begin{eqnarray}
q_{m,\mathfrak{K}}^{m+k+ n } h_D (x)^{ q_{m,\mathfrak{K}}-k-1 } \leq e^{-m} .
\end{eqnarray}
It follows that
\begin{eqnarray}
 \sup_{V_{m^{-\mathcal{C}}} } \left\Vert \frac{\rho_{M ,\omega ,m} (x)}{\rho_{V , \omega_{V} ,m} \left( x \right) } -1 \right\Vert_{C^k ;\omega_{0} } & \leq & C_9 e^{-m} ,\;\; \forall m\geq m_1 ,
\end{eqnarray}
where $C_9 >0$ is a constant depending only on $k $, $ r $ and $\left( V,\omega_V ,\linebundle_V ,h_V \right) $.

When $x\in V\sq V_{m^{-\mathcal{C}}} $, the estimate (\ref{inequalityballquotient3}) is a consequence of Proposition \ref{propestimatequotientCkballquotient1}. This completes the proof.
\end{proof}

Now we consider Theorem \ref{thmsupballquotient}. In order to consider the case where $\Gamma $ is not torsion free, we need to recall some basic notations about orbifolds and orbifold line bundles. 

An $n$-dimensional complex orbifold is a composite concept formed of a second countable Hausdorff paracompact topological space $X$ and a collection of data $\left\lbrace \widetilde{U}_x , G_x ,U_x ,p_x \right\rbrace_{x\in X} $ with the following property. For each point $x\in X$, there exist an open neighborhood $U_x$ of $x$ in $X$, a connected open subset $\widetilde{U}_x \subset \mathbb{C}^n $ and a continuous map $p_{x} :\widetilde{U}_x \to U_x$ together with a finite group $G$ of biholomorphic automorphisms of $\widetilde{U}_x$ such that $p_x $ induces an homeomorphism between $U_x$ and $\widetilde{U}_x /G $. If $x\in U_y $ for some $y\in X$, then there exist a neighborhood $V_x$ of $x$ and a holomorphic map $\phi : p_{x}^{-1} \left( V_x \right) \to \widetilde{U}_y $ such that $p_y = p_x \circ \phi $. 

Let $G_x =\left\lbrace g\in G: g(x)=x \right\rbrace $. Then $G_x$ is unique up to isomorphism, and the complex manifold $X_{reg} =\left\lbrace x\in X: \left|G_x \right| =1 \right\rbrace $ is an open dense subset of $X$. A K\"ahler metric $\omega $ (resp. line orbibundle $\linebundle$) on $X$ is a K\"ahler metric $\omega $ (resp. line bundle $\linebundle $) on $X_{reg}$ such that for each $x\in X$, $p_x^* \omega $ is a restriction of a K\"ahler metric $\widetilde{\omega} $ (resp. line bundle $\widetilde{\linebundle }$) on $\widetilde{U}_x$ on $ p_x^{-1} \left( U_x \cap X_{reg} \right) $. Similarly, a hermitian metric $h$ on $\linebundle$ is a hermitian metric $h$ on $\linebundle\big|_{X_{reg}}$ such that for each $x\in X$, $p_x^* h $ is a restriction of a hermitian metric $\widetilde{h}_{ \widetilde{U}_x } $ on $\widetilde{\linebundle} \big|_{\widetilde{U}_x } $ on $ p_x^{-1} \left( U_x \cap X_{reg} \right) $. 

A K\"ahler orbifold is defined in terms of a complex orbifold $X$ and a K\"ahler metric $\omega$ on $X$. We say that a line orbibundle $\linebundle$ on a complex orbifold $X$ is positive, if there exists a hermitian $h$ on $\linebundle $ such that $\frac{1}{2\pi} \Ric(h)$ is a K\"ahler metric on $X$, where $\Ric(h) =-\sqrt{-1} \partial\partialbar \left( \log (h) \right) $. This $\omega =\frac{1}{2\pi} \Ric(h)$ is a polarized K\"ahler metric on $X$. A global section $S$ of $\linebundle$ on $X$ is a section $S\in H^0 \left( X_{reg} , \linebundle\big|_{X_{reg}} \right) $ that can be extended on each $\widetilde{U}_x $. Let $ H^0 \left( X , \linebundle \right) $ denote the linear space of all these global sections of $\linebundle$. Then we can define the Bergman kernels on orbifolds just as in the beginning of Section \ref{intro}.

Then we recall some basic properties of Bergman kernels on orbifolds. Let $(X,\omega )$ be a complete K\"ahler orbifold, and $( \linebundle ,h )$ be a Hermitian line orbibundle on $ X $. Fix $x\in X $. Then there exists an open neighborhood $U_x $ of $x$ such that $U_x \cong \widetilde{U}_x / G_x $ for some open neighborhood $ \widetilde{U}_x $ of $0\in \mathbb{C}^n $. Let $p_x $ denote the quotient map $ \widetilde{U}_x \to \widetilde{U}_x / G_x \cong U_x $. Assume that $p_x (0) =x $, $\widetilde{ \omega } $ is an extension of $  p_x^* {\omega \big|_{ U_x \cap X_{reg}} } $ on $\widetilde{U}_x $, and $ \big( \widetilde{\linebundle} , \widetilde{h} \big) $ is an extension of $ \big( p_x^* {\linebundle \big|_{ U_x \cap X_{reg}} } , p_x^* {h \big|_{ U_x \cap X_{reg}} } \big) $ on $\widetilde{U}_x $. Then we see that the action of $G_x $ on $U_x $ can be extended to an action on $ \widetilde{\linebundle} $ that preserves $ \widetilde{h} $, and $ \widetilde{\linebundle} /G_x \cong \linebundle $. Let $\hl \big( \widetilde{ U }_x ,\widetilde{ \omega } ,\widetilde{ \linebundle }^m ,\widetilde{ h }^m \big)_{G_x } $ be the subspace of $\hl \big( \widetilde{ U }_x ,\widetilde{ \omega } ,\widetilde{ \linebundle }^m ,\widetilde{ h }^m \big) $ containing all holomorphic sections that are invariant under the action of $G_x$. Then there exist an orthonormal basis $\left\lbrace S_{m,i} \right\rbrace_{i\in I_m } $ of $\hl \big( \widetilde{ U }_x ,\widetilde{ \omega } ,\widetilde{ \linebundle }^m ,\widetilde{ h }^m \big) $ and a subset $I_{m,1} \subset I_m $, such that $\left\lbrace S_{m,i} \right\rbrace_{i\in I_{m,1} } $ is an orthonormal basis of $\hl \big( \widetilde{ U }_x ,\widetilde{ \omega } ,\widetilde{ \linebundle }^m ,\widetilde{ h }^m \big)_{G_x } $. Then
\begin{eqnarray}
\rho_{U_x ,\omega , \linebundle^m ,h^m } \left( p_x (z) \right) & = & |G_x | \sum_{i\in I_{m,1}} \left\Vert S_{m,i} (z) \right\Vert_{h^m}^2 = \sum_{g\in G_x } \sum_{i\in I_{m }} h^m \left( g \left( S_{m,i} \right) (z) , S_{m,i} (z)  \right) , \label{equationbergmankernelorbifold}
\end{eqnarray}
and similarly,
\begin{eqnarray}
p_x^* B_{U_x ,\omega , \linebundle^m ,h^m } \left( p_x (z) , p_x (w) \right) & = & \sum_{g\in G_x } \sum_{i\in I_{m }} g \left( S_{m,i} \right) (z) \otimes S_{m,i} (w)^* . \label{equationbergmankernelorbifold2}
\end{eqnarray}
Assume that there exists a section $ e_{\widetilde{\linebundle} } \in H^0 \big( \widetilde{U}_x ,\widetilde{\linebundle} \big) $ such that $ e_{\widetilde{\linebundle} } \neq 0 $ on $\widetilde{U}_x $, and $ \big\Vert e_{\widetilde{\linebundle} } \big\Vert_{h } $ is invariant under the action of $G_x $. Then there exists a homomorphism $\Theta_{G_x }^{orb} : G_x \to U(1) \subset \mathbb{C}^* $ such that $g \big( e_{\widetilde{\linebundle} } \big) = \Theta_{G_x }^{orb} (g) e_{\widetilde{\linebundle} }  $, $\forall g \in G_x $. Clearly, $\Theta_{G_x }^{orb} $ is independent of the choice of $ e_{\widetilde{\linebundle} }$. 

Now we can give an estimate of the Bergman kernel function on $\mathbb{B}^n /\Gamma $ far away from the cusp. The following lemma has actually been proved by Dai-Liu-Ma \cite{dailiuma1} using the heat kernel method, but here we give a proof that only relies on the peak section method.

\begin{lmm}
\label{lmmbergmankernelhyperbolicSUnfinitequotient}
We follow the notations and assumptions above. For any constants $n,\Lambda ,k,\epsilon >0$, there are constant $\delta ,C ,m_0 >0 $ with the following property. Let $(X, \omega )$ be a complete K\"ahler orbifold with $\Ric (\omega) \geq -\Lambda \omega $, $ (\linebundle ,h ) $ be a Hermitian line orbibundle on $X$ with $\Ric (h) \geq \epsilon \omega $, $x\in X$ and $m\geq m_0 $. Assume that $G_x \leqslant U(n) $, and there exists a constant $r \geq \frac{\epsilon \log (m)}{\sqrt{m} } $ such that the open set $U_x $ containing $B_r (x)$. Suppose that the bisectional curvature $BK =-1 $ on $U_x $, and $\Ric (h) = (n+1) \omega $ on $U_x $. Then
\begin{eqnarray}
\left\Vert \rho_{ \omega , m} \circ p_x (z) - \frac{(2 \pi)^{-n} (mn-1)!}{ (mn-n-1)!} \sum_{g\in G_x}  \frac{\Theta_{G_x }^{orb} (g)^m \left( 1-|z|^2 \right)^{mn} }{ \left( 1-\left\langle g^{-1} z,z \right\rangle \right)^{mn} } \right\Vert_{C^k ;B_{\frac{r}{2}} (0) }  \leq C e^{-\delta mr^2 } ,
\end{eqnarray}
where $ \rho_{ \omega ,m} $ is the Bergman kernel function on $(X, \omega ,\linebundle ,h )$, and $ \left\langle z , w \right\rangle = \sum_{i=1}^n z_i \bar{w}_i $, $ \forall z,w\in\mathbb{C}^n $.
\end{lmm}

\begin{proof}
By assumptions, we can consider $\widetilde{U}_x = p_x^{-1} \big( U_x \big) $ as an open neighborhood of $ 0\in \mathbb{B}^n $ such that $\widetilde{U}_x $ is invariant under the action of $G_x \leqslant U(n) $. Then $ \big( \widetilde{U}_x , \widetilde{\omega} ,\widetilde{\linebundle } , \widetilde{h} \big) \cong \big( \widetilde{U}_x ,  \omega_{\mathbb{B}^n} ,\mathcal{K}_{\mathbb{B}^n} , h_{\mathbb{B}^n} \big) $, where $ \big( \mathcal{K}_{\mathbb{B}^n} , h_{\mathbb{B}^n} \big) $ is the Hermitian line bundle in Theorem \ref{thmsupballquotient}.

Let $e_{\mathcal{K}_{\mathbb{B}^n} } \in H^0 \left( \mathbb{B}^n ,\mathcal{K}_{\mathbb{B}^n} \right) $ such that $ \big\Vert e_{\mathcal{K}_{\mathbb{B}^n}} \big\Vert_{h_{\mathcal{K}_{\mathbb{B}^n}}}^2 = \big( 1 -|z|^2 \big)^n $. Then we can describe the action $g: H^0 \left( \mathbb{B}^n ,\mathcal{K}_{\mathbb{B}^n} \right) \to H^0 \left( \mathbb{B}^n ,\mathcal{K}_{\mathbb{B}^n} \right) $ by $ g \big( fe_{ \mathcal{K}_{\mathbb{B}^n} } \big) = \big( \Theta_{G_x }^{orb} (g) \cdot f\circ g^{-1} \big) e_{ \mathcal{K}_{\mathbb{B}^n} } $, for any holomorphic section $fe_{\mathcal{K}_{\mathbb{B}^n}} \in H^0 \left( \mathbb{B}^n ,\mathcal{K}_{\mathbb{B}^n} \right) $. Write 
$$S_{m,P} = \sqrt{\frac{ ( mn+|P|-1 ) ! }{2^n \pi^n p_1 ! \cdots p_n ! (mn-n-1)! }} z^P e_{\mathcal{K}_{\mathbb{B}^n} }^m = f_{m,P} (z) e_{\mathcal{K}_{\mathbb{B}^n} }^m \in H^0 \left( \mathbb{B}^n ,\mathcal{K}^m_{\mathbb{B}^n} \right) , $$
for each multi-index $P \in \mathbb{Z}_{\geq 0}^n $ and integer $m\geq 2$. Since $\big\lbrace S_{m,P} \big\rbrace_{ P \in \mathbb{Z}_{\geq 0}^n } $ is an orthonormal basis of $ \hl \big( \mathbb{B}^n ,  \omega_{\mathbb{B}^n} ,\mathcal{K}^m_{\mathbb{B}^n} , h^m_{\mathbb{B}^n} \big) $, we can apply (\ref{equationbergmankernelorbifold}) and (\ref{equationbergmankernelorbifold2}) to show that for any integer $m\geq 2$,
\begin{eqnarray}
\rho_{\mathbb{B}^n /G_x ,  \omega_{\mathbb{B}^n} , \mathcal{K}_{\mathbb{B}^n} /G_x , h^m_{\mathbb{B}^n} ,m } \left( p_x (z) \right) & = & \sum_{g\in G_x } \sum_{P \in \mathbb{Z}_{\geq 0}^n} h^m \left( g \left( S_{m,P} \right) (z) , S_{m,P} (z)  \right) \nonumber \\
& = & \sum_{g\in G_x } \Theta_{G_x }^{orb} (g)^m \left( \sum_{P \in \mathbb{Z}_{\geq 0}^n} f_{m,P} \big( g^{-1} (z) \big) \overline{f_{m,P} (z) } \right) \left( 1-|z|^2 \right)^{mn} \label{equationbergmankernelhyperbolicorbifold} \\
& = & \frac{(2 \pi)^{-n} (mn-1)!}{ (mn-n-1)!} \sum_{g\in G_x} \Theta_{G_x }^{orb} (g)^m \frac{\left( 1-|z|^2 \right)^{mn} }{ \left( 1-\left\langle g^{-1} z,z \right\rangle \right)^{mn} } \nonumber ,
\end{eqnarray}
and
\begin{eqnarray}
& & p_x^* B_{\mathbb{B}^n /G_x ,  \omega_{\mathbb{B}^n} , \mathcal{K}_{\mathbb{B}^n} /G_x , h^m_{\mathbb{B}^n} ,m } \left( p_x (z) , p_x (w) \right) \nonumber \\
& & \;\;\;\;\;\;\;\;\;\;\;\;\;\;\;\;\;\;\;\;\;\; = \frac{(2 \pi)^{-n} (mn-1)!}{ (mn-n-1)!} \sum_{g\in G_x} \Theta_{G_x }^{orb} (g)^m \frac{\left( 1-|w|^2 \right)^{mn} }{ \left( 1-\left\langle g^{-1} z,w \right\rangle \right)^{mn} } e_{ \mathcal{K}_{\mathbb{B}^n }}^{m} (z) \otimes e^{-m}_{\mathcal{K}_{\mathbb{B}^n }} (w) , \label{equationbergmankernelhyperbolicorbifold2}
\end{eqnarray}

Now we only need to show that 
\begin{eqnarray}
\left\Vert \rho_{X, \omega ,\linebundle ,m} \circ p_x (z) - \rho_{\mathbb{B}^n /G_x ,  \omega_{\mathbb{B}^n} , \big( \mathcal{K}_{\mathbb{B}^n} /G_x \big)^m , h^m_{\mathbb{B}^n} } \circ p_x (z) \right\Vert_{C^k ,B_{\frac{r}{2}} (0) }  \leq C e^{-\delta mr^2 } . \label{estimatelocalizationbergmankernelhyperbolicorbifold}
\end{eqnarray}
We can see that the proof of Theorem \ref{thmcusplocalizationprinciple} actually leads to the estimate (\ref{estimatelocalizationbergmankernelhyperbolicorbifold}), noting that (\ref{equationbergmankernelhyperbolicorbifold}) and (\ref{equationbergmankernelhyperbolicorbifold2}) implies that the integral of peak sections defined by a distribution that support is a point in $B_{\frac{r}{2}} (x)$ on $ B_{\frac{4r}{5}} (x) \sq B_{\frac{3r}{5}} (x) $ decays exponentially as $mr^2 \to\infty $.
\end{proof}

Now we are in place to prove Theorem \ref{thmsupballquotient}.

\vspace{0.2cm}

\noindent \textbf{Proof of Theorem \ref{thmsupballquotient}: }
Since $\Gamma $ is a lattice, the quotient space $ X= {\mathbb{B}}^n / \Gamma $ is a complete K\"ahler orbifold, and there exist a finite collection of disjoint open subsets $ U_i \subset X $, $i=1,\cdots , N $, and complex hyperbolic cusps $\left( V_i /\Gamma_{V_i } ,\omega_{V_i }  ,\linebundle_{V_i } /\Gamma_{V_i } ,h_{V_i } \right) $ such that $ \left( U_i ,\omega_{\mathbb{B}^n} \right) \cong \left( V_i /\Gamma_{V_i } ,\omega_{V_i } \right) $, and $ X \sq  \left( \cup_{i=1}^N U_i \right) $ is compact. When $\Gamma $ is cocompact, the orbifold $ X $ is compact, and hence $ N =0 $.

Note that the localization principle, Theorem \ref{thmcusplocalizationprinciple}, holds for complete K\"ahler orbifolds with asymptotic complex hyperbolic cusps. Then we can apply Lemma \ref{lmmbergmankernelhyperbolicSUnfinitequotient} to conclude that
\begin{eqnarray}
\sup_{X \sq  \left( \cup_{i=1}^N U_i \right) } \rho_{X , \omega_{\mathbb{B}^n} , \mathcal{K}_X , h_X ,m} & = & c_m m^{n} + O\left( m^{n-1} \right) ,\textrm{ as $m\to\infty $.} \label{inequalitythmsupballquotient1proofversion}
\end{eqnarray}
where $\left\lbrace c_i \right\rbrace_{i=1}^{\infty} $ is a sequence of positive constants depending only on $n$ and $\Gamma$, such that $c_{i} =c_{i+\mathcal{N}_1 } $ for some integer $ \mathcal{N}_1 = \mathcal{N}_1 (n,\Gamma ) >0 $, $ \forall i\in\mathbb{N} $.

Similarly, Corollary \ref{convergencecorobgmkernelsup} now implies that there exist an integer $ \mathcal{N}_2 = \mathcal{N}_2 (n,\Gamma ) >0 $ and a sequence of positive constants $\left\lbrace c'_i \right\rbrace_{i=1}^{\infty} $ depending only on $n$ and $\Gamma$, such that $c'_{i} =c'_{i+\mathcal{N}_2 } $, $ \forall i\in\mathbb{N} $, and
\begin{eqnarray}
\sup_{\cup_{i=1}^N U_i} \rho_{X , \omega_{\mathbb{B}^n} , \mathcal{K}_X , h_X ,m} & = & c'_m m^{n+\frac{1}{2}} + O\left( m^{n } \right) ,\textrm{ as $m\to\infty $.} \label{inequalitythmsupballquotient2proofversion}
\end{eqnarray}

Combining (\ref{inequalitythmsupballquotient1proofversion}) with (\ref{inequalitythmsupballquotient2proofversion}), we can obtain (\ref{inequalitythmsupballquotient1}) and (\ref{inequalitythmsupballquotient2}). When $\Gamma $ is neat, $X$ becomes a manifold, and the all group $\Gamma_i $ are trivial. It follows that $\mathcal{N} =1 $ in this case, which proves the theorem. 
\qed

\vspace{0.2cm}

\section{Conflict of interest statement}

On behalf of all authors, the corresponding author states that there is no conflict of interest.

\section{Data availability statements}

My manuscript has no associated data.


\end{document}